\documentclass[a4paper,11pt]{amsart}

\usepackage[top=1.25in, bottom=1.25in, left=1.25in, right=1.25in]{geometry}
\usepackage[all]{xy}
\usepackage{amsthm, amsmath, amssymb, amscd, latexsym, multicol, verbatim, enumerate, graphicx,xy, color, mathtools}
\usepackage{mathrsfs}
\usepackage{tikz,stackrel}
\usepackage[all]{xy}
\usetikzlibrary{calc}
\usetikzlibrary{matrix,arrows,decorations.pathmorphing}
\usetikzlibrary{intersections}
\usetikzlibrary{decorations.markings}
\usepackage{tikz-cd}
\usepackage[english]{babel}
\usepackage{stmaryrd}
\usepackage{float}
\usepackage{comment}
\usepackage[colorlinks=true, pdfstartview=FitV, linkcolor=blue, citecolor=blue, urlcolor=blue, breaklinks=true]{hyperref}
\usepackage{theoremref}
\usepackage{caption}
\usepackage{youngtab}
\usepackage{anyfontsize}
\usepackage[utf8]{inputenc}
\usepackage{bm}
\usepackage{scalerel}
\usepackage{marvosym}
\DeclareFontFamily{U}{mathx}{\hyphenchar\font45}
\DeclareFontShape{U}{mathx}{m}{n}{
      <5> <6> <7> <8> <9> <10>
      <10.95> <12> <14.4> <17.28> <20.74> <24.88>
      mathx10
      }{}
\DeclareSymbolFont{mathx}{U}{mathx}{m}{n}
\DeclareFontSubstitution{U}{mathx}{m}{n}
\DeclareMathAccent{\widecheck}{0}{mathx}{"71}
\usepackage{verbatim}

\makeatother
\theoremstyle{remark}
\numberwithin{equation}{section}

\theoremstyle{definition}
\newtheorem{theorem*}{Theorem}
\newtheorem{definition*}{Definition}
\newtheorem{theorem}{Theorem}[section]
\newtheorem{definition}[theorem]{Definition}
\newtheorem{proposition}[theorem]{Proposition}
\newtheorem{lemma}[theorem]{Lemma}
\newtheorem{corollary}[theorem]{Corollary}
\newtheorem{remark}[theorem]{Remark}
\newtheorem{notation}[theorem]{Notation}
\newtheorem{example}[theorem]{Example}

\newtheorem{quest}[theorem]{Question}

\newcommand{\tc}[2]{\textcolor{#1}{#2}}
\newcommand{\lrp}[1]{\left(#1\right)}
\newcommand{\lrb}[1]{\left[#1\right]}
\newcommand{\lrm}[1]{\left|#1\right|}
\newcommand{\lrc}[1]{\left\{#1\right\}}

\newcommand{\Q}{\mathbb{Q} }
\newcommand{\R}{\mathbb{R} }
\newcommand{\C}{\mathbb{C} }
\newcommand{\A}{\mathbb{A} }
\newcommand{\PP}{\mathbb{P} }

\newcommand{\Z}{\mathbb{Z} }

\newcommand{\XXo}{\mathbb{X}^\circ }

\newcommand{\XXdo}{\widecheck{\mathbb{X}}^\circ }

\makeatletter
\newcommand{\oversetcustom}[3][0ex]{%
  \mathrel{\mathop{#3}\limits^{
    \vbox to#1{\kern-0.5\ex@
    \hbox{$\scriptstyle#2$}\vss}}}}
\makeatother

\newcommand{\cF}{\mathcal{F}}
\newcommand{\T}{\mathbb{T}}

\newcommand{\orT}{\oversetcustom{\longrightarrow}{\mathbb{T}^n}}

\newcommand{\by}{\mathbf{y}}
\newcommand{\bp}{\mathbf{p}}
\newcommand{\gv}{\mathbf{g} }
\newcommand{\cv}{\mathbf{c} }
\newcommand{\vb}[1]{\mathbf{#1}}
\newcommand{\cA}{\mathcal{A} }
\newcommand{\cAp}{\mathcal{A}_{\mathrm{prin}} }
\newcommand{\cXp}{\mathcal{X}_{\mathrm{prin}} }
\newcommand{\cAps}[1]{\mathcal{A}_{\mathrm{prin},{#1}} }

\newcommand{\cX}{\mathcal{X} }

\newcommand{\ssO}{\mathcal{O} }

\newcommand{\lb}{\mathcal{L} }
\newcommand{\trop}{\mathrm{trop} }
\newcommand{\tf}{\vartheta }
\newcommand{\gp}{\mathrm{gp} }

\newcommand{\Xsp}{\widehat{\cX} }

\newcommand{\eq}[2]{\begin{equation}\label{#2} \begin{split} #1  \end{split} \end{equation}}
\newcommand{\eqn}[1]{\begin{equation*} \begin{split} #1 \end{split} \end{equation*}}
\newcommand{\cc}{{\Delta^+} }
\newcommand{\ccF}{{\Delta^+_F} }
\newcommand{\TV}[1]{ {\tv(#1)} }
\newcommand{\Xfam}{\mathscr{X} }

\newcommand{\Xfams}[1]{\mathscr{X}_{#1} }
\newcommand{\Xfsps}[1]{\widehat{\Xfam}_{#1} }
\newcommand{\Xt}{\widetilde{X} }

\newcommand{\Nuf}{N_{\text{uf}}}
\newcommand{\Iuf}{I_{\text{uf}}}
\newcommand{\Ifr}{I_{\text{fr}}}

\newcommand{\sk}[2]{\lrc{ #1 , #2 } }

\newcommand{\sgn}{\operatorname*{sgn}}

\newcommand{\cham}{\mathcal{G}}
\newcommand{\chams}[1]{\mathcal{G}_{#1}}
\newcommand{\chamdual}{\mathcal{C}}

\DeclareMathOperator{\Proj}{Proj}
\DeclareMathOperator{\val}{val}

\newcommand{\bt}{\mathbf{t}}

\newcommand{\bc}{\mathbf{c}}

\newcommand{\yd}[1]{ \scaleto{\Yng({#1})\mathstrut}{5pt} }

\DeclareMathOperator{\PGL}{PGL}

\DeclareMathOperator{\Hom}{Hom}

\DeclareMathOperator{\Spec}{Spec}

\DeclareMathOperator{\Grass}{Gr}

\DeclareMathOperator{\tv}{TV}
\DeclareMathOperator{\flow}{Flow}
\newcommand{\Yng}{\Yboxdim4pt \yng}
\DeclareMathOperator{\Star}{Star}

\title{Toric degenerations of cluster varieties and cluster duality}
\author[L.~Bossinger]{Lara Bossinger}
\email{lara@im.unam.mx}
\address{Instituto de Matem\'aticas UNAM Unidad Oaxaca\\
Le\'on 2, altos, Oaxaca de Ju\'arez\\
Centro Hist\'orico\\
68000 Oaxaca\\
Mexico}

\author[B.~Fr\'ias-Medina]{Bosco Fr\'ias-Medina}
\email{bosco@matmor.unam.mx}
\address{
Centro de Ciencias Matem\'aticas UNAM Campus Morelia\\
Antigua Carretera a P\'atzcuaro 8701\\
Col. Ex Hacienda San Jos\'e de la Huerta\\
58089 Michoac\'an\\ 
Mexico}

\author[T.~Magee]{Timothy Magee}
\email{timothy.magee@kcl.ac.uk}
\address{Department of Mathematics\\
Faculty of Natural \& Mathematical Sciences\\
King's College London\\
Strand\\
London WC2R 2LS\\
United Kingdom}

\author[A.~N\'ajera Ch\'avez]{Alfredo N\'ajera Ch\'avez}
\email{najera@matem.unam.mx}
\address{CONACYT-Instituto de Matem\'aticas UNAM Unidad Oaxaca\\
Le\'on 2, altos, Oaxaca de Ju\'arez\\
Centro Hist\'orico\\
68000 Oaxaca\\
Mexico}


\thanks{This work was partially supported by CONACYT grant CB2016 no. 284621. The first author was partially supported by the Max Planck Institute of Mathematics in the Sciences, Leipzig and the ``Programa de Becas Posdoctorales en la UNAM 2018'' Instituto de Matem\'aticas, UNAM. The second and third authors acknowledge the support of Fondo Institucional de Fomento Regional para el Desarrollo Cient\'ifico, Tecnol\'ogico y de Innovaci\'on, FORDECYT 265667. The second author was additionally supported by ``Programa de Becas Posdoctorales 2019'', DGAPA, UNAM during the revisions of this article. The third author was additionally partially supported by EPSRC grant EP/P021913/1 during revisions of this article.}

\begin{document}

\calclayout

\newpage

\begin{abstract}
  We introduce the notion of a $Y$-pattern with coefficients
  and its geometric counterpart: an $\cX$-cluster variety with coefficients. 
  We use these constructions to build a flat degeneration 
  of every skew-symmetrizable specially completed $\cX$-cluster variety $\Xsp$ to the toric variety associated to its {\bf g}-fan. 
  Moreover, we show that the fibers of this family are stratified in a natural way, with strata the specially completed $\cX$-varieties encoded by $\Star(\tau)$ for each cone $\tau$ of the $\gv$-fan. These strata degenerate to the associated toric strata of the central fiber.
  We further show that the family is cluster dual to $\cAp$ of \cite{GHKK}, and the fibers cluster dual to $\cA_t$.
  Finally, we give two applications.  First, we use our construction to identify the Rietsch--Williams toric degeneration of Grassmannians from \cite{RW17} with the Gross--Hacking--Keel--Kontsevich degeneration in the case of $\Grass_2(\C^5)$. 
  Next, we use it to link cluster duality to Batyrev--Borisov duality of Gorenstein toric Fanos in the context of mirror symmetry. 

\end{abstract}

\maketitle

\setcounter{tocdepth}{1}

\tableofcontents

\section{Introduction}

\subsection{Overview}
In 
\cite{GHK_birational}, Gross, Hacking and Keel began a systematic study of cluster varieties from the perspective of mirror symmetry and the minimal model program. 
Together with Kontsevich, they used this geometric approach in \cite{GHKK} to settle various important and long-standing conjectures in cluster theory. 
The insights of \cite{GHK_birational} and \cite{GHKK} reveal a strong link between cluster theory and toric geometry. 
One of the main purposes of this paper is to explore the connections between these theories. 
In particular, we pursue the idea that compactifications of $\cA$-cluster varieties are a generalization of toric varieties associated to polytopes,
while partial compactifications of $\cX$-cluster varieties generalize toric varieties associated to fans.
Recall that toric degenerations of compactified $\cA$-cluster varieties were studied in great detail in \cite{GHKK}. 
They arise from {\it{positive polytopes}}, which define a family of graded rings, and hence projective varieties.
Here we develop a theory of toric degenerations for partially compactified $\cX$-cluster varieties.
Both the partial compactifications and the toric degenerations arise from a fan construction rather than a polytope construction.
More precisely, given a special completion $\Xsp$ of $\cX$ (in the sense of \cite{FG_X}) of complex dimension $n$, we construct for each seed $s$ a flat family 
\begin{equation}
\xymatrix{
\widehat{\Xfam}_s \ar^{\pi}[d]\\ 
\ \ \mathbb{A}_{\C}^n
}
\end{equation}
defining a degeneration of $\Xsp$ with the following properties:
\begin{itemize}
\item the fiber over $(1,\dots,1)$ is $\Xsp$;
\item if $ \Xsp$ has no frozen directions then the fiber over $(0,\dots,0)$ (also called the central fiber) is the toric variety associated to the cluster complex $\Delta^+_s$;
\item if $ \Xsp$ has frozen directions then the fiber over $(0,\dots,0)$ is the toric variety associated to a fan contained in a superpotential cone. The fan structure comes from the cluster structure;
\item in both cases the cluster variety $\cX$ degenerates to the dense torus inside the toric central fiber.
\end{itemize}
One of the key ideas is to make a distinction between {\it{coefficients}} and {\it{frozen variables}}, concepts which are usually identified. 
We think of coefficients as parameters deforming the mutation formulas-- deformation parameters for the families.
Frozen variables are $\cA $-cluster variables associated to \emph{frozen directions}; 
they relate to the geometry of individual fibers rather than the deformations of these fibers.
This approach allows us to consider the toric varieties in the central fibers as a very degenerate class of cluster varieties.

To construct the family $\Xfsps{s} $ we develop the notion of a $Y$-pattern with coefficients and its geometric counterpart: $\cX$-cluster varieties with coefficients. Moreover, we develop cluster duality for cluster varieties with coefficients and prove that our degeneration is cluster dual in this generalized sense to the Gross--Hacking--Keel--Kontsevich $\cAp$-degeneration of $\cA$-cluster varieties.
By doing so, we solve simultaneously two problems in cluster theory. On the one hand we find a way to include coefficients in the $Y$-mutation formula. This was an aim of Fomin and Zelevinsky \cite{communication_w_Fomin} when they developed the foundations of cluster algebras. On the other hand, we define spaces that are cluster dual to  $\cA$-cluster varieties with coefficients, a problem posed in \cite[\S7]{GHKK} which we are about to describe. Let $\cA_{s}$ be an $\cA$-cluster variety with initial seed $s$. In \cite{GHK_birational} the authors introduced a family $ \cAps{s}\to \mathbb{A}^n_{\C} $ of deformations of $\cA_s$.
The space $\cAps{s}$ and the fibers $\cA_{\vb{t}}$ over closed points $\vb{t}\in \mathbb{A}^n_{\C}$ are themselves cluster varieties with coefficients.
In \cite[\S7]{GHKK} the authors predicted the existence of a cluster dual $\cA_{\vb{t}}^{\vee}$ to the fibers but stated it was unclear how the coefficients should be handled under dualization.  
Consider the Langlands dual seed $s^{\vee}$. Inside $ \Xfsps{s^{\vee}}$ there is an $\cX$-cluster variety with coefficients $\Xfam_{s^{\vee}}$. The restriction of $\pi$ to $\Xfam_{s^{\vee}}$ gives a map $\Xfam_{s^{\vee}}\to \A^n_{\C}$ with fiber $\cX_{\bf t}$ over ${\bf t}$. We show that the varieties $\Xfam_{s^{\vee}}$ and ${\cAps{s}}$ are cluster dual as cluster varieties with coefficients. Moreover, we show that the fibers $\cX_{\vb{t}}$ are exactly the wanted cluster duals $\cA_{\vb{t}}^{\vee}$. For more details the reader is referred to \S\ref{sec:cluster_duality}, in particular \thref{rem:dual_specified_coeff}.

This is the first of a series of papers. In a sequel, along with M.-W. Cheung, we will explore how the notion of cluster duality with coefficients is related to Batyrev--Borisov duality in the context of mirror symmetry. In \cite{BCMNC} we will show that the toric degeneration of the Grassmannian given by \cite{RW17} coincides with the toric degeneration given by the $\cAp$-construction for all $\Grass_k\lrp{\C^n}$.

\subsection{Toric degenerations}
For the past 25 years constructions of toric degenerations have been of increasing interest in algebraic geometry, with particular development in representation theory (we refer to the overview \cite{FFL16} and references therein). 
The core idea is to glean information about varieties of interest by passing through the world of toric geometry, where otherwise difficult questions become readily approachable.
If we construct a flat family whose generic fiber is isomorphic to the variety of interest while a special fiber is a toric variety, this allows us to deduce properties of the generic fiber from the special fiber (e.g. degree, Hilbert-polynomial or smoothness).
Typically for polarized projective varieties $\lrp{Y,\lb}$ this is realized using Newton--Okounkov bodies,
where a valuation on the section ring of $\lb$ determines a convex body generalizing the toric notion of a Newton polytope (see e.g. \cite{Oko98,LM09,KK12,An13}).

Cluster varieties (and their compactifications) entered the stage of toric degenerations in \cite{GHKK} and \cite{RW17}. 
In both aforementioned papers the authors construct toric degenerations using polytopes (in a broad sense) related to section rings of ample line bundles over compactified cluster varieties. While \cite{RW17} use Newton--Okounkov bodies for Grassmannians directly, the construction given in \cite{GHKK} is more general and is expected to recover toric degenerations of many representation theoretic objects.     
For example, the case of flag varieties is addressed in \cite{Mag17,BF,GKS,GKS2}.

We take a slightly different approach in this paper and construct toric degenerations associated to fans naturally occurring in cluster theory.
The theory developed in this paper can be used to understand the relation between the approach taken in \cite{GHKK} and that of \cite{RW17}.
We illustrate this in the example of $\Grass_2(\C^5)$ (see \S\ref{ex:RW}), and in a sequel \cite{BCMNC} we will treat the general case.

\subsection{The geometric setting} 
\label{intro:geometric_setting}
In the world of toric geometry, there are two main combinatorial devices encoding toric varieties-- fans and polytopes.
A fan gives a recipe for gluing affine schemes together to construct the toric variety.
A polytope gives a graded ring in terms of a basis of homogeneous elements and relations.
So, it gives a projective toric variety equipped with an ample line bundle and a vector space basis for each homogeneous component of the section ring of this line bundle.
A philosophical take-away of this paper is that $\cX$-varieties compactify with a fan construction, while $\cA$-varieties compactify with a polytope construction.

Take as a starting point the scattering diagram description of $\cA$- and $\cX$-varieties introduced in \cite{GHKK}.
The same scattering diagram (living in $\cX^{\trop}\lrp{\R}$) encodes both $\cA$ and $\cX$, but in different ways.
We outline both below.

The scattering diagram tells us how to write down a basis of global regular functions on $\cA$-- whose elements are called $\tf$-functions-- as well as the structure constants for multiplying these $\tf$-functions.
That is, the scattering diagram encodes $\ssO(\cA)$, together with a distinguished vector space basis.
This generalizes the toric notion that the cocharacter lattice $M$ of $T^\vee$ is a distinguished vector space basis for $\ssO(T)$, where $T\cong \lrp{\C^*}^n$ is an algebraic torus and $T^\vee$ is the dual torus.
To address projective varieties compactifying $\cA$, 
Gross, Hacking, Keel and Kontsevich introduce {\it{positive polytopes}} $\Xi$ in $\cX^\trop\lrp{\R}$ (see \cite[Definition~8.6]{GHKK}).
Positivity is precisely what is needed for the integer points ($\tf$-functions) of $\Xi$ to define a {\emph{graded}} ring.
The result is a polarized projective variety $\lrp{Y,\lb}$ together with a basis of $\tf$-functions for the section ring of $\lb$.
In the $\mathcal A_{\text{prin}}$-construction, they then give a flat family of deformations
of $\cA$ over $\A^n$, with central fiber simply the torus $T_{N,s} \subset \cA$ associated to the seed $s$.
The $\tf$-functions extend canonically to $\cAps{s}$, yielding a toric degeneration of $\lrp{Y,\lb}$.

On the other side of the picture, 
a subset of the scattering diagram (known as the $\vb{g}$-fan) has a simplicial fan structure.
The key idea here is to exploit this fan structure.
In the atlas for the $\cX$-variety itself, there is a torus for every maximal cone of the $\gv$-fan.
Specifically, given a maximal cone $\cham$ (a {\it{$\gv$-cone}}) with dual cone $\chamdual$ (a {\it{$\cv$-cone}}),
we have the torus 
\eqn{
T_{M,\cham}:= \Spec\lrp{ \C \lrb{ \chamdual(\Z)^\gp} }, 
}
where $\chamdual(\Z)$ is the monoid of integral points in $\chamdual$ and $\chamdual(\Z)^\gp$
is its group completion.
These tori are glued by the $\cX$-mutation formula, with the (non-zero) monoid generators of $\chamdual(\Z)$ giving rise to an $\cX$-cluster.
From the perspective of toric geometry, it is far more natural to glue the affine spaces
\eqn{
\A_{M,\cham}^n:= \Spec\lrp{ \C \lrb{ \chamdual(\Z)} } }
by $\cX$-mutation.
The resulting space is a partial compactification of $\cX$-- the {\it{special completion}} $\Xsp$ of \cite{FG_X}.
Refinements $\Sigma$ of the $\gv$-fan similarly yield partial compactifications.

Rather than deforming global sections of line bundles as for the compactification of $\cA$, we can deform the gluing of patches $\A_{M,\cham}^n$.
To define these deformations, we
introduce the notion of $\cX$-cluster varieties with coefficients, see \S\ref{sec:cluster_var_coeff}.
As an important case of this notion, we describe how to deform the gluing in a manner dual to the deformation of $\vartheta$-functions in $\mathcal{A}_{\mathrm{prin}}$:

\begin{definition}
Fix an initial seed $s_0$, with associated $\gv$-cone $\cham_0$,
and take principal coefficients $t_1, \dots, t_n$ at $s_0$.
Let $R= \C\lrb{t_1,\dots, t_n}$. 
The \emph{$\cX$-cluster variety with coefficients} in $R$ and its special completion are the $R$-schemes
\[
\Xfam_{\cham_0}:=\bigcup_{\mathcal G} T_{M,\mathcal G}(R)
\ \ \text{and} \ \ \Xfsps{\cham_0} :=\bigcup_{\mathcal G} \A^n_{M,\mathcal G}(R)
\]
with patches glued birationally via the mutation formula with coefficients \eqref{eq:mutfamily}. 
\end{definition}

Below we summarize some of our main results. 

\begin{theorem}(\thref{flat}, \thref{toric_degen}, \thref{iso of fibres}, \thref{smooth})
The family $\Xfsps{\cham_0} \to \Spec(R)$ is a flat toric degeneration of $\Xsp$ to the toric variety defined by the $\gv$-fan, realized here with $\cham_0$ as the positive orthant.
Generic fibers of this family are isomorphic and smooth.
\end{theorem}

\begin{theorem}(\thref{c-deg}, \thref{tto0})
The $\cX$-variables $X_{i;\cham}$ extend canonically to homogeneous rational functions on $\Xfsps{\cham_0}$.
The degree of the extension of $X_{i;\cham}$ is its $\cv$-vector $\cv_{i;\cham}$, and in the central fiber the extension restricts to the monomial whose exponent vector is this $\cv$-vector.
\end{theorem}

\begin{theorem}(\thref{strata}, \thref{strata2strata})
The fibers of $\Xfsps{\cham_0} \to \Spec\lrp{R}$ are stratified, with each stratum encoded by $\Star(\tau)$ for some cone $\tau$ of the $\gv$-fan.
These strata $V(\tau)_{\vb{t}}$ are again specially completed $\cX$-varieties with coefficients, and the stratum $V(\tau)_{\vb{t}}$ degenerates to the toric stratum $V(\tau)$ defined by $\Star(\tau)$ in the central fiber.
\end{theorem}

We can replace the $\gv$-fan with any refinement $\Sigma$ and obtain results analogous to each of the theorems above (see \thref{refine}).

In the toric setting, fans and polytopes are both used to compactify tori.
Here, we find fans compactifying $\cX$-varieties and polytopes compactifying $\cA$-varieties.
This is completely natural in the scattering diagram description of cluster varieties, as discussed above.  
But there are other reasons to expect different types of compactifications for the two flavors of cluster varieties.
First, $\cA$-variables are $\tf$-functions, globally regular on $\cA$, while $\cX$-variables are only locally defined.
In fact, as observed in \cite{GHK_birational}, $\cX$-varieties may not have many global functions at all.
From this perspective, $\cA$-varieties are primed for a polytope-$\Proj$ construction, while it is more natural to think of an $\cX$-variety in terms of an atlas of affine schemes.
Moreover, it is usually {\emph{hopeless}} to attempt to compactify an $\cX$-variety with a $\Proj$ construction as $\cX$-varieties are generally not separated \cite[Remark~2.6]{GHK_birational}, while $\Proj$ of a graded ring is always separated.

\begin{remark}\thlabel{rem:AprinXfam}
That said, there are many interesting cases where the $\cA$- and $\cX$-varieties are isomorphic.
Namely, if there is a unimodular $p^*$-map \cite[Equation~2.1]{GHK_birational}, then we have isomorphic $\cA$- and $\cX$-varieties.
However, the different ways of viewing this space lead to very different compactifications.
In this case, our family $\Xfams{s_0}$ of \thref{def:Xfam} is just $\cAps{s_0}$, with $p$ giving fiberwise identifications of $\cA_{\vb{t}}$ and $\cX_{\vb{t}}$.
But $\Xfsps{s_0}$ and the partial compactifications of $\cAps{s_0}$ considered in \cite{GHKK} differ.
\end{remark}

We would like to emphasize here that $\cX$-varieties, although generally non-separated, are indeed spaces of interest.
They are nice in many other ways-- they are for instance smooth log Calabi--Yau varieties equipped with a positive structure.
They entered center stage through the work of Fock and Goncharov, particularly \cite{FG_Teich}.
Here, a stacky version of the $\cX$-variety serves as the moduli stack of framed $G$-local systems on a marked surface with boundary, where $G$ is a split semisimple algebraic group with trivial center.
Furthermore, Fock and Goncharov describe how the positive real points of the $\cX$-variety define higher Teichm\"uller spaces, while tropical points define lamination spaces.
They study Teichm\"uller and lamination spaces further in \cite{FG_DualTeichLam}, where they exhibit a canonical pairing between $\cX$ (resp. $\cA$) Teichm\"uller spaces and $\cA$ (resp. $\cX$) lamination spaces.
Moreover, in \cite{FG_X} they introduce completions of Teichm\"uller spaces for decorated surfaces with marked points in the boundary. 
They prove that the set of positive points of a specially completed $\cX$-cluster variety gives a completion of the corresponding Teichm\"uller space with boundary components corresponding to simple laminations. 
In this context $\cX$-cluster varieties provide a formalism for quantizing such Teichm\"uller spaces following the approach of Chekhov and Fock \cite{FC99}.
Le studies higher Teichm\"uller and lamination spaces {\it{\`a la}} Thurston in \cite{Le} and similarly gives a compactification of higher Teichm\"uller space with projective laminations.  
Allegretti and Bridgeland relate the moduli stacks mentioned above for $G=\PGL_2(\C)$ to projective structures on the corresponding surfaces in \cite{AlBr}. On the algebraic side of the picture, cluster variables associated to finite type $\cX$-cluster varieties have been studied in great detail in \cite{S-B18}.

\subsection{The algebraic foundations}
At the heart of cluster theory there are two central notions: \emph{cluster patterns} and \emph{$Y$-patterns}. These are certain recurrences extensively studied by Fomin and Zelevinsky in \cite{FZ_clustersI,FZ_clustersII,FZ_Y,FZ_clustersIV}. 
The geometric incarnation of a cluster pattern (resp. $Y$-pattern) is an $ \cA$-cluster variety (resp. $\cX $-cluster variety). 
To define a cluster algebra with coefficients one considers a cluster pattern with coefficients in a $Y$-pattern. 
Such a cluster algebra lives in an ambient field $\cF=K(x_1,\dots,x_n)$ of rational functions. 
By construction, $Y$-patterns are \emph{clusters} of elements of the coefficient field $K$ and cluster patterns are \emph{clusters} of functions in $\cF$ generating the cluster algebra. 
The geometric considerations discussed above suggest that there should exist a notion of a \emph{$Y$-pattern with coefficients}, which we define as follows:

Every seed of such a pattern is given by a triple $(\mathbf y,\mathbf p, B)$, where $(\mathbf p,B)$ is a $Y$-seed in some semifield $\PP$ and $(\mathbf y,B)$ is a $Y$-seed in a universal semifield with coefficients in $\Q\PP$ (see \eqref{eq:univ_PP_sf}).
Then $(\mathbf y',\mathbf p',B')$ is the mutation in direction $k$ of $(\mathbf y,\mathbf p,B)$, if $(\mathbf p',B')$ is the usual $Y$-mutation in direction $k$ of $(\mathbf p,B)$ and $\mathbf y'=(y'_1,\dots,y'_n)$ is given by
\begin{eqnarray*}
    y'_j := \left\{\begin{matrix} y_j^{-1} &\text{ if } j=k, \vspace{1mm} \\
    y_j \lrp{ p_k^{\llbracket b_{kj}\rrbracket} +  p_k^{\llbracket -b_{kj}\rrbracket} y_k^{-\sgn(b_{kj})} } ^{-b_{kj}} &\text{ if } j\not =k,\end{matrix}\right.
\end{eqnarray*}
where $p_k^{\llbracket \pm  b_{kj}\rrbracket}$ are certain expressions in $\PP$ (see \thref{p_+-} below, resp. \S\ref{y_pat_ceof} for more details).
As expected such a pattern lives in an ambient field of rational functions. 
Indeed, $x$-cluster variables are functions on an $\cA$-cluster variety, while the $y$-variables are local functions on an $\cX$-cluster variety\footnote{Due to the Laurent phenomenon, Fomin and Zelevinsky's $x$-cluster variables are global functions. 
There is not a general Laurent phenomenon for Fomin and Zelevinsky's $y$-variables; therefore, they are not global functions. However, they are always regular functions in at least one torus of the cluster atlas.}. 

We prove similar to the case of cluster patterns with coefficients, that $Y$-patterns with coefficients satisfy \emph{separation formulas} in \thref{separation_formulas}. Further, in \thref{periodicity} we show that the \emph{periodicities} of $Y$-patterns with coefficients agree with those of the coefficient-free case.

\subsection{An example} \label{sec:ex_A2}
We illustrate our theory in the smallest non-trivial example. Let $\Sigma$ be the complete fan in $\R^2 $ depicted in Figure~\ref{fig:A2_family}. It is defined by the rays spanned by $(1,0), (0,1), (-1,1), (-1,0)$ and $(0,-1)$. This fan is the ${\bf g}$-fan of an $\cX$-cluster variety of type $A_2$. 
There are two geometric objects naturally associated to the fan $\Sigma$:
\begin{enumerate}
    \item the toric variety $\text{TV}(\Sigma)$;
    \item the special completion of the corresponding $\cX$-cluster variety\footnote{The usual $\cX$-variety is also naturally associated to $\Sigma$ by gluing complex tori instead of affine spaces.}.
\end{enumerate}

\begin{figure}[!htbp]
    \centering
\begin{tikzpicture}
\draw[thick, teal] (0,-2.5) -- (0,3.5);
\draw[thick, teal] (-4,0) -- (3,0);
\draw[thick, teal] (0,0) -- (-3.25,3.25);

\node at (.25,.25) {\tiny \bf$\textcolor{teal}{\sigma_0}$};
\node at (.25,-.25) {\tiny \bf$\textcolor{teal}{\sigma_1}$};
\node at (-.25,-.25) {\tiny \bf$\textcolor{teal}{\sigma_2}$};
\node at (-.5,.25) {\tiny \bf$\textcolor{teal}{\sigma_3}$};
\node at (-.25,.5) {\tiny \bf$\textcolor{teal}{\sigma_4}$};

\node at (2.25,2.75) {\tiny $X_2$}; 
\node at (1,1.75) {\tiny $X_1$};

\draw[<->,dashed] (2.25,2.5) -- (2.25,-1.75);
\draw[<->,dashed] (1,1.5) -- (1,-.75);

\node at (2.25,-2) {\tiny $X_2^{-1}$}; 
\node at (1.25,-1) {\tiny $X_1(\textcolor{magenta}{t_2}X_2+1)$};

\draw[<->,dashed] (1.75,-2) -- (-1.75,-2) ;
\draw[<->,dashed] (.25,-1) -- (-.5,-1);

\node at (-3,-2) {\tiny $\frac{\textcolor{magenta}{t_1}X_1(1+\textcolor{magenta}{t_2}X_2)+1}{X_2}$}; 
\node at (-1.5,-1) {\tiny $\frac{1}{X_1(\textcolor{magenta}{t_2}X_2+1)}$};

\draw[<->,dashed] (-3,-1.75) -- (-3,1.25);
\draw[<->,dashed] (-1.75,.5) -- (-1.75,-.75);

\node at (-1.75,.75) {\tiny $\frac{\textcolor{magenta}{t_1}X_1+1}{X_1X_2}$};
\node at (-3.25,1.675) {\tiny $\frac{X_2}{\textcolor{magenta}{t_1}X_1(1+\textcolor{magenta}{t_2}X_2)+1}$};

\draw[<->,dashed] (-2.125,1.75) -- (-.875,1.75);
\draw[<->,dashed] (-1.75,1.125) -- (-1.75,2.5);

\node at (-.5,1.75) {\tiny $X_1^{-1}$};
\node at (-1.75,2.75) {\tiny $\frac{X_1X_2}{\textcolor{magenta}{t_1}X_1+1}$};

\draw[<->,dashed] (-.25,1.75) -- (.75,1.75);
\draw[<->,dashed] (-1.25,2.75) -- (1.75,2.75);
\end{tikzpicture}
    \caption{Visualization of $\widehat{\Xfam}_s$ defined by a $Y$-pattern with principal coefficients in type $A_2$. For more details see Table~\ref{tab:A2_principal coeff}.}
    \label{fig:A2_family}
\end{figure}
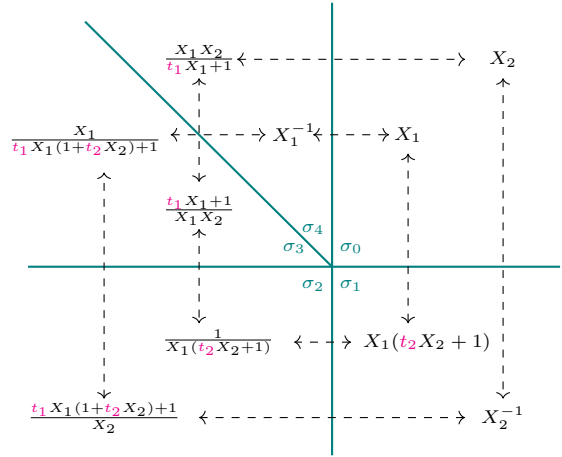

Both these varieties are obtained by gluing affine spaces isomorphic to $\A^2_{\C}$ via birational transformations. 
The cluster gluing can be deformed to obtain the toric gluing. 
To see this, we describe the family $\pi:\widehat{\Xfam}_s\to \A^2_{\C}$ (see \thref{def:Xfam}) explicitly and show how it simultaneously embodies both of the aforementioned geometric objects.
The family $\widehat{\Xfam}_s$ is an $R$-scheme for $R=\C[t_1,t_2]$ glued from five affine spaces $U_{\sigma_0}, \dots, U_{\sigma_4}$ each one isomorphic to $\A^2_R$. These affine spaces are in one-to-one correspondence with maximal cones (or $\mathbf g$-cones) $\sigma\in\Sigma$.
Figure~\ref{fig:A2_family} captures the local coordinates of the affine pieces defining $ \widehat{\Xfam_s}$ (in the algebraic language we will refer to this data as the $Y$-seeds of the $Y$-pattern with principal coefficients).
They are depicted in their associated $\mathbf g$-cones.
Explicitly, each affine piece $U_{\sigma_i}$ is the spectrum of the polynomial ring generated by the $Y$-cluster variables of the underlying $Y$-seed.
We have,
\[
U_{\sigma_i}:=\Spec\lrp{ R[X_1^i,X_2^i]},
\]
where we set $X_j:=X_j^0$.
The dashed arrows in Figure~\ref{fig:A2_family} indicate the expressions of the variables $X_j^i$ in terms of the initial variables $X_1,X_2$.
The pullbacks of the gluing morphisms are obtained from these.
For example, $U_{\sigma_0}\dashrightarrow U_{\sigma_1}$ is given by 
\begin{eqnarray*}
R[X_1^1,X_2^1] \to R[X_1,X_2], \ \text{ where } \
X_2^1 \mapsto X_2^{-1}\ \text{ and } \
X_1^1 \mapsto X_1(t_2X_2+1).
\end{eqnarray*}
From this family we can recover the 
toric and the cluster gluing associated with the fan $\Sigma$ from above as fibers by specifying values for the coefficients $(t_1,t_2)$:
\begin{enumerate}
    \item The special completion $\Xsp_s$ is the $(1,1)$-fiber of the family.
In particular, it is a $\C$-scheme glued from affine pieces $U_{\sigma_j}|_{t_i=1}$.
It can be read from Figure~\ref{fig:A2_family} by specifying $t_1=t_2=1$.
For example, the affine piece corresponding to $\sigma_1$ is $U_{\sigma_1}|_{t_i=1}=\Spec \lrp{\C[X_1^1,X_2^1]}$ with gluing $U_{\sigma_0}|_{t_i=1}\dashrightarrow U_{\sigma_1}|_{t_i=1}$ given by
\begin{eqnarray*}
\phantom{XXX} \C[X_1^1,X_2^1] \to \C[X_1,X_2], \ \text{ where } \
X_2^1  \mapsto X_2^{-1} \ \text{ and } \
X_1^1 \mapsto X_1(X_2+1).
\end{eqnarray*}
It can be verified that we obtain the usual cluster transformations for the $\cX$-cluster variety this way.
    \item The toric variety $\text{TV}(\Sigma)$ is the $(0,0)$-fiber of the family.
Hence, it is the $\C$-scheme glued from affine pieces $U_{\sigma_j}|_{t_i=0}$ and is encoded in Figure~\ref{fig:A2_family} by specifying $t_1=t_2=0$.
The affine piece $U_{\sigma_1}|_{t_i=0}$ for example, is $\Spec\lrp{ \C[X_1^1,X_2^1]}$ with gluing $U_{\sigma_0}|_{t_i=0}\dashrightarrow U_{\sigma_1}|_{t_i=0}$ given by
\begin{eqnarray*}
\C[X_1^1,X_2^1] \to \C[X_1,X_2], \ \text{ where } \ 
X_2^1 \mapsto X_2^{-1} \ \text{ and } \
X_1^1 \mapsto X_1.
\end{eqnarray*}
Note that all variables degenerate to monomials.
Further, the gluing morphisms indicated by the dashed arrows in Figure~\ref{fig:A2_family} degenerate to the classical toric gluing as described e.g. in \cite[\S3.1]{CLS}. In this case the toric variety $\text{TV}(\Sigma)$ is the blow up of $\PP^1\times \PP^1$ in a point.
\end{enumerate}

The toric variety associated to the polytope $P$ (see Figure~\ref{fig:A2_gfan_poly}) whose face-fan is $\Sigma$ is the central fiber of the corresponding $\cA_{\text{prin}}$-family. Both $\text{TV}(\Sigma)$ and $\text{TV}(P)$ are toric Fano varieties and they are in fact Batyrev-dual. 
We elaborate on this example in \S\ref{ex:dP5}, and we will explore the connection between cluster duality and Batyrev duality in general in a sequel with M.-W. Cheung.

A natural question to ask is which kind of fans can arise in cluster theory. This problem is closely related to the representation theory of quivers.  
For instance, in \cite{DWZ10,FK} the reader can find a representation-theoretic approach to {\bf g}-vectors (which are the primitive integer vectors spanning the {\bf g}-cones). 
For a representation theoretic approach to {\bf c}-vectors (the primitive integer vectors spanning the cones dual to the {\bf g}-cones) the reader can consult 
\cite{Fu17,NC_acyclic,NC_finite,NS}. 
It is therefore desirable to understand the relationship between toric varieties and quiver representations. 
Another combinatorial approach to study these fans along with their richer structure of a scattering diagrams is given in \cite{Reading,Reading18}. Moreover, the reader might also want to consult \cite{Bri17} for an approach to scattering diagrams closer to representation theory.

We would like to stress that  $\cX$-mutation formulas with principal coefficients can be obtained using the approach of \cite{LR} where the authors construct symplectic groupoids integrating log-canonical Poisson structures on $\cX$-cluster varieties and their special completions. In their setting the deformation of the toric gluing is obtained using Hamiltonian flows of the groupoid charts. It is a very interesting problem to explore the relation of our approach and that of \cite{LR}.

\subsection{Structure of the paper} 
We try to make this paper as self contained as possible and seek to overcome the discrepancy between the algebraic and the geometric notation in cluster theory by writing this paper in a bilingual fashion. 
Therefore, we survey some of the algebraic and geometric foundations of the theory we will use.

In \S\ref{sec:Y-pattern} we develop the algebraic foundations of the paper. 
Namely, we introduce the notion of a $Y$-pattern with coefficients in complete generality. 
We study the periodicities of these patterns and obtain their separation formulas. 
For convenience of the reader we also recall some basic facts of cluster theory as developed by Fomin and Zelevinsky in \cite{FZ_clustersI,FZ_clustersIV}.

Sections \S\ref{sec:cluster_duality} to \S\ref{sec:examples} are of geometric flavour. We start \S\ref{sec:cluster_duality} by recalling some basic facts of cluster varieties as presented in \cite{GHK_birational} (see \S\ref{sec:cluster_var}). 
In \S\ref{sec:cluster_var_coeff} we introduce a new notion of cluster varieties with coefficients, and we study cluster duality in this setting in \S\ref{sec:cluster duality coeff}.
This is followed by a detailed discussion of the case of principal coefficients in \S\ref{sec:cluster_var_prin_coeff}.
All throughout section \S\ref{sec:cluster_duality} we translate some of the results in  \S\ref{sec:Y-pattern} from an algebraic language to a geometric one. Readers with more interest in the geometry might consider reading this paper from \S\ref{sec:cluster_duality} in a first read. However, the general algebraic features of $Y$-patterns with coefficients analyzed in Sections \S\ref{y_pat_ceof} to \S\ref{sec:relation_cluster_patterns} are of great importance and might also be interesting to geometers.

The succeeding \S\ref{sec:tropical duality} is dedicated to the tropical geometry of cluster varieties and the special completion $\Xsp$.
We recall necessary notions from \cite{FG09} in \S\ref{sec:tropicalization} and from \cite{NZ} in \S\ref{sec:trop_duality}.

In \S\ref{sec:toric_degenerations} we construct toric degenerations of $\Xsp$.
We review the construction of toric varieties via fans in \S\ref{sec:fan_toric}. 
In \S\ref{sec:toric_degen_x} we introduce the degeneration, and show that it is indeed a flat family over $\mathbb{A}^n_{\C}$ with generic fibers being isomorphic. 

Moreover, in \S\ref{sec:strata} we study the stratification of $\Xsp$ and show that the cluster strata of $\Xsp$ degenerate to the toric strata of the central fiber.

In \S\ref{sec:examples} we present detailed examples and applications of the theory developed in this paper. 
In particular, we treat the example of the Grassmannian $\Grass_2(\C^5)$ and of the del Pezzo surface of degree five.

\medskip

\subsection*{Acknowledgements}
We are deeply grateful to Bernhard Keller for pointing us to the new result \cite[Theorem~6.2]{CaoLi} of Peigen Cao and Fang Li, which is essential for our \thref{strata}. 
We would further like to thank Man-Wai Cheung, Chris Fraser, Sean Keel, Konstanze Rietsch, Harold Williams, Lauren Williams and Tony Yue Yu for enlightening discussions during the preparation of this paper. We are grateful to Melissa Sherman-Bennett for pointing out misprints in an earlier version of this paper. 
The third and fourth authors additionally extend their gratitude to Reyna S\'anchez for hospitality and support during the preparation of this article.

\section{\texorpdfstring{$Y$-patterns with coefficients}{Y-patterns with coefficients}}\label{sec:Y-pattern}

In this section we define $Y$-patterns with coefficients and study their basic properties. This concept will be used throughout the paper.

The section is structured as follows: in \S\ref{sec:semifields} we summarize preliminaries on semifields and in \S\ref{sec:coeff_free_cluster} we recall coefficient-free cluster patterns and $Y$-patterns as introduced by Fomin and Zelevinsky.
In \S\ref{sec:coeff_cluster} we remind the reader about cluster patterns with coefficients (as in \cite{FZ_clustersIV}). In particular, we recall the definition of {\bf g}-vectors and {$F$-polynomials}. In \S\ref{y_pat_ceof} we define $Y$-patterns with coefficients.
Finally in \S\ref{sec:separation formulas} we treat periodicities and separation formulas of the latter.


\subsection{Semifields}\label{sec:semifields}
We summarize below the necessary background on semifields and introduce relevant notions that are being used throughout the rest of the paper. 
For a natural number $n$ let
\[
[1,n]:=\lbrace 1, \dots , n \rbrace.
\]
For $x$ any real number, set
\begin{eqnarray*}
\sgn(x):= \left\{\begin{matrix}-1 &\text{ if } x<0,\\
0 &\text{ if } x=0,\\
1 &\text{ if } x>0.\end{matrix}\right.
\end{eqnarray*}
Further, let $[x]_+:=\max(x,0)$, and extend to 
$[\lrp{x_1,\dots,x_n}]_+:=\lrp{\lrb{x_1}_+,\dots,\lrb{x_n}_+}$.

\begin{definition}
A \emph{semifield} is a quadruple $(\PP,\oplus,\cdot,1)$ that satisfies the axioms of a field, with the exception that there might not exist a neutral element nor inverses for the operation $\oplus$, called the auxiliary addition. 
\end{definition}

\begin{example}[Tropical semifields]\thlabel{exp:trop sf}
Let $I$ be a set. The free abelian group (written multiplicatively) with generating set $\lbrace p_i:i\in I\rbrace$ can be endowed with the structure of a semifield by setting
\begin{equation*}
    \left( \prod_{i\in I}p_i^{a_i}\right)\oplus \left(\prod_{i\in I}p_i^{b_i}\right):=\prod_{i\in I}p_i^{\min(a_i,b_i)}.
\end{equation*}
We denote the resulting semifield by $ \text{Trop}(p_i:i\in I) $ and call it the \emph{tropical semifield generated by} $\lbrace p_i:i\in I\rbrace$.
\end{example}

\begin{example}[Universal semifields]\thlabel{exp:univ_sf}
Let $S$ be a set and $\Q (S)$ be the field of rational functions on commuting variables $s\in S$ with coefficients in $\Q$. 
The \emph{universal semifield} $\Q_{\text{sf}}(S)$ is the subset of $\Q(S)$ formed by the elements that can be expressed as a ratio of two subtraction-free polynomials, i.e. polynomials in $\Z_{>0}[S]$.
Hence, $\Q_{\text{sf}}(S)$ is a semifield with respect to the usual operations of multiplication and addition. 
\end{example} 

Universal semifields satisfy the following universal property that will be used frequently. When $S$ is finite this property is precisely \cite[Lemma~2.1.6]{BFZ96}. The proof of the infinite case is a straightforward generalization of {\it{loc. cit.}}.

\begin{lemma}[Universal property of universal semifields]
\thlabel{universal_property}
Let $\PP $ be a semifield, $S$ a set and $\varphi: S\to\PP$ a function of sets. 
Then there exists a unique morphism of semifields $\Q_{\text{sf}}(S)\to \PP $, such that $s \mapsto \varphi(s)$ for all $s\in S$.
\end{lemma}

\begin{example}
The expression $x^2-x+1\in \Q_{\text{sf}}(x)$. Indeed, $x^2-x+1$ can be expressed as the ratio $\frac{x^3+1}{x+1} $. Notice that by the universal property in \thref{universal_property} this expression as a rational function is unique.
\end{example}

It was shown in \cite[p.5]{FZ_clustersIV} that the group ring $\Z \PP$  associated to any semifield $\PP $ is a domain. We let $\Q \PP  $ be the field of fractions of $\Z \PP$. Notice that any field of rational functions of the form $\Q \PP(u_1,\dots, u_\ell)$ contains a semifield $\Q \PP_{\text{sf}}(u_1,\dots, u_\ell)$ of subtraction-free rational functions. Explicitly, an element $h\in\Q\PP_{\text{sf}}(u_1,\dots,u_\ell)$ is a rational function $\frac{f}{g}$ where $f$ and $g$ are polynomials in $u_1,\dots,u_\ell$ with coefficients in $\Z_{>0}[\PP]\subset \Z \PP$.
In particular, 
\begin{eqnarray}\label{eq:univ_PP_sf}
\Q\PP_{\text{sf}}(u_1,\dots,u_\ell)=\Q_{>0}[\PP](u_1,\dots,u_\ell)\subset \Q_{\text{sf}}(\PP\cup\{u_1,\dots,u_\ell\}).
\end{eqnarray}
Therefore, $\Q\PP_{\text{sf}}(u_1,\dots,u_\ell)$ satisfies the following universal property.

\begin{corollary}[(Universal property of $\Q\PP_{\text{sf}}(u_1,\dots,u_\ell)$)]\thlabel{universal_property QPsf}
Let $\PP'$ and $\PP$ be semifields, $p_1,\dots,p_\ell\in \PP'$ and $\PP\to \PP'$ a function of the underlying sets. 
Then there exists a unique morphism of semifields $\Q\PP_{\text{sf}}(u_1,\dots,u_\ell)\to \PP'$ such that $u_i \mapsto p_i$ for all $i\in[1,\ell]$.
\end{corollary}

\begin{proof}
Let $\PP\cup\{u_1,\dots,u_\ell\}\to \PP'$ be the extension of $\PP \to \PP' $ given by sending  $u_i$ to $ p_i$ for $i\in[1,\ell]$.
The claim then follows by restricting the unique morphism of semifields 
\[
\Q_{\text{sf}}(\PP\cup\{u_1,\dots,u_\ell\})\to \PP'
\]
of \thref{universal_property} to $\Q\PP_{\text{sf}}(u_1,\dots,u_\ell)$ exploiting \eqref{eq:univ_PP_sf}.
\end{proof}

\begin{notation}
\thlabel{universal_ev}
Given elements $p_1,\dots , p_\ell$ of a semifield $ \PP' $, a function of the underlying sets $\PP\to \PP'$ and $f\in \Q \PP_{\text{sf}}(u_1,\dots,u_\ell)$ we denote by 
\begin{equation*}
f^{\PP}(p_1,\dots,p_\ell)    
\end{equation*}
the image of $f$ under the morphism $\Q\PP_{\text{sf}}(u_1,\dots,u_\ell)\to \PP' $ given by the universal property of $\Q\PP_{\text{sf}}(u_1,\dots,u_\ell)$. 
\end{notation}

\begin{definition}
\thlabel{p_+-}
Let $p$ be an element of $\PP $. We define $p^+$ and $p^-$ as 
\begin{eqnarray}\label{eq:p+-}
    p^+:=\dfrac{p}{p\oplus 1} \phantom{xxx}\text{ and  } \phantom{xxx}p^-:=\dfrac{1}{p\oplus 1}.
\end{eqnarray}
\end{definition}

Notice that the expression $p^+$ (resp. $p^-$) is in general very different from the expression $p^{+1}$ (resp. $p^{-1}$). We will introduce a mutation formula that involves both kinds of expressions. Therefore, to avoid confusion we use the notation for $p\in\PP$ and $x\in \mathbb R$:
\begin{eqnarray*}
p^{\llbracket x \rrbracket }:= \left\{\begin{matrix}p^{-} &\text{ if } x<0,\\
1 &\text{ if } x=0,\\
p^+ &\text{ if } x>0.\end{matrix}\right.
\end{eqnarray*}

\begin{example}
The notation introduced in \thref{p_+-} is particularly transparent for tropical semifields. For example, let $p=p_1^2p_2^{-1}\in \text{Trop}(p_1,p_2)$. Then we have $p^{{\llbracket x \rrbracket }}=p^+=p_1^2$ for $x\in \mathbb R_{>0}$, $p^{{\llbracket 0 \rrbracket }}=1$, and $p^{{\llbracket x \rrbracket }}=p^-=p_2$ for $x\in \mathbb R_{<0}$.
\end{example}


\subsection{Reminder on cluster and $Y$-patterns}\label{sec:coeff_free_cluster}

In this section we recall the notion of (coefficient-free) cluster patterns, originally called \emph{exchange patterns} in \cite{FZ_clustersI}.
At the same time we recall the notion of (coefficient-free) $Y$-patterns as introduced in \cite{FZ_clustersIV}. 

Throughout \S\ref{sec:coeff_free_cluster} and \S\ref{sec:coeff_cluster} we fix a semifield $(\PP,\oplus, \cdot, 1)$ and a positive integer $n$.
Further, for \S\ref{sec:coeff_free_cluster} we fix the \emph{ambient field} $\mathcal F$, which is isomorphic to rational functions in $n$ independent variables with coefficients in $\Q$.
In \S\ref{sec:coeff_cluster} we generalize $\mathcal{F}$ to have coefficients in $\Q\PP$.

\begin{definition}\thlabel{def:seed}
A (skew-symmetrizable) \emph{labeled seed} (resp. \emph{labeled $Y$-seed}) is a pair $(\mathbf x, B )$ (resp. $(\mathbf y, B )$), where
\begin{itemize}
    \item $B=(b_{ij})$ is a skew-symmetrizable $n\times n$-matrix, called an \emph{exchange matrix},
    \item $\mathbf x=(x_1,\dots,x_n)$  is an $n$-tuple of elements in the ambient field $\mathcal F$ forming a free generating set, i.e. $\mathcal{F}\cong \Q(x_1,\dots,x_n)$,
    \item resp. $\mathbf y=(y_1,\dots,y_n)$ is a $n$-tuple of elements of $\PP$.
\end{itemize}
The tuple $\mathbf x$ is called a \emph{cluster} and its components $x_1,\dots,x_n$ \emph{cluster variables}. The components $y_1,\dots,y_n$ of $\mathbf y$ are referred to as \emph{$Y$-variables}.
\end{definition}

Notice that a skew-symmetric (resp. skew-symmetrizable) matrix determines a quiver (resp. valued quiver). Therefore, we will often talk about quivers instead of matrices.

\begin{definition}
\thlabel{def mutation no coeff}
Let $(\mathbf x,B)$ be a labeled seed (resp. $(\mathbf y,B)$ a labeled $Y$-seed) and $k\in[1,n]$. The \emph{mutation in direction $k$ of} $(\mathbf x,B)$ (resp. $(\mathbf y,B)$) is the labeled seed $(\mathbf x',B')$ (resp. labeled $Y$-seed $(\mathbf y',B')$), where

\begin{itemize}
    \item $B'=(b'_{ij})$ is the usual mutation of $B$ in direction $k$, that is
        \begin{eqnarray}\label{eq:matrix mut}
        b'_{ij}:= \left\{\begin{matrix} -b_{ij} &\text{ if } i=k\ \text{or}\ j=k,\\
        b_{ij} +\sgn\lrp{b_{ik}}\lrb{b_{ik}b_{kj}}_+ &\text{ else;}\end{matrix}\right.
        \end{eqnarray}
    \item $\mathbf x'=(x_1',\dots,x_n')$, where $x'_j:=x_j$ for $j\not=k$ and 
        \begin{eqnarray}
        x_k' := \frac{\prod x_i^{[b_{ik}]_+} + \prod x_i^{[-b_{ik}]_+}}{x_k}\in \mathcal F;
        \end{eqnarray}    
    \item resp. $\mathbf y'=(y'_1,\dots, y'_n)$, where 
        \begin{eqnarray}\label{eq:Y-pattern_mutation}
        y'_j := \left\{\begin{matrix} y_j^{-1} &\text{ if } j=k,\\
        y_j(1 \oplus  y_k^{-\sgn(b_{kj})})^{-b_{kj}} &\text{ if } j\not =k.\end{matrix}\right.
        \end{eqnarray}
\end{itemize}
\end{definition}

\begin{lemma}
Mutation of labeled seeds (resp. $Y$-seeds) is involutive. That is, $(\mathbf x', B')$ is the mutation in direction $k$ of $(\mathbf x, B)$ if and only if $(\mathbf x, B)$ is the mutation in direction $k$ of $(\mathbf x', B')$ (and equivalently for $Y$-seeds).
\end{lemma}

To construct cluster patterns and $Y$-patterns of rank $n$ one has to consider the $n$-regular tree $\T^n$.
The edges of $\T^n$ are labeled by the numbers $1,\dots,n$ in such a way that the $n$ edges incident to the same vertex receive different labels. 
We write $v\in \T^n$ to express that $v$ is a vertex of $\T^n$. Moreover, we assume that $\T^n$ is endowed with a distinguished vertex $v_0$ called the \emph{initial vertex}.

\begin{definition}
\thlabel{cluster pattern}
A \emph{cluster pattern} (resp. $Y$-\emph{pattern}) \emph{of rank} $n$ is an assignment
${v\mapsto (\mathbf x_v, B_v)}$ (resp. $v\mapsto (\mathbf y_v,B_v)$) of a labeled seed (resp. labeled $Y$-seed) to each $v \in \T^n$ such that 
$(\mathbf x_{v'},B_{v'})$  (resp. $(\mathbf y_{v'},B_{v'})$) is the mutation in direction $k$ of $(\mathbf x_v,B_v)$ (resp. $(\mathbf y_v,B_v)$) whenever $v, v'\in \T^n$ are vertices joined by an edge labeled by $k$. 
The elements of a cluster pattern (resp. $Y$-pattern) are written as follows
\begin{equation*}
    \begin{array}{lcr}
    B_v=(b_{ij}^v), \phantom{xxx}
    \mathbf x_v=(x_{1;v},\dots,x_{n;v}),\phantom{xx} 
    \text{resp. }\phantom{xx} \mathbf y_v=(y_{1;v} ,\dots,y_{n;v}).
\end{array}
\end{equation*}
\end{definition}

\begin{remark}
Notice that a cluster (and similarly a $Y$-pattern) is completely determined by its value at a single vertex. 
In particular, considering an initial assignment $v_0 \mapsto ((y_1,\dots , y_n), B_{v_0}) $ of a $Y$-seed to the initial vertex $v_0$, the rest of the $Y$-pattern is obtained by iterating the mutation formula (\ref{eq:Y-pattern_mutation}).
We will later make use of this observation.
\end{remark}


\subsection{Reminders on cluster patterns with coefficients}\label{sec:coeff_cluster}

In this section we recall Fomin and Zelevinsky's extension of the notion of cluster pattern to cluster pattern with coefficients as introduced in \cite{FZ_clustersIV}. We want to treat cluster patterns and $Y$-patterns on equal footing in the sense that both kinds of patterns can be endowed with coefficients. Therefore, we will make very small changes to the original nomenclature used by Fomin and Zelevinsky.

\begin{definition}\thlabel{def:seed_coeff}
A (skew-symmetrizable) \emph{labeled seed with coefficients}\footnote{Observe that this is precisely \cite[Defintion 2.3]{FZ_clustersIV} where $(\mathbf x, \bp,B)$ is simply referred to as a \emph{labeled seed}. We slightly modify the nomenclature in order to distinguish from the coefficient-free case introduced in \thref{def:seed}.} is a tuple $(\mathbf x, \bp,B)$:
\begin{itemize}
    \item $(\bp,B)$ is a labeled $Y$-seed in $\PP$,
    \item $\mathbf x=(x_1,\dots,x_n)$ is an $n$-tuple of elements in the ambient field $\mathcal F$ forming a free generating set, i.e. ${\mathcal{F}\cong \Q\PP(x_1,\dots,x_n)}$.
\end{itemize}
\end{definition}

\begin{remark}\thlabel{rmk:y vs p}
In subsection \S\ref{y_pat_ceof} we will combine two $Y$-patterns $v\mapsto (\mathbf y_v,B_v)$ and $v\mapsto (\mathbf p_v,B_v)$ to obtain the notion of a $Y$-pattern with coefficients. The two $Y$-patterns play different (non-symmetric) roles in \thref{Y-pattern with coeff}: one is the coefficient $Y$-pattern, for which we reserve the notation $(\bp, B)$, while the other one is the $Y$-pattern that will be endowed with coefficients, denoted by $(\by,B)$.
This explains the use of $(\bp,B)$ in \thref{def:seed_coeff} opposed to $({\bf y},B)$.
\end{remark}

\begin{definition}\thlabel{def:cluster_patt_coeff}
Let $(\mathbf x,\bp,B)$ be a labeled seed with coefficients and take $k\in[1,n]$ with ${\bp=(p_1,\dots, p_n)}$. The \emph{mutation in direction $k$ of} $(\mathbf x,\bp,B)$  is the labeled seed with coefficients $(\mathbf x',\bp',B')$, where $(\bp',B')$ is the mutation in direction $k$ of the $Y$-seed $(\bp,B)$ and $\mathbf x'=(x_1',\dots,x_n')$ with $x'_j:=x_j$ for $j\not=k$ and
        \begin{eqnarray}
        x_k':=\frac{\displaystyle p_k^-
\prod_{i:b_{ik}<0} x^{ -b_{ik} }_{i}  +p_k^+
\prod_{i:b_{ik}>0} x^{ b_{ik} }_{i}}{x_k}\in \mathcal F.
        \end{eqnarray}
\end{definition}

\begin{definition}
\thlabel{cluster pattern coeff}
A \emph{cluster pattern with coefficients of rank $n$} is an assignment of a labeled seed with coefficients ${(\mathbf x_v,\bp_v, B_v)}$ to each  ${v \in \T^n}$ such that $(\mathbf x_{v'},\bp_{v'},B_{v'})$ is the mutation in direction $k$ of $(\mathbf x_v,\bp_v,B_v)$ whenever $v,  v'\in \T^n$ are vertices joined by an edge labeled by $k$. The elements of a cluster pattern with coefficients are written as
\begin{equation*}
    \begin{array}{lcr}
    \mathbf x_v=(x_{1;v},\dots,x_{n;v}), \phantom{xxx}
    \bp_v=(p_{1;v} ,\dots,p_{n;v}) \phantom{xxx} & \text{and} & \phantom{xxx} B_v=(b_{ij}^v).
\end{array}
\end{equation*}
\end{definition} 

The following theorem is \cite[Theorem 3.1]{FZ_clustersI}.

\begin{theorem}[Laurent phenomenon]
\thlabel{Laurent_phenomenon}
Let $v\mapsto ( \mathbf{x}_v,\bp_v, B_v)$ be a cluster pattern with coefficients and $\mathbf x_{v_0}=(x_1,\dots, x_n)$ the initial cluster. Then for $j\in [1,n]$ and $v\in \T^n$:
\[ 
x_{j;v}\in \Z\PP[x_1^{\pm 1},\dots, x_n^{\pm 1}].
\] 
\end{theorem}

\begin{definition}
A cluster pattern with \emph{principal coefficients} at vertex $v_0$ 
refers to a cluster pattern with coefficients $v\mapsto (\mathbf x_v,\mathbf p_v,B_v)$ with $\PP=\text{Trop}(p_{1},\dots,p_{n})$ as underlying semifield, where $\bp_{v_0}= (p_{1},\dots,p_n)$.
\end{definition}

\begin{definition}
The \emph{Laurent polynomials} 
\[
X^{\text{FZ}}_{j;v}\in\Z[x_1^{\pm 1},\dots,x_n^{\pm 1};p_1,\dots,p_n]
\]
are the cluster variables of a cluster pattern with principal coefficients (by \cite[Proposition~3.6]{FZ_clustersIV}). Further, we have the \emph{$F$-polynomials} $F_{j;v}$ defined as
\begin{eqnarray}\label{F-polynomial}
F_{j;v}(p_1,\dots,p_n) := X^{\text{FZ}}_{j;v}(1,\dots,1;p_1,\dots,p_n) \in \Z[p_1,\dots,p_n].
\end{eqnarray} 
\end{definition}

\begin{remark}
The superscript FZ is used to avoid confusion: in the next sections we adopt the standard geometric notation in which cluster variables are denoted by $A$ and $y$-variables by $X$. In \S\ref{sec:cluster_duality} we provide an explicit dictionary between the algebraic and geometric framework. See \thref{rem:notation}, \thref{dictionary}, \thref{dictionary_2} and \thref{tilde_extension}.  
\end{remark}

\begin{definition}
\thlabel{g-vectors}
(\cite[Proposition 6.1]{FZ_clustersIV})
Every Laurent polynomial $X_{j;v}^{\text{FZ}}$ is homogeneous with respect to the $\Z^n$-grading in $\Z[x_1^{\pm 1}, \dots , x_n^{\pm 1};p_1,\dots , p_n]$ given by deg$(x_i)={\bf e}_i$ and deg$(p_j)=-{\bf b}_j$, where ${\bf e}_1,\dots , {\bf e}_n$ is the standard basis of $\Z^n$ and ${\bf b}_j$ is the $j^\text{th}$ column of $B_{v_0}$. The degree of $X^{\text{FZ}}_{j;v}$ is called the ${\bf g}$-\emph{vector} associated to $j$ and $v$ denoted
\[
{\bf g}_{j;v}\in \Z^n.
\]
\end{definition}


\subsection{$Y$-patterns with coefficients}\label{y_pat_ceof}
In analogy with the previous section we introduce a generalization of $Y$-pattern to $Y$-pattern with coefficients.
We study their periodicities in \S\ref{sec:separation formulas} and obtain separation formulas for $Y$-patterns with coefficients which turn out to be very similar to the separation formulas for cluster patterns with coefficients presented in  \cite[\S3]{FZ_clustersIV}.

To have labeled $Y$-seeds on equal footing with labeled seeds we endow them similarly with coefficients as follows.

\begin{definition}\thlabel{def Y seed with coeff}
A (skew-symmetrizable) \emph{labeled $Y$-seed with coefficients} (in $\PP$) is a triple $(\by,\bp,B)$, where
\begin{itemize}
    \item $(\bp,B)$ is a $Y$-seed in $\PP$,
    \item $(\by, B)$ is a $Y$-seed in some universal semifield $ \Q\PP_{\text{sf}}(u_1,\dots, u_\ell)$.  
\end{itemize}
\end{definition}

We define analogously the mutation of labeled $Y$-seeds with coefficients below.  

\begin{definition}
\thlabel{main_definition}
Let $(\by,\bp,B)$ be a labeled $Y$-seed with coefficients, where as above $\by=(y_1,\dots,y_n)$, $\bp= (p_1,\dots,p_n)$ and $B=(b_{ij})$. 
For $k\in [1,n]$, \emph{the mutation in direction $k $ of} $(\bf{y},\bf{p},B)$ is the labeled $Y$-seed with coefficients  $(\by',\bp',B')$, where
\begin{itemize}
    \item $(\bp',B')$ is the mutation in direction $k$ of $(\bp,B)$, and
    \item $\by'=(y_1',\dots,y_n')$, where
\begin{eqnarray}\label{eq:genralformula}
    y'_j := \left\{\begin{matrix} y_j^{-1} &\text{ if } j=k, \vspace{1mm} \\
    y_j \lrp{ p_k^{\llbracket b_{kj}\rrbracket} +  p_k^{\llbracket -b_{kj}\rrbracket} y_k^{-\sgn(b_{kj})} } ^{-b_{kj}} &\text{ if } j\not =k.\end{matrix}\right.
\end{eqnarray}
\end{itemize}
See \thref{p_+-} for the $p^{\llbracket x \rrbracket}$ notation.
\end{definition}

\begin{remark}
\thlabel{unlabeling}
Unlabeled $Y$-seeds with coefficients are obtained by identifying labeled $Y$-seeds with coefficients that differ from each other by simultaneous permutations of the components in $\bp$, in  $\by$, and of the rows and columns of $B$. In a similar fashion, one obtains unlabeled seeds, unlabeled $Y$-seeds and unlabeled seeds with coefficients as in \cite[Definition 4.1]{FZ_clustersIV}. 
\end{remark}

It is straightforward to prove that mutation of $Y$-seeds with coefficients is involutive. This implies that the following notion of $Y$-pattern with coefficients is well defined. 
The underlying semifield for such a pattern is $\Q\PP_{\text{sf}}(u_1,\dots,u_\ell)$.

\begin{definition}\thlabel{Y-pattern with coeff}
A $Y$-\emph{pattern with coefficients} is an assignment $ v \mapsto (\by_v , \bp_v, B_v) $ of a labeled $Y$-seed with coefficients to each  $v \in \T^n$ such that $(\by_{v'},\bp_{v'},B_{v'})$ is the mutation in direction $k$ of $(\by_v, \bp_v,B_v)$ whenever $v,  v'\in \T^n$ are vertices joined by an edge labeled by $k$. The components $\by_v$ are written as $(y_{1;v},\dots,y_{n;v})$.
\end{definition}

A particular choice of coefficients for cluster patterns is introduced in \cite{FZ_clustersIV}. We extend it to $Y$-patterns:

\begin{definition}
A $Y$-pattern with \emph{principal coefficients} at vertex $v_0$ 
refers to a $Y$-pattern with coefficients $v\mapsto (\mathbf y_v,\mathbf p_v,B_v)$ with $\PP=\text{Trop}(p_{1},\dots,p_{n})$ as underlying semifield for coefficients, where $\bp_{v_0}= (p_{1},\dots,p_n)$.
\end{definition}

Recall from \cite[Definition~3.10]{FZ_clustersIV} the rational functions for $Y$-patterns
\begin{eqnarray}\label{Y-rational function}
Y_{j;v}\in \Q_{\text{sf}}(y_1,\dots,y_n).
\end{eqnarray}
They are the components of the (coefficient-free) $Y$-pattern $ v\mapsto ((Y_{1;v},\dots , Y_{n;v}),B_v)$ with values in $\Q_{\text{sf}}(y_1,\dots,y_n)$ and $((y_1,\dots, y_n),B^0)$ as initial labeled $Y$-seed at $v_0$.
In analogy, we have the following definition for the case with coefficients.
We fix for the rest of the section an initial exchange matrix $B^0$. 

\begin{definition}
\thlabel{tildeY-pattern}
Fix an initial $Y$-seed of coefficients $((p_1,\dots,p_n),B^0)$. We define the rational functions 
\begin{eqnarray}\label{tilde Y rational function}
\widetilde{Y}_{j;v}\in \Q\PP_{\text{sf}}(y_1,\dots,y_n)
\end{eqnarray}
as the components of the $Y$-pattern with coefficients $ v\mapsto ((\widetilde{Y}_{1;v},\dots , \widetilde{Y}_{n;v}),\bp_v,B_v)$ whose initial labeled $Y$-seed with coefficients is $((y_1,\dots, y_n),(p_1,\dots, p_n),B^0)$ at $v_0$. Notice that $\widetilde{Y}_{j;v}$ depends on $((p_1,\dots,p_n),B^0)$. Nevertheless, we will omit this in the notation since we think of $((p_1,\dots,p_n),B^0)$ as fixed once and for all. We illustrate this definition in \thref{ex:A2_Y-tildes}.
\end{definition}

In later sections we will specify this definition to the case of $Y$-patterns with principal coefficients. To make notation lighter we denote $\Q \PP_{\text{sf}}(y_1,\dots, y_n)$ by $\mathcal{S}$ from now on. The following result provides a way to compute $\widetilde{Y}_{j;v}$ from $Y_{j;v}^{\mathcal{S}}$ and $p_{j;v}$  (see \thref{universal_ev}). 

\begin{lemma}
\thlabel{separation}
Fix an initial vertex $v_0$ and let $p_j:=p_{j;v_0}$ for all $1\leq j \leq n $. Then for every vertex $v\in \T^n$ and every $j\in [1,n]$ we have 
\begin{equation}
\label{general_separation_formula}
    \widetilde{Y}_{j;v}=\dfrac{Y_{j;v}^{\mathcal{S}}(p_1y_1,\dots,p_ny_n)}{p_{j;v}}.
\end{equation}
\end{lemma}

\begin{proof}
The statement is clear for the initial vertex $v_0$ and all $i\in [1,n]$. 
Let $v,v'\in\mathbb T^n$ be connected by an edge with label $k$ and assume that \eqref{general_separation_formula} holds for $v$ and for all $j$. 
We have to prove that \eqref{general_separation_formula} also holds for $v'$ and all $j$.
The case $j=k$ follows from $p_{j;v'}=p_{j;v}^{-1}$, $Y_{j;v'}=Y_{j;v}^{-1}$ and $ \widetilde{Y}_{j;v'}= \widetilde{Y}_{j;v}^{-1}$. 
So we assume that $j\neq k$. Since (\ref{general_separation_formula}) holds for $v$ we have the following
\begin{align*}
Y^{\mathcal{S}}_{j;v'}(p_1y_1,\dots,p_ny_n)&=Y^{\mathcal{S}}_{j;v}(p_1y_1,\dots,p_ny_n)(1+Y^{\mathcal{S}}_{k;v}(p_1y_1,\dots,p_ny_n)^{-\sgn(b_{kj})})^{-b_{kj}}\\
&= p_{j;v}\widetilde{Y}_{j;v}(1+(p_{k;v}\widetilde{Y}_{k;v})^{-\sgn(b_{kj})})^{-b_{kj}}.
\end{align*}
We distinguish two cases:
\begin{itemize}
    \item[\bf Case 1:] If $b_{kj}\le 0$ we apply the $Y$-pattern mutation formula (\ref{eq:Y-pattern_mutation}) and 
obtain
\begin{equation*}
\dfrac{Y^{\mathcal{S}}_{j;v'}(p_1y_1,\dots,p_ny_n)}{p_{j,v'}}= \dfrac{\widetilde{Y}_{j;v}(1+p_{k;v}\widetilde{Y}_{k;v})^{-b_{kj}}}{(1\oplus p_{k;v})^{-b_{kj}}}=\widetilde{Y}_{j;v}(p^-_{k;v}+p^+_{k;v}\widetilde{Y}_{k;v})^{-b_{kj}},
\end{equation*}
which is exactly what we wanted to prove.

\item[\bf Case 2:] If $b_{jk}>0$ we proceed in an analogous way.
\end{itemize}
\end{proof}

\begin{example}
\thlabel{ex:A2_Y-tildes}
We identify the vertices of the 2-regular tree $\T^2$ with the integers by fixing $v_0=0$. 
In Table~\ref{tab:A2} we illustrate both, a $Y$-pattern and a $Y$-pattern with coefficients in type $A_2$ (i.e. $B^0$ corresponds to an orientations of a Dynkin diagram of type $A_2$). 
Computing only $5$ mutations of the initial $Y$-seed with coefficients, we observe that for all $i\in \Z$ the labeled Y-seeds (with and without coefficients) at a vertex $i$ and $i+5$ are equal as unlabeled $Y$-seeds.
\end{example}

\begin{table}[!htbp]
\label{example:A_2}
    \centering
    \begin{tabular}{|c|c|c|c|c|c|}
    \hline
         $v$ & $B_v$ & $p_{1;v}$ &$p_{2;v}$ & $\widetilde{Y}_{1;v}$ & $\widetilde{Y}_{2;v}$ \\
         \hline \hline
          
          0  & $\left(\begin{smallmatrix} 0& 1 \\ -1 &0\end{smallmatrix}\right)$ & 
          $p_1$& 
          $ p_2$&
          $y_1$ & 
          $y_2$  \\  \hline
          $\updownarrow \mu_2$ \\ \hline
          1  & $\left(\begin{smallmatrix} 0& -1 \\ 1 &0\end{smallmatrix}\right)$ & 
          $p_1(p_2\oplus 1)$  &
          $\dfrac{1}{p_2} $&
          $\dfrac{y_1(p_2y_2+1)}{p_2\oplus 1}$ & 
          $\dfrac{1}{y_2}$ \\ \hline
          $\updownarrow \mu_1$ \\ \hline
          2  & $\left(\begin{smallmatrix} 0& 1 \\ -1 &0\end{smallmatrix}\right)$ & 
          $\dfrac{1}{p_1(p_2\oplus 1)}$&
          $\dfrac{p_1p_2\oplus p_1 \oplus 1}{p_2}$ &
          $\dfrac{p_2\oplus 1}{y_1(p_2y_2+1)}$ &
          $\dfrac{p_1p_2y_1y_2+p_1y_1+1}{y_2(p_1p_2\oplus p_1\oplus 1)} $  \\ \hline
          $\updownarrow \mu_2$ \\ \hline
          3  & $\left(\begin{smallmatrix} 0& -1 \\ 1 &0\end{smallmatrix}\right)$ & 
          $\dfrac{p_1\oplus 1}{p_1 p_2}$&
          $\dfrac{p_2}{p_1p_2\oplus p_1\oplus 1} $&
          $\dfrac{p_1y_1+1}{y_1y_2(p_1\oplus 1)}$ &
          $\dfrac{y_2(p_1p_2\oplus p_1\oplus 1)}{p_1p_2y_1y_2+p_1y_1+1}$ \\ \hline
          $\updownarrow \mu_1$ \\ \hline
          4  & $\left(\begin{smallmatrix} 0& 1 \\ -1 &0\end{smallmatrix}\right)$ & 
          $\dfrac{p_1p_2}{p_1\oplus 1}$&
          $\dfrac{1}{p_1} $ &
          $\dfrac{y_1y_2(p_1\oplus 1)}{p_1y_1+1}$ &
          $\dfrac{1}{y_1}$  \\ \hline
          $\updownarrow \mu_2$ \\ \hline
          5 & $\left(\begin{smallmatrix} 0& -1 \\ 1 &0\end{smallmatrix}\right)$ & 
          $p_2$ &
          $p_1$ &
          $y_2$ & 
          $y_1$ \\ \hline
    \end{tabular}
    \caption{A $Y$-pattern $v\mapsto ((p_{1;v},p_{2;v}),B_v)$ and the corresponding $Y$-pattern with coefficients $v\mapsto ((\widetilde{Y}_{1;v},\widetilde{Y}_{2;v}),(p_{1;v},p_{2;v}),B_v)$ in type $A_2$ introduced in \thref{tildeY-pattern} (see also \thref{rmk:y vs p}). Observe that the first four columns coincide with the first three columns of \cite[Table 1]{FZ_clustersIV} upon change of notation.
    }
    \label{tab:A2}
\end{table}


\subsection{Separation formulas and periodicities}\label{sec:separation formulas}
In this section we analyze properties of $Y$-patterns with coefficients in complete generality (unless specified otherwise).

We apply the universal property of $\Q_{\text{sf}}(y_1,\dots,y_n)$ to the tropical semifield with $ y_1,\dots,y_n \in \text{Trop}(y_1,\dots,y_n)$.
As elements of $\text{Trop}(y_1,\dots, y_n)$ are Laurent monomials in $y_1,\dots,y_n$,
the rational function $Y_{j;v}$ in \eqref{Y-rational function} yield Laurent monomials:
\begin{eqnarray}\label{defining c vector}
Y_{j;v}^{\text{Trop}(y_1,\dots, y_n)} (y_1,\dots, y_n )= \prod_{i=1}^n y_i^{c_{ij;v}}
 =:{\mathbf y}^{\mathbf c_{j;v}}.
\end{eqnarray}

\begin{definition}
\thlabel{c-vectors}
The vector ${\mathbf c_{j;v}}=(c_{1j;v},\dots,c_{nj;v})\in\mathbb Z^n$ in \eqref{defining c vector} is called the 
{\bf c}\emph{-vector} associated to $j$ and $v$.
\end{definition}

As a consequence of \thref{separation} we obtain for $Y$-patterns the analog of Fomin and Zelevinsky's separation formulas for cluster patterns. To state this result we consider the $F$-polynomials in \eqref{F-polynomial} associated to the initial matrix $B$. 

\begin{theorem}[Separation formulas]\thlabel{separation_formulas}
For every $j$ and $v$ we have the following
\begin{equation}
\label{general_separation_formula_2}
    \widetilde{Y}_{j;v}=
    \left(
    \prod^n_{i=1}(F_{i;v}^{-b^v_{ij}})^{\PP}(p_1,\dots, p_n)\right)
    \left(
    \prod^n_{i=1}(F_{i;v}^{b^v_{ij}})^{\mathcal{S}}(p_1y_1,\dots,p_ny_n)\right)
    \mathbf{y}^{{\mathbf c_{j;v}}}.
\end{equation}
\end{theorem}

\begin{proof}
We denote the exchange matrix $B_v=(b_{ij})$ and by $F_{i;v}$ the $F$-polynomial associated to $i$ and $v$. Recall that by \cite[Proposition 3.13]{FZ_clustersIV} 
\begin{eqnarray*}
    Y_{j;v}= \prod_{i=1}^n y_i^{c_{ij;v}}
    \prod^n_{i=1}F_{i;v}^{b_{ij}}.
\end{eqnarray*}
Hence, the universal property of $\Q_{\text{sf}}(y_1,\dots,y_n)$ applied to $\mathcal S\ni p_1y_1,\dots,p_ny_n$ gives
\begin{equation}
\label{eq:sep_1}
    Y^{\mathcal{S}}_{j;v}(p_1y_1,\dots,p_ny_n)=\prod^n_{i=1}(p_iy_i)^{c_{ij;v}}\prod^n_{i=1}(F_{i;v}^{b_{ij}})^{\mathcal{S}}(p_1y_1,\dots,p_ny_n).
\end{equation}
By the universal property of $\Q_{\text{sf}}(y_1,\dots,y_n)$ applied to $\PP\ni p_1,\dots,p_n$ we have
\begin{equation}
\label{eq:sep_2}
p_{j;v}=\prod^n_{i=1}p_i^{c_{ij;v}}\prod^n_{i=1}(F_{i;v}^{b_{ij}})^{\PP}(p_1,\dots, p_n).  
\end{equation}
We obtain the desired formula dividing \eqref{eq:sep_1} by \eqref{eq:sep_2} and then using \eqref{general_separation_formula}.
\end{proof}

Note that only the first factor on the right hand side of \eqref{general_separation_formula_2} involves the auxiliary addition in $\PP$.
 
 \begin{theorem}
 \thlabel{periodicity}
The $Y$-pattern $v\mapsto ((Y_{1;v},\dots, Y_{n;v}),B_v)$ and the $Y$-pattern with coefficients $v \mapsto ((\widetilde{Y}_{1;v},\dots, \widetilde{Y}_{n;v}),\bp_v,B_v)$ share the same periodicities. That is, 
\[
Y_{i;v}=Y_{j;v'} \text{ if and only if } \widetilde{Y}_{i;v}=\widetilde{Y}_{j;v'}.
\]
\end{theorem}
 
\begin{proof}
Using the universal property of $\Q_{\text{sf}}(y_1,\dots,y_n)$ the coefficient $Y$-pattern given by ${v\mapsto (\bp_v,B_v)}$ can be obtained from the $Y$-pattern ${v \mapsto ((Y_{1;v},\dots, Y_{n;v}),B_v)}$ by specifying ${y_i\mapsto p_i}$. 
In particular, $Y_{i;v}=Y_{j;v'}$ implies $p_{i;v}=p_{j;v'}$. Now we can use \thref{separation} to conclude that $Y_{i;v}=Y_{j;v'}$ implies $\widetilde{Y}_{i;v}=\widetilde{Y}_{j;v'}$. 

For the reverse implication we use the universal property of $\Q\PP_{\text{sf}}(y_1,\dots,y_n)$ 
applied to $\Q_{\text{sf}}(y_1,\dots,y_n)$.
We fix $y_1,\dots,y_n\in\Q_{\text{sf}}(y_1,\dots,y_n)$ and the map $\PP\to\Q_{\text{sf}}(y_1,\dots,y_n)$ sending every $p\in \PP$ to $1$. 
The induced morphism of semifields $${\Q\PP_{\text{sf}}(y_1,\dots,y_n)\to \Q_{\text{sf}}(y_1,\dots,y_n)}$$ sends $\widetilde{Y}_{i;v}$ to $Y_{i;v}$ for all $i$ and $v$.
\end{proof}

\subsection{Geometric coefficients and frozen directions}\label{geometric_coefficietns}

So far we have considered patterns in which all $n$ directions are mutable. However, it is very useful to consider patterns in which there are \emph{frozen} (or \emph{non-mutable}) directions. Indeed, this formalism can be used to describe in a simplified way cluster patterns with coefficients in a tropical semifield (see \thref{lem:geometric_and_frozen} below). 
Further, frozen directions can be used to construct partial compactifications of cluster varieties as discussed in \cite[\S9]{GHKK}.
In this subsection we let $m\in \mathbb{N} $ be an arbitrary positive integer.

\begin{definition}
A \emph{cluster pattern with coefficients of rank $n$ with $m$ frozen directions} is an assignment $v \mapsto ({\mathbf x}_v, \mathbf p_v, B_v)$, where 
\begin{itemize}
    \item $v\in \T^n$;
    \item ${\mathbf x}_v$ is a cluster of size $n+m$;
    \item ${\mathbf p}_v$ is a $Y$-seed of size $n+m$;
    \item $B_v$ is a skew-symmetrizable $(n+m)\times (n+m)$-matrix;
    \item for every pair of vertices
    $v,v'\in \T^n$ joined by an edge labelled by $k\in [1,n]$, we have that $({\mathbf x}_{v'}, B_{v'})= \mu_k ({\mathbf x}_v, B_v)$. 
\end{itemize}
We can define in an analogous way \emph{$Y$-patterns with coefficients of rank $n$ and $m$ frozen directions}, and the coefficient-free version of these concepts.
\end{definition}

Observe that given a cluster pattern (or a $Y$-pattern) of rank $n$ without frozen directions one can declare some of these directions to be \emph{frozen} to obtain a cluster pattern of lower rank with frozen directions. 
Similarly, given a cluster pattern (or a $Y$-pattern) with coefficients of rank $n$ and $m$ frozen directions we can decide to \emph{unfreeze} some directions to obtain a pattern of rank $>n$ and with fewer frozen directions.

\begin{definition}
A cluster pattern has \emph{geometric coefficients} if the underlying semifield $\PP$ is a tropical semifield.
\end{definition}

\begin{notation}
Let $B=(b_{ij})$ be a $(n + m)\times (n+m)$-square matrix. We introduce the following notation to denote 4 distinguished submatrices of $B$. 
\begin{itemize}
    \item The \emph{upper-left matrix} $B_{\text{ul}}:=(b_{i,j})_{1\leq i,j \leq n}$, an $n\times n $ subamtrix;
    \item The \emph{lower-left matrix} $B_{\text{ll}}:=(b_{n+i,j})_{1\leq i \leq m, \ 1 \leq j \leq n}$, an $m\times n $ subamtrix;
    \item The \emph{upper-right matrix} $B_{\text{ur}}:=(b_{i,n+j})_{1\leq i \leq n, \ 1 \leq j \leq m}$, an $n\times m $ subamtrix;
    \item The \emph{lower-right matrix} $B_{\text{lr}}:=(b_{n+i,n+j})_{1\leq i,j \leq m}$, an $m\times m $ subamtrix.
\end{itemize}
\end{notation}

\begin{definition}\thlabel{def:extended cluster}
Let $v\mapsto ({\bf x}_v,\bp_v,B_v)$ be a rank $n$ cluster pattern with geometric coefficients in  $\PP=\text{Trop}(p_1,\dots , p_m)$. For every $v\in \T^n$ we define
$$
{\bf x}^{\text{ext}}_v:=(x_{1;v},\dots , x_{n;v},p_1,\dots , p_m),
$$
and refer to it as \emph{the extended cluster}. Further, let $B^{\text{ext}}_v$ be the $(n+m)\times (n+m)$-matrix defined as follows
\begin{itemize}
    \item $(B^{\text{ext}}_v)_{\text{ul}}=B_v$;
    \item $(B^{\text{ext}}_v)_{\text{ll}}=(a_{ij}^v)$, where $p_{j;v}= \prod_{i=1}^m p_i^{a_{ij}^v}$ for every $j \in [1,n] $;
    \item $(B^{\text{ext}}_v)_{\text{ur}}=-(B^{\text{ext}}_v)_{\text{ll}}^T$, where $A^T$ denotes the transpose of a matrix $A$;
    \item $(B^{\text{ext}}_v)_{\text{lr}}$ is the zero $m\times m$-matrix.
\end{itemize}
\end{definition}

\begin{lemma}
\thlabel{lem:geometric_and_frozen}
The assignment $v\mapsto ({\bf x}^{\text{ext}}_v,B^{\text{ext}}_v)$ defines a (coefficient free!) cluster pattern of rank $n$ with $m$ frozen directions. 
\end{lemma}
\begin{proof}
This follows from the discussion at the end of \cite[\S2]{FZ_clustersIV}. 
\end{proof}

{\bf{Warning:}}{ \emph{We use \thref{lem:geometric_and_frozen} frequently to interpret cluster patterns with geometric coefficients as coefficient-free cluster patterns with frozen directions. 
As an important special case, a cluster pattern with principal coefficients of rank $n$ can be treated as a coefficient-free cluster pattern of rank $n$ with $n$ frozen direction.
However, the meaning in geometry of coefficients is very different form the meaning in geometry of frozen directions. 
We systematically think of coefficients as parameters deforming the mutation formulas. 
Frozen directions are just directions in which we do not perform mutations and they can be used to partially compactify cluster varieties.
We discuss these ideas extensively in \S\ref{sec:cluster_duality}.
}}

\begin{example}
Let $v\mapsto (\mathbf x_v,\mathbf p_v,B_v)$ be the cluster pattern with principal coefficients
at $v_0$ of type $A_2$. 
We follow the recipe of \thref{def:extended cluster} using the ingredients from Table~\ref{tab:A2}:
the matrices $B_v$ are in the second column and the monomials $p_{i;v}\in\text{Trop}(p_1,p_2)$ can be read from columns three and four specifying $\oplus$ as in \thref{exp:trop sf}. 

The matrices $B_{v_i}^{\text{ext}}$ for $i\in[0,4]$ are incidence matrices of the quivers shown in Figure~\ref{fig:A_2 prin coeff}. We enclose in a box ${\color {blue}\boxed{\phantom{:}}}$ the vertices corresponding to the rows of the matrices $ (B^{\text{ext}}_v)_{\text{ll}}$. Observe that they change according to the matrix mutation in \eqref{eq:matrix mut}, as \thref{lem:geometric_and_frozen} predicts.
\end{example}

\noindent
\begin{minipage}{\linewidth}
\captionsetup{type=figure}
\begin{center}
\begin{tikzpicture}

\draw[fill] (0,0) circle [radius=0.05];
\draw[fill] (1,0) circle [radius=0.05];
\draw[->] (0.25,0) -- (0.75,0);
\draw[->] (0,.75) -- (0,.25);
\draw[->] (1,.75) -- (1,.25);

\node[below] at (0,0) {\small 1};
\node[below] at (1,0) {\small 2};
\node[blue] at (0,1) {\tiny\boxed{3}};
\node[blue] at (1,1) {\tiny\boxed{4}};

\node at (.5,-1) {0};

\begin{scope}[xshift=3cm]
\draw[fill] (0,0) circle [radius=0.05];
\draw[fill] (1,0) circle [radius=0.05];
\draw[<-] (0.25,0) -- (0.75,0);
\draw[->] (0,.75) -- (0,.25);
\draw[<-] (1,.75) -- (1,.25);

\node[below] at (0,0) {\small 1};
\node[below] at (1,0) {\small 2};
\node[blue] at (0,1) {\tiny\boxed{3}};
\node[blue] at (1,1) {\tiny\boxed{4}};

\node at (.5,-1) {1};

\begin{scope}[xshift=3cm]
\draw[fill] (0,0) circle [radius=0.05];
\draw[fill] (1,0) circle [radius=0.05];
\draw[->] (0.25,0) -- (0.75,0);
\draw[<-] (0,.75) -- (0,.25);
\draw[<-] (1,.75) -- (1,.25);

\node[below] at (0,0) {\small 1};
\node[below] at (1,0) {\small 2};
\node[blue] at (0,1) {\tiny\boxed{3}};
\node[blue] at (1,1) {\tiny\boxed{4}};

\node at (.5,-1) {2};

\begin{scope}[xshift=3cm]
\draw[fill] (0,0) circle [radius=0.05];
\draw[fill] (1,0) circle [radius=0.05];
\draw[<-] (0.25,0) -- (0.75,0);
\draw[<-] (0,.75) -- (0,.25);
\draw[->] (1,.75) -- (1,.25);
\draw[->] (.25,.25) -- (.75,.75);

\node[below] at (0,0) {\small 1};
\node[below] at (1,0) {\small 2};
\node[blue] at (0,1) {\tiny\boxed{3}};
\node[blue] at (1,1) {\tiny\boxed{4}};

\node at (.5,-1) {3};

\begin{scope}[xshift=3cm]
\draw[fill] (0,0) circle [radius=0.05];
\draw[fill] (1,0) circle [radius=0.05];
\draw[->] (0.25,0) -- (0.75,0); 
\draw[->] (0,.75) -- (0,.25); 
\draw[<-] (.25,.25) -- (.75,.75); 
\draw[->] (.75,.25) -- (.25,.75); 

\node[below] at (0,0) {\small 1};
\node[below] at (1,0) {\small 2};
\node[blue] at (0,1) {\tiny\boxed{3}};
\node[blue] at (1,1) {\tiny\boxed{4}};

\node at (.5,-1) {4};
\end{scope}
\end{scope}
\end{scope}
\end{scope}

\end{tikzpicture}
    
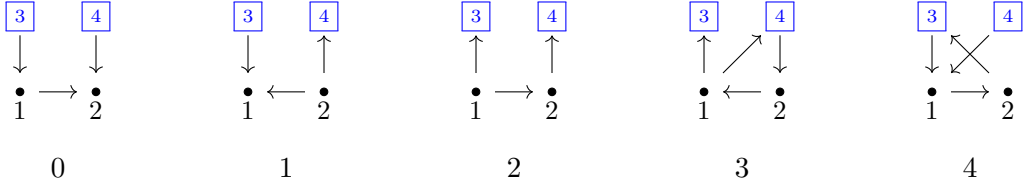
\captionof{figure}{\label{fig:A_2 prin coeff}
Quivers associated to a type $A_2$ cluster pattern with principal coefficients at $v_0$. The boxes refer to frozen vertices.    }
\end{center}
\end{minipage}

\begin{example}\thlabel{ex:GrassCoeff}
This example corresponds to the cluster structure on the coordinate ring of the affine cone of the Grassmannian $\Grass_2(\C^5)$ (for more details see \S\ref{ex:RW}). Consider the cluster pattern $v\mapsto (\mathbf x_v,\mathbf p_v,B_v)$ of type $A_2$ with geometric coefficients in $\text{Trop}(p_{12},p_{23},p_{34},p_{45},p_{15})$.
Let 
\[
\mathbf{x}_{v_0}=(x_1,x_2), \ \mathbf{p}_{v_0}=(p_{12}^{-1}p_{23}p_{34}^{-1},p_{34}p_{45}^{-1}p_{15}) \ \text{ and } \
B_{v_0}=\left(\begin{smallmatrix}0&-1\\ 1&0\end{smallmatrix}\right).
\]
The associated $7\times7$-matrix $B_{v_0}^{\text{ext}}$ is the incidence matrix of the quiver shown in
Figure~\ref{fig:quiver Gr(2,5)}, identifying $x_1=p_{13}$ and $x_2=p_{14}$. 
The corresponding extended cluster is
\[
\mathbf x_{v_0}^{\text{ext}}=(p_{13},p_{14},p_{12},p_{23},p_{34},p_{45},p_{15}).
\]
\end{example}


\subsection{Comparison of cluster and $Y$-patterns with coefficients}\label{sec:relation_cluster_patterns}

We explain how every cluster pattern with coefficients $v \mapsto (\mathbf{x}_v,\bp_v, B_v)$ gives rise to a $Y$-pattern with coefficients in a natural way. For each $v\in \T^n $ let $B_v=(b_{ij}^v)$. We define
\begin{eqnarray}\label{y-tilde}
\tilde{y}_{j;v}:=\displaystyle \prod_{i=1}^n x_{i;v}^{\ b^v_{ij} }
\end{eqnarray}
for each $v\in \T^n $ and $j\in [1,n]$. Moreover, let $\tilde{\by}_v=(\tilde{y}_{1;v},\dots ,\tilde{y}_{n;v} )$.

\begin{remark}
\thlabel{about_y-tilde}
Notice that \eqref{y-tilde} still makes sense when we have trivial coefficients, i.e. $p_{i;v}=1$ for all $i \in [1,n]$ and all $v\in \T^n$. Moreover, we can extend the definition of the $ \tilde{y}$'s for cluster patterns with coefficients and frozen directions.
\end{remark}

\begin{lemma}
\thlabel{comparison}
The assignment $v \mapsto (\tilde{\by}_v,\bp_v ,B_v) $ is a $Y$-pattern with coefficients.
\end{lemma}

\begin{proof}
For pairs of vertices $v,v' \in \T^n$ joined by an edge labeled by $k$ we have to show
\[
\tilde{y}_{j;v}\lrp{ p_{k;v}^{\llbracket b^v_{kj}\rrbracket} +  p_{k;v}^{\llbracket -b^v_{kj}\rrbracket}\tilde{y}_{k;v} } ^{-b^v_{kj}}=\tilde{y}_{j;v'}.
\]
To make notation lighter let $b_{ij}^v=b_{ij}$ and $ b_{ij}^{v'}=b'_{ij}$ for all $i,j \in[1,n] $. 
Since $b'_{ik}=-b_{ik}$ for all $i\in [1,n]$ the result follows for $j=k$. 
So we assume $j\neq k$. 
Notice that if $b_{kj}=0$ then $b'_{ij}=b_{ij}$ for all $i\in [1,n]$. 
In particular, $\tilde{y}_{j,v'}=\tilde{y}_{j,v}$, which is what we need to prove.
Therefore, we assume $b_{jk}\neq0$ and distinguish two cases.
\begin{itemize}
    \item[\bf Case 1:] If $b_{kj}<0$, we have the following
\begin{eqnarray*}
 \tilde{y}_{j;v} \left( p_{k;v}^{\llbracket b_{kj}\rrbracket} \right. &+& \left. p_{k;v}^{\llbracket -b_{kj}\rrbracket}\tilde{y}_{k;v} \right) ^{-b_{kj}}
=
\displaystyle 
\prod_{i=1}^n x^{ b_{ij} }_{i;v}
\left(
p_{k;v}^- +  p_{k;v}^+
\displaystyle\prod_{i=1}^{n}
x_{i;v}^{ b_{ik} }
\right)^{-b_{kj}}\\
\displaystyle 
&=&\prod_{i=1}^n x^{ b_{ij} }_{i;v}
\left(
\left(
\displaystyle p_{k;v}^-
\prod_{i:b_{ik}<0} x^{ -b_{ik} }_{i;v}  +p_{k;v}^+
\prod_{i:b_{ik}>0} x^{ b_{ik} }_{i;v} \right) 
\displaystyle \prod_{i:b_{ik}<0} x^{ b_{ik} }_{i;v}
\right)^{-b_{kj}}\\
\displaystyle 
&=&x^{ -b_{kj}}_{k;v'}
\prod_{i\neq k} x^{b_{ij} }_{i;v}
\prod_{i:b_{ik}<0} x^{-b_{ik}b_{kj}}_{i:v}= \prod_{i=1}^n x^{\ b'_{ij}}_{i;v'}=\tilde{y}_{j;v'}.
\end{eqnarray*}

\item[\bf Case 2:] If $b_{kj}>0$, the proof is done in exactly the same way.
\end{itemize}
\vspace{-.5cm}\end{proof}

\begin{remark}
\thlabel{y-hat-rem}
Given a cluster pattern with coefficients $v\mapsto (\mathbf x_v,\mathbf p_v,B_v)$ in \cite[Equation (3.7)]{FZ_clustersIV} the authors define $\hat{\by}_v=(\hat{y}_{1;v},\dots , \hat{y}_{n;v})$, where
\begin{equation}
\label{y-hat}
\hat{y}_{j;v}:=p_{j;v}\displaystyle \prod_{i=1}^n x_{i;v}^{\ b^v_{ij} }.    
\end{equation}
They prove in \cite[Proposition 3.9]{FZ_clustersIV} that $v\mapsto (\hat{\by}_v,B_v)$ is a $Y$-pattern. If the cluster pattern with coefficients $v\mapsto (\mathbf x_v,\mathbf p_v,B_v)$ has frozen directions then the $Y$-pattern $v\mapsto (\hat{\by}_v,B_v)$ has the same frozen directions.
Moreover, observe that \eqref{y-hat} is also well defined for cluster patterns with trivial coefficients.

If the cluster pattern $v\mapsto (\mathbf x_v,\mathbf p_v,B_v)$ has geometric coefficients we can consider the cluster pattern with frozen directions $v\mapsto ({\bf x}^{\text{ext}}_v,B^{\text{ext}}_v)$ afforded by \thref{lem:geometric_and_frozen}. 
Let $v\mapsto (\hat{\by}^{\text{ext}}_v,B^{\text{ext}}_v)$ be the $Y$-pattern with frozen directions given by \cite[Proposition 3.9]{FZ_clustersIV} applied to the cluster pattern with frozen directions ${v\mapsto ({\bf x}^{\text{ext}}_v,B^{\text{ext}}_v)}$.
Similarly, let $v\mapsto (\tilde{\by}^{\text{ext}}_v,B^{\text{ext}}_v)$ be the $Y$-pattern with frozen directions (and trivial coefficients) given by \thref{comparison} applied to $v\mapsto ({\bf x}^{\text{ext}}_v,B^{\text{ext}}_v)$. 
Then $\hat{\by}^{\text{ext}}_v=\tilde{\by}^{\text{ext}}_v$ for all $v$.
\end{remark}

\begin{remark}\thlabel{y-tilde and p-map}
In the context of cluster varieties the patterns defined by the $\hat{y}$'s and $\tilde{y}$'s are both very useful. The former will define a map $p:\cA \to \cX$ between cluster varieties with frozen directions (see \thref{def:p-map}); the latter a map $p:\cA_{\bp} \to \cX_{\bp}$ of cluster varieties with coefficients and frozen directions.

\end{remark}

\section{Cluster duality for cluster varieties with coefficients}\label{sec:cluster_duality}
In this section we give a geometric perspective of cluster and $Y$-patterns with coefficients. One of the most important ideas we want to communicate is the following:

\begin{center}  
    \emph{we make a distinction between frozen variables
and coefficients}.
\end{center}
A cluster variety with $n$ mutable directions and $m$ frozen directions is a scheme over $\C $ built by gluing complex tori isomorphic to $(\C^{\ast})^{n+m}$ via the usual mutation formulas.
The coefficients we introduce live in a ring $R$ of form $\C [p_1^{\pm 1},\dots, p_r^{\pm 1}]$ in case they are invertible, respectively of form $\C [p_1,\dots, p_r] $ if they are not invertible.
A cluster variety with (geometric) coefficients, with $n$ mutable directions and $m$ frozen directions is a scheme over $R$ obtained by gluing schemes of the form $\Spec (R[z^{\pm 1}_1,\dots, z^{\pm 1}_{n+m}])$. 
To glue these schemes we consider \emph{mutations with coefficients} which we think of as deformations of the usual mutation formulas. 
On the $\cA$-side, the notions of frozen variables and coefficients interact with each other as explained in \S\ref{geometric_coefficietns}: 
on the one hand the coefficients are organized in a $Y$-pattern; 
on the other, $\cA$-cluster variables associated to frozen directions give rise in a natural way to a $Y$-pattern in a tropical semifield. 
In this case the mutation formulas of the cluster pattern with frozen directions will be exactly the same as the mutation formulas of the cluster pattern with coefficients in the $Y$-pattern associated to the frozen $\cA$-variables.
However, the disparity between coefficients and frozen variables becomes more evident when we consider $\cX$-cluster varieties since $\cX$-cluster variables associated to frozen directions usually change after applying a mutation in a mutable direction. 
Moreover, it is crucial to make this distinction to discover the correct notion of duality for cluster varieties with coefficients. 
This notion of duality was predicted in \cite[\S7, Definition~7.15]{GHKK} and here we develop it in full detail.

In \cite{GHK_birational} and \cite{GHKK} $\cA$-cluster varieties with (geometric)  coefficients are considered as $\cA$-cluster varieties with frozen directions. Therefore, they are treated as schemes over $\C$. 
In \S\ref{sec:cluster_var_coeff} we develop a new perspective and treat them as schemes over $\Spec(R)$.
Further, we introduce the concept of $\cX $-cluster varieties with coefficients and cluster varieties with specified coefficients. 
The former are schemes over $\Spec(R)$
while the latter are schemes over $\C $. 
Moreover, in \S\ref{sec:cluster duality coeff} we show how the results of the previous section lead to the notion of cluster duality for cluster varieties with coefficients and with specified coefficients. 

\subsection{Reminders on cluster varieties}\label{sec:cluster_var}
Cluster patterns and $Y$-patterns define schemes over $\C$ by interpreting the mutation formulas as birational maps between algebraic tori (see \cite[p.149]{GHK_birational}). 
For this geometric perspective, it is convenient to fragment the notion of a seed (resp. $Y$-seed) into two parts, one changes under mutation while the other remains unchanged.

We introduce some new nomenclature. We believe this makes it easier to work with the algebraic and geometric settings simultaneously.

\begin{definition}
The (skew-symmetrizable) \emph{fixed data} $\Gamma$ consists of 

\begin{itemize}
    \item a finite set $I$ of \emph{directions} with a subset of \emph{unfrozen directions} $\Iuf$; 
    \item a lattice $N$ of rank $ |I|$;
    \item a saturated sublattice $ \Nuf \subseteq N$ of rank $|I_{\text{uf}}|$;
    \item a skew-symmetric bilinear form $\lbrace \cdot , \cdot \rbrace: N\times N \to \Q$;
    \item a sublattice $N^{\circ}\subseteq N$ of finite index satisfying
    \[\lbrace N_{\text{uf}}, N^{\circ}\rbrace \subset \Z  \ \text{ and }\  \lbrace N,N_{\text{uf}}\cap N^{\circ}\rbrace \subset \Z;
    \]
    \item a tuple of positive integers $(d_i:i\in I)$ with greatest common divisor 1;
    \item $M=\Hom(N,\Z)$ and $M^{\circ}=\Hom(N^{\circ},\Z)$.
\end{itemize}
We say that $\Gamma$ is \emph{skew-symmetric} if $d_i=1$ for all $i\in I$.
\end{definition}

\begin{definition}
An \emph{$N$-seed} for $\Gamma$ is a tuple $s=(e_i\in N :i\in I)$ such that $\{e_i:i\in I\}$ is a basis of $N$ and further
\begin{itemize}
\item $\{ e_i : i\in I_{\text{uf}}\} $ is a $\Z $-basis of $N_{\text{uf}}$;
\item $\{d_i e_i : i\in I\}$ is a $\Z$-basis of $N^{\circ}$;
\item the elements $v_i:=\{ e_i, \cdot \}\in M^{\circ}$ are non-zero for every $i\in \Iuf$.
\end{itemize}
Consider the dual basis $\{e^*_i\in M:i\in I\}$ of $M$. 
The \emph{associated $M^{\circ}$-seed} is the tuple $(f_i:i\in I)$, where $\{f_i:i\in I\}$ is a basis of $M^{\circ}$ given by
\[
f_i:=d_i^{-1}e_i^*\in M^{\circ}.
\]
\end{definition}
For a fixed $N$-seed $s$ we define the matrix $\boldsymbol{\varepsilon}_s:=(\epsilon_{ij})_{i,j\in I}$ by $\epsilon_{ij}:=\{ e_i,e_j\}d_j$.
For the pairing given by evaluation we use the notation
\[
\langle \cdot,\cdot \rangle:N\times M^{\circ}\to \Q.
\]

\begin{definition}
Given an $N$-seed $s$ and $k \in \Iuf $, the \emph{mutation in direction $k$ of} $s$ is the $N$-seed $\mu_k(s)=({e'}_i:i\in I)$ given by 
\begin{equation}
\label{e_mutation}
e_i':=\begin{cases} e_i+[\epsilon_{ik}]_+e_k & i\neq k,\\
-e_k&i=k.
\end{cases}
\end{equation}
The mutation in direction $k$ of the associated $M^{\circ}$-seed is $(f'_i:i\in I)$, where
\begin{equation}
\label{f_mutation}
f_i':=\begin{cases} -f_k+\sum_j [-\epsilon_{kj}]_+f_j& i=k,\\
f_i&i\not=k.
\end{cases} 
\end{equation}
\end{definition}

Unlike the mutations introduced in the preceding section, notice that mutation of $N$-seeds is not an involution, i.e. $\mu_k(\mu_k(s))\neq s $. This motivates the following definition.

\begin{definition}
The \emph{oriented tree} $\orT$ is the canonical orientation of $\T^n$ determined by the following conditions 
\begin{itemize}
    \item the $n$ edges incident to $v_0$ are oriented in outgoing direction from $v_0$;
    \item every vertex $v\neq v_0$ has one incoming edge and $n-1$ outgoing edges.
\end{itemize}
We write $v\overset{k}{\longrightarrow}v'\in \orT$ to indicate that the edge in between the vertices $v,v'$ of $\orT$ is oriented from $v$ to $v'$ and labeled by $k$.
\end{definition}

\begin{definition}
Let $n= |\Iuf|$. An \emph{N-pattern} is the assignment $ v \mapsto s_v$ of an $N$-seed $s_v$ to every $v\in \orT$, such that $s_{v'}=\mu_k(s_v)$ whenever $v\overset{k}{\longrightarrow}v'\in \orT$. We let $s_0:=s_{v_0}$ be the initial $N$-seed. 
\end{definition}

In what follows we identify a vertex $v\in \orT$ with the $N$-seed $s_v$ and write $s_v\in \orT$.

\begin{definition}
Let $\Gamma$ be fixed data, $s$ an $N$-seed and $k\in I_{\text{uf}} $. The \emph{$\cA$-cluster mutation} is the birational map on the torus $T_{N^{\circ}}:=\Spec\C[M^{\circ}]$
\[
\mu_{k;s}: T_{N^{\circ}} \dashrightarrow T_{N^{\circ}} 
\] 
specified by the pull-back formula at the level of characters of $T_{N^{\circ}}$ by
\begin{equation}
\label{eq:A_mut}
\mu_{k;s}^*(z^{m}):=z^{m} (1 +  z^{v_k})^{-\langle d_ke_k,m\rangle},    
\end{equation}
for every $z^m\in \C [M^{\circ}]=\C [T_{N^{\circ}}]$.
Similarly, the \emph{$\cX$-cluster mutation} is the birational map on the torus $T_M:=\Spec\C[N]$
\[
\mu_{k;s}: T_M  \dashrightarrow T_M 
\] 
specified by the pull-back formula
\begin{equation}
\label{eq:X_mut}
    \mu_{k;s}^*(z^{n}):=z^{n} (1 +  z^{e_k})^{- [n,e_k]_s},
\end{equation}
for every $z^n \in \C[N]=\C[T_M]$, and $[\cdot,\cdot]_s:N\times N\to \Q$ the bilinear form given by
\[
[e_i,e_j]_s=\epsilon_{ij}.
\]
\end{definition}
\begin{remark}
To make notation lighter we denote the $\cA$- and $\cX$-mutation by the same symbol. 
The kind of tori that we are considering tells us which kind of mutation we have to consider. These mutation formulas are also called \emph{cluster transformations}.
\end{remark}

\begin{remark}
\thlabel{rem:notation}
We can pass from the geometric framework of this section to the algebraic framework of the last section as follows: 
\begin{itemize}
    \item set $ I=[1,n+m]$ and $I_{\text{uf}}=[1,n]$;
    \item the set of frozen directions is $\Ifr:=I\setminus \Iuf$ (see \S\ref{geometric_coefficietns});
    \item the skew-symmetrizable matrix $B=(b_{ij})_{(i,j)\in \Iuf \times \Iuf}$ is the transpose of the (sub-)matrix associated to $[\cdot,\cdot]_s$, i.e. for all $i,j \in \Iuf$:
    \[
    b_{ij}=\epsilon_{ji};
    \]
    \item to construct a coefficient-free $Y$-pattern with frozen directions we take as ambient semifield $\PP=\Q_{\text{sf}}(z^{e_1},\dots, z^{e_{n+m}})$. The initial $Y$-seed is $((z^{e_1},\dots,z^{e_{n+m}}),B)$;
    \item to construct a coefficient-free cluster pattern with frozen directions we take as ambient field $ \mathcal{F}=\Q(z^{f_1},\dots,z^{f_{n+m}}) $. The initial seed is $((z^{f_1},\dots , z^{f_{n+m}}),B)$. 
\end{itemize}
\end{remark}

\begin{definition}\thlabel{mu_maps}
Fix an $N$-pattern $v \mapsto s_v$. For every seed $s \in \orT$ let $T_{N^{\circ},s}$ be a torus isomorphic to $T_{N^{\circ}}$ endowed with the basis of characters $\{ z^{f_i}\in \C[M^{\circ}]:s=(e_i:i\in I) \}$. 
We define the birational map $\mu^{\phantom{-1}}_{s_0,s}:T_{N^{\circ},s_0}\dashrightarrow T_{N^{\circ},s}$ as
\[
\mu^{\phantom{-1}}_{s_0,s}:=\begin{cases} \text{Id} & \text{if } s=s_0, \\ \mu_{k_l;s_{l}}\circ \cdots \circ \mu_{k_0;s_0} & \text{if } s\neq s_0,
\end{cases}
\]
where $s_0 \overset{k_0}{\longrightarrow} s_1 \overset{k_1}{\longrightarrow} \cdots \overset{k_l}{\longrightarrow}  s_{l+1}=s$ is the unique oriented path from $s_0$ to $s$. For two seeds $s,s'\in \orT$ we define 
\begin{equation}
    \mu^{\phantom{-1}}_{s,s'}:=\mu^{\phantom{-1}}_{s_0,s'}\circ \mu_{s_0,s}^{-1}: T_{N^{\circ},s} \dashrightarrow T_{N^{\circ},s'}.
\end{equation}
By construction $\mu^{\phantom{-1}}_{s',s''}\circ \mu^{\phantom{-1}}_{s,s'}= \mu^{\phantom{-1}}_{s,s''}$. We define the analogous transformations for the tori $T_{M,s}$.
\end{definition}

\begin{lemma}
[{(\cite[Proposition 2.4]{GHK_birational})}]
\thlabel{gluing}
Let $\mathcal{I}$ be a set and $\{ S_i :i \in \mathcal{I}  \}$ be a collection of integral separated schemes of finite type over a locally Noetherian ring $R$, with birational maps $f_{ij}:S_i \dashrightarrow S_j$ for all $i,j$, verifying the cocycle condition $f_{jk}\circ f_{ij}=f_{ik}$ as rational maps and such that $f_{ii}=\text{Id}$. 
Let $U_{ij}\subset S_i$ be the largest open subscheme such that $f_{ij}:U_{ij}\to f_{ij}(U_{ij})$ is an isomorphism. 
Then there is a scheme 
\[
S=\bigcup\limits_{i \in \mathcal{I}}S_{i}
\]
obtained by gluing the $S_i$ along the open sets $U_{ij}$ via the maps $f_{ij}$.
\end{lemma}

\begin{remark}\thlabel{gluing_general}
In \cite[Proposition 2.4]{GHK_birational} the authors prove \thref{gluing} when $R$ is a field because this is the generality they need. 
In the next section we will take $R$ to be a (Laurent) polynomial ring, which is Noetherian. 
The proof of \cite[Proposition 2.4]{GHK_birational} is based on \cite[Proposition 6.5.4]{EGAI} and \cite[Proposition 8.2.8]{EGAI}.
Both results hold in the more general case in which $R$ is a locally Noetherian ring. 
Therefore, the same proof can be used to prove \thref{gluing}. 
\end{remark}

\begin{definition}
Fix an $N$-pattern. We apply \thref{gluing} to the collection of tori ${\{  T_{N^{\circ},s}: s\in \orT \}}$ and the birational maps $\mu_{s,s'} $ to define the scheme
\[
\cA_{\Gamma,s_0}:=\bigcup\limits_{s\in \orT}T_{N^{\circ},s}.
\]
This is the \emph{$\cA $-cluster variety} associated to  $\Gamma$ and $s_0$. 
In a completely analogous way we define the \emph{$\cX$-cluster variety} associated to $\Gamma$ and $s_0$
\[
\cX_{\Gamma,s_0}:=\bigcup\limits_{s\in \orT}T_{M,s}.
\]
We say that a cluster variety (of any kind) has \emph{frozen directions} if $I\backslash\Iuf\neq \varnothing$.
\end{definition}

\begin{notation}
\thlabel{dictionary}
The characters of the torus $T_M$ (resp. $T_{N^\circ}$) are of the form $z^n$ (resp. $z^m$) for $n\in N$ (resp. $m\in M^{\circ}$). Fix an $N$-pattern, from now on we use the following identifications relating the notation of \S\ref{sec:Y-pattern} and \S\ref{sec:cluster_duality}:
\[
\begin{tabular}{l c l}
    $Y_{i;v}=X_{i;s_v}=z^{e_{i;s_v}}$ & and & $x_{i;v}=A_{i;s_v}=z^{f_{i;s_v}}$.
\end{tabular}
\]
However, notice that the identification $X_{i;s}=z^{e_{i;s}}$ (resp. $A_{i;s}=z^{f_{i;s}}$) is only valid on the torus $T_{M,s}$ (resp. $T_{N^{\circ},s}$). We refer to the symbols $X_{i;s}$ (resp. $A_{i;s}$) as $\cX$-cluster variables (resp. $\cA$-cluster variables).
\end{notation}

\begin{remark}
In this paper (coefficient-free) cluster varieties  are schemes over the complex numbers. It is possible to work over more general rings of definition such as $\Z$. 
This is the approach taken for example in \cite{FG_quantum}.  The varieties considered in this paper are simply the $\C$-points of such $ \Z$-schemes. 
\end{remark}

\begin{remark}\thlabel{rem:no s in mu}
We will write $\mu_k$ instead of $\mu_{k;s}$ since the tori we glue already tell us which $s$ we are considering.
\end{remark}

There is a fundamental map relating $\cA $- and $\cX$-cluster varieties. It is induced by the skew-symmetric bilinear form $\{\cdot,\cdot\}$ as follows. There are two natural maps
\hspace{-.5cm}

\begin{center}
\begin{tabular}{lcllclclll}
     $p_1^*:$ & $N_{\text{uf}}$ & $\to$ & $M^{\circ}$  &  and \ \  &  $p_2^*:$ & $N$ & $\to$ & $M^{\circ}/N_{\text{uf}}^{\perp}$ \\
      & $n$& $\mapsto$ & $(n_1\mapsto \{n,n_1\})$& & & $n$ & $\mapsto$ & $(n_2\mapsto\{n,n_2\})$ 
\end{tabular}
\end{center}
for $n_1\in N^{\circ}$, respectively $ n_2\in N_{\text{uf}}\cap N^\circ$.

\begin{definition}\thlabel{def:p-map}
A \emph{$p^*$-map} is a map $N\to M^\circ$, such that the diagram commutes:
\begin{center}
\begin{tikzcd}
    N_{\text{uf}} \arrow{r}{p_1^*} \arrow[hookrightarrow]{d} & M^\circ  \arrow[twoheadrightarrow]{d} \\
    N \arrow{r}{p_2^*}  \arrow{ru}{p^*} & M^\circ/N_{\text{uf}}^\perp 
\end{tikzcd}
\end{center}
Different choices of such a map differ by different choices of maps $N/N_{\text{uf}}\to N_{\text{uf}}^\perp$. 
A map ${p^*:N\to M^\circ}$ gives a map $p:T_{N^\circ,s}\to T_{M,s}$ for every seed $s$. We can see from \thref{comparison} and \thref{y-hat-rem} that $p$ commutes with the cluster transformations. Therefore, it further extends to a globally defined map $p:\cA_{\Gamma,s_0}\to \cX_{\Gamma,s_0}$.
\end{definition}

\begin{remark}\thlabel{dictionary p-map}
For simplicity assume we are in the skew-symmetric case and fix an $N$-seed $s=(e_i:i\in I)$ in an $N$-pattern with $s=s_v$. 
We can choose $p^*(e_i)=\{e_i,\cdot\}=e_i B_v^T=b_i$, the $i$th column of the matrix $\boldsymbol{\varepsilon}_s^T=B_v$ (see \thref{rem:notation}).
Comparing to \thref{y-tilde and p-map} and the $\hat{y}_{i;v}$ defined in \eqref{y-hat} this corresponds to
\[
p^*(X_{i;s})=p^*(z^{e_i})=z^{b_i}=\prod_{j\in I} (z^{f_j})^{b_{ji}}=\prod_{j\in I} A_{j;s}^{b_{ji}}=\hat{y}_{i;v}.
\]
\end{remark}

\subsection{Cluster varieties with coefficients}\label{sec:cluster_var_coeff}

In what follows we give a new perspective on $\cA$-cluster varieties with coefficients treating them as schemes over (Laurent) polynomial rings.
Further, we introduce the concept of $\cX $-cluster varieties with coefficients and cluster varieties with specified coefficients. 

We fix the tropical semifield $\PP=\text{Trop}(p_1,\dots, p_r)$. In particular, recall that  $\PP $ is a free abelian group generated by $p_1, \dots, p_r$ written multiplicatively. 

\begin{definition}
Let $\Gamma$ be fixed data and $s_{v_0}=(e_i:i\in I)$ an initial $N$-seed (in this situation we will usually call the pair $(\Gamma,s_{v_0})$ initial data). 
A $Y$-pattern $v \mapsto (\bp_v, B_v)$ of rank $|I_{\text{uf}}|$ is \emph{compatible with} $\Gamma $ and $s_{v_0}$ if the initial matrix $B_{v_0}=(b_{ij})$ is the transpose of the matrix associated to $[ \cdot , \cdot ]_{s_{v_0}} $ (see \thref{rem:notation}). 
\end{definition}

Let $R$ be either $\C[p_1,\dots, p_r]$ or $\C[p^{\pm 1}_1,\dots, p^{\pm 1}_r]$\footnote{In \S\ref{sec:tropical duality} (see Figure~\ref{fig:Delta F +}) we elaborate on the difference between invertible and non-invertible coefficients.}. 
In particular, $\Spec (R)$ is either an algebraic torus or an affine space. 
For a lattice $L$ we denote 
\[
T_{L}(R) :=\Spec(\C [L^*] \otimes_{\C} R) = T_{L} \times_{\C} \Spec(R).
\]
In particular, the coordinate ring of $T_{L}(R)$ is isomorphic to $R[L^{\ast}]$, where $L^*=\Hom(L,\Z)$.

Recall that for every $p\in \PP$, we have $p^+$ and $p^-$, which are monomials in $ p_1,\dots, p_r$, from \thref{p_+-}.

\begin{definition}
\thlabel{geometric_mut_w_coef}
Let $\Gamma$ be fixed data, $s$ an $N$-seed, $p\in \PP$ and $k\in I_{\text{uf}}$. The associated $\cA$-\emph{cluster mutation with coefficients} is the birational map
\[
\mu_{k;s;p}: T_{N^{\circ}}(R) \dashrightarrow T_{N^{\circ}}(R)
\] 
specified by the pull-back formula 
\begin{equation}
\label{eq:A_mut_w_coef}
\mu_{k;s;p}^*(z^{m}):=z^{m} (p^- +  p^+z^{v_k})^{-\langle d_ke_k,m\rangle},    
\end{equation}
for every $z^m \in R[M^\circ]$.
Similarly, the associated $\cX$-\emph{cluster mutation with coefficients} is the birational map
\[
\mu_{k;s;p}: T_M (R) \dashrightarrow T_M (R)
\] 
specified by the pull-back formula for every $z^n \in R[N]$:
\begin{equation}
\label{eq:X_mut_w_coef}
    \mu_{k;s;p}^*(z^{n}):=z^{n} (p^- +  p^+z^{e_k})^{- [n,e_k]_s}.
\end{equation}
\end{definition}

Using the above mutations with coefficients we construct cluster varieties with coefficients. Before defining them, first note that given fixed data $\Gamma $, an initial $N$-seed $s_0=(e_i : i \in I)$ and a compatible $Y$-pattern $v \mapsto (\bp_v, B_v)$, then for every seed $s\in\orT$ we obtain a birational map
$\mu_{s_0,s}: T_{N^\circ, s_0}(R) \dashrightarrow T_{N^\circ, s}(R) $ analogously to \thref{mu_maps}:
\[
\mu^{\phantom{-1}}_{s_0,s}:=\begin{cases} \text{Id} & \text{if } s=s_0, \\ \mu_{k_l;s_l;p_{k_l;v_l},}\circ \cdots \circ \mu_{k_0;s_0;p_{k_0;v_0}} & \text{if } s\neq s_0,
\end{cases}
\]
taking the unique oriented path from $s_0$ to $s$, $s_0 \overset{k_0}{\longrightarrow} s_1 \overset{k_1}{\longrightarrow} \cdots \overset{k_l}{\longrightarrow}  s_{l+1}=s$. For $s,s'\in\orT$ we define 
\begin{equation}
    \mu^{\phantom{-1}}_{s,s'}:=\mu^{\phantom{-1}}_{s_0,s'}\circ \mu_{s_0,s}^{-1}: T_{N^{\circ},s}(R) \dashrightarrow T_{N^{\circ},s'}(R).
\end{equation}

We define the birational transformation $\mu_{s,s'}: T_{M,s}(R) \dashrightarrow T_{M,s'}(R)$ for $s,s'\in \orT$ in a completely analogous way.
For both kinds of tori the cocycle condition $  {\mu_{s',s''} \circ \mu_{s,s'} =\mu_{s,s''} }$ follows by definition.

\begin{definition}
\thlabel{cluster_varieties_w_coef}
Consider fixed data $\Gamma $, an initial $N$-seed $s_0=(e_i : i \in I)$ and a compatible $Y$-pattern $v \mapsto (\bp_v, B_v)$. We apply \thref{gluing} and \thref{gluing_general} to the tori $\{ T_{N^{\circ},s}(R):s \in \orT \}$ and the $\cA$-cluster mutation with coefficients $\mu_{s,s'} $ to define the $\cA$-\emph{cluster variety with coefficients} as the scheme 
\[
\cA_{\Gamma,s_0,\bp_{v_0}}:=\bigcup_{s\in \orT}T_{N^{\circ},s}(R).
\]
In a completely analogous way using the tori $\{ T_{M,s}(R): s \in \orT \}$ and the $\cX$-cluster mutation with coefficients we define the $\cX$-\emph{cluster variety with coefficients}
\[
\cX_{\Gamma,s_0,\bp_{v_0}}:=\bigcup_{s\in \orT}T_{M,s}(R).
\]
We say that a cluster variety with coefficients has \emph{frozen directions} if $I\setminus I_{\text{uf}}\not = \varnothing$. 
\end{definition}

\begin{remark}
Similarly to the coefficient-free case, we write $\mu_{k;p}$ instead of $\mu_{k;s;p}$.
\end{remark}

\begin{remark}\thlabel{p-map with coefficients}
Note that the notion of initial data for cluster varieties with coefficients coincides with the notion of initial data in the coefficient-free case. 
Moreover, the map $p^*:N\to M^{\circ}$ defined in \thref{def:p-map} relies only on the notion of initial data. 
Hence, we obtain as above for every seed $s$ a map
$p:T_{N^\circ,s}(R)\to T_{M,s}(R)$ via the pull-back on characters $z^{m}\in R[M^\circ]$ given by $p^*$. 
In view of \thref{comparison} this extends to a map
\[
p:\cA_{\Gamma,s_0,\mathbf p_{v_0}}\to \cX_{\Gamma,s_0,\mathbf p_{v_0}}.
\]
\end{remark}

\begin{remark}
\thlabel{dictionary_2}
\thref{rem:notation} extends to cluster and $Y$-patterns with coefficients in $\PP$:
\begin{itemize}
    \item the associated $Y$-pattern with coefficients and frozen directions lives in
    \[
    \Q\PP_{\text{sf}}(z^{e_1},\dots,z^{e_{n+m}})=\Q(p_1,\dots,p_r)_{\text{sf}}(z^{e_1},\dots,z^{e_{n+m}});
    \]
    \item the associated cluster pattern with coefficients and frozen directions lives in
    \[
    \Q\PP(z^{f_1},\dots,z^{f_{n+m}})=\Q(p_1,\dots,p_r,z^{f_1},\dots,z^{f_{n+m}}).
    \]
\end{itemize}
\end{remark}

In account of \thref{rem:notation}, \thref{dictionary} and \thref{dictionary_2} we obtain the following:

\begin{lemma}
\thlabel{equiv_formulas}
The $\cA$-cluster mutation with coefficients corresponds to cluster mutation with geometric coefficients, and the $\cX$-cluster mutation with coefficients corresponds to $Y$-mutation with geometric coefficients.
\end{lemma}
\begin{proof}
We only prove the statement on the $\cX$-side, the proof on the $\cA$-side is analogous. We have to check that formula $ \eqref{eq:X_mut_w_coef}$ is equivalent to formula \eqref{eq:genralformula} for ${\PP=\text{Trop}(p_1,\dots,p_r)}$. This follows at once from the following computation

\begin{eqnarray*}
  \mu_{k;p}^*(z^{e'_k})&=& z^{e'_k}(p^{-}+p^{+}z^{e_k})^{-\epsilon_{ik}}=\begin{cases}
  z^{e_k}(p^{-}z^{-e_k}+p^+)^{-\epsilon_{ik}} & \text{if } \epsilon_{ik}>0, \\
  z^{e_k}(p^{-}+p^{+}z^{e_k})^{-\epsilon_{ik}}& \text{if } \epsilon_{ik}\leq 0.
\end{cases}
\end{eqnarray*}
\end{proof}

\begin{remark}
\thlabel{coefficients_vs_frozen_variables}
We can view cluster varieties with coefficients as schemes over $\C$ via the canonical inclusion $ \C \hookrightarrow R$. 
On the $\cA$-side frozen variables can be thought of as coefficients as explained in \S\ref{geometric_coefficietns}. 
Therefore, an $\cA$-cluster variety with coefficients considered as a scheme over $\C$ has the structure of a coefficient-free $\cA$-cluster variety with frozen directions. 
We just have to construct an appropriate $\Gamma$ following \thref{def:extended cluster}. 
However, this is not the case on the $\cX$-side because $\cX$-variables associated to frozen directions may change after mutation in a non-frozen direction. 
\end{remark}

Let $\Gamma $ be fixed data, $s_0$ an initial $N$-seed and $v \mapsto (\bp_v, B_v)$ a compatible $Y$-pattern.
Let $L$ be either the lattice $N^{\circ}$ or $M$. The canonical inclusion
\[
T_{L}\hookrightarrow T_{L}(R)=T_{L}\times_{\C}\Spec(R)
\]
induces a left action of the group scheme $T_{L}$ on $T_L(R)$. 
Explicitly, this action is given by the canonical projection $T_{L}\times_{\C} \Spec(R)\to \Spec (R)$.
Moreover, if $s, s'\in \orT$ 
we have a commutative diagram 
\begin{equation*}
   \xymatrix{
    T_{L,s}(R)\ar@{-->}^{\mu_{s,s'}}[rr]\ar[dr] &  & T_{L,s'}(R) \ar[dl] \\
    &  \Spec(R)   &
    }.
\end{equation*}
This implies that we can glue these morphisms to obtain morphisms of schemes 
\begin{equation}
   \xymatrix{
    \cA_{\Gamma,s_0,\bp_{v_0}}\ar_{\pi_{\cA}}[dr] &  & \cX_{\Gamma,s_0,\bp_{v_0}} \ar^{\pi_{\cX}}[dl] \\
    &  \Spec(R)   &
    }.
\end{equation}
Notice that these morphism are flat. 
Indeed, as flatness is a local property to verify it we can restrict to $T_{L,s}(R)$.
The induced map of rings $R \to R[L^*]$ is flat, so both $\pi_{\cA}$ and $\pi_{\cX}$ are flat morphisms of schemes.

\begin{definition}
\thlabel{def:cv_w_special_coeffs}
Let $\lambda\in \Spec(R)$ be a fixed parameter. The \emph{$\cA$-cluster variety with specified coefficients} $\cA_{\Gamma,s_0,\bp_{v_0},\lambda}$ is the fiber $\pi^{-1}_{\cA}(\lambda)$. 
We define the \emph{$\cX$-cluster variety with specified coefficients} $\cX_{\Gamma, s_0,\bp_{v_0}, \lambda}$ as the fiber $\pi_{\cX}^{-1}(\lambda)$. 
\end{definition}

Notice that if $\lambda\in \Spec(R)$ is a closed point then $\cA_{\Gamma,s_0,\bp_{v_0},\lambda}$ is a scheme over $\C $ in a canonical way. 
In particular, if $\vb{1}=(1,\dots,1)$, then
\[
\cA_{\Gamma,s_0,\bp_{v_0},\vb{1}}=\cA_{\Gamma,s_0}\ \ \ \text{and} \ \ \  \cX_{\Gamma,s_0,\bp_{v_0},\vb{1}}=\cX_{\Gamma,s_0}.
\]
More generally, let $\mathcal{M} =\{\prod_{i=1}^r p_i^{a_i}\in \PP:a_i\geq 0 \}$ and $p\in \PP$. Fix a morphism of monoids $\phi: \mathcal{M} \to \C$. The \emph{mutation with specified coefficients associated to $\mu_{k;p}$ and $\phi$} is the birational map
\[
\mu_{k;\phi(p)}:T_{N^{\circ}}\dashrightarrow T_{N^{\circ}}
\]
given by the pull-back formula 
\[
\mu_{k;\phi(p)}^*(z^m):= z^m(\phi(p^-)+\phi(p^+)z^{v_k})^{-\langle d_ke_k,m\rangle},
\]
for every $z^m\in R[M^{\circ}]$. 
We define in an analogous way the $\cX$-mutation with specified coefficients. 
If $\lambda=(\lambda_1,\dots,\lambda_r)$ is a closed point of $\Spec(R)$ and $\phi:\mathcal{M}\to \C$ is defined by $\phi(p_i)=\lambda_i$ then 
\[
\cA_{\Gamma, \bp, \lambda}= \bigcup_{s\in \orT} T_{N^{\circ},s},
\]
where the tori $T_{N^{\circ},s}$ are glued by $\mu_{k;\phi(p)}$. Similarly, we have
\[
\cX_{\Gamma, \bp, \lambda}= \bigcup_{s\in \orT} T_{M,s},
\]
where the tori $T_{M,s}$ are glued with the mutations with specified coefficients.

\begin{remark}\thlabel{rem:dual_specified_coeff}
In \cite{GHK_birational} the authors introduce the cluster varieties $\cA_{\mathbf t}$ for $\mathbf t\in \C^r$ and call them \emph{cluster varieties with general coefficients}. 
These varieties are obtained specifying   \thref{def:cv_w_special_coeffs} to the case of principal coefficients and identifying the parameters $p_1,\dots , p_r$ with the complex values  $\phi(p_1)=t_1,\dots,\phi(p_r)=t_r$. 
In \cite[\S7, p.~555]{GHKK}, Gross, Hacking, Keel and Kontsevich mention that 
it is not clear how to dualize the birational gluing maps for $\cA$-varieties with specified coefficients as it is not obvious how the parameters should be treated.
The formalism developed in this section is meant to solve this problem.
We would like to mention that $\cA_{\mathbf t}$ and $\cA_{\mathbf t'}$ are isomorphic if $\mathbf t,\mathbf t'\in \Spec(R)$ are generic closed points. 
We prove the analogous result for the $\cX$-side in \thref{iso of fibres} below. 
\end{remark}

\subsection{Cluster duality with coefficients}\label{sec:cluster duality coeff}

As explained in \cite{GHK_birational}, the $\cA$- and $\cX$-mutation formulas are canonically dual to each other. 
In this subsection we recall this duality and extend it to the context of cluster varieties with coefficients and specified coefficients.

\begin{definition}
The \emph{Langlands dual of $\Gamma$} is the fixed data $\Gamma^{\vee}$ given by
\begin{itemize}
    \item $I^{\vee}=I$ and $\Iuf^{\vee}=\Iuf$; 
    \item $N^{\vee}=N^{\circ}$;
    \item $ (\Nuf)^{\vee} =\Nuf \cap N^{\circ}$ ;
    \item $\lbrace \cdot , \cdot \rbrace^{\vee}=d^{-1}\lbrace \cdot , \cdot \rbrace $, where $d:=\text{lcm}(d_i:i\in I)$;
    \item $(N^{\vee})^{\circ}=dN$;
    \item $d^{\vee}_i=dd_i^{-1}$ for all $i\in I$;
    \item $M^{\vee}=M^{\circ}$ and $(M^{\vee})^{\circ}=d^{-1}M$.
\end{itemize}
The \emph{Langlands dual of $s$} is $s^{\vee}:=(d_ie_i:i\in I)$. The \emph{Fock--Goncharov dual cluster varieties of $\cA_{\Gamma,s}$ and $\cX_{\Gamma,s}$} are
\[
\cA^{\vee}_{\Gamma,s}:=\cX_{\Gamma^{\vee},s^{\vee}}\ \ \ \text{and} \ \ \ \cX^{\vee}_{\Gamma,s}:=\cA_{\Gamma^{\vee},s^{\vee}}.
\]
\end{definition}

\begin{remark}
Notice that if $B$ is the initial matrix associated to $ \Gamma$ and $s$ then $-B^{T}$ is the initial matrix associated to $\Gamma^{\vee}$ and $s^{\vee}$. 
In particular, in the skew-symmetric case we have $\Gamma^{\vee}=\Gamma$ and $s^{\vee}=s$. 
Moreover, $\Gamma$ (resp. $s$) is equal to $(\Gamma^{\vee})^{\vee}$ (resp. $(s^{\vee})^{\vee}$) up to the morphism of lattices $dN \to N$ given by $n\mapsto d^{-1}n$.
\end{remark}

\begin{definition}
\thlabel{dual_mutation}
Every pair $(n,m)\in N^{\circ} \times M^{\circ}$ with $\langle n,m \rangle =0$ defines a mutation\footnote{In this context, a mutation is a birational map $\mu: T \dashrightarrow T$ of algebraic tori preserving the canonical volume form of $T$.} 
\[
\mu^{\phantom{x}}_{(n,m)}:T_{N^{\circ}} \dashrightarrow T_{N^{\circ}}
\]
specified by the formula
\begin{equation*}
\mu_{(n,m)}^*(z^{m'}):=z^{m'} (1+z^m)^{\langle m',n\rangle},
\end{equation*}
for every $z^{m'}\in \C [M^{\circ}]$. The \emph{dual mutation} is defined as the birational map
\[
\mu^{\phantom{x}}_{(m,-n)}:T_{M^{\vee}} \dashrightarrow T_{M^{\vee}}
\]
associated to the pair $(m,-n)\in \Hom(N^{\circ}\times M^{\circ}, \Z)=M^{\circ} \times N^{\circ} =M^{\vee}\times N^{\vee}$. 
\end{definition}

The $ \cA$-cluster mutation \eqref{eq:A_mut} associated to $\Gamma$ and $s$ is $\mu_{(-d_ke_k,v_k)}$, and the $ \cX$-cluster mutation \eqref{eq:X_mut} associated to $\Gamma^{\vee}$ and $s^{\vee}$ is $\mu_{(v_k,d_ke_k)}$. 
This explains the duality between $\cA$- and $\cX$-cluster varieties.

\begin{definition}
Let $p\in \text{Trop}(p_1,\dots, p_r)$ and $(n,m)\in N^{\circ} \times M^{\circ}$. The \emph{mutation associated to} $(n,m)$ \emph{with deformation parameter} $p$ is the birational map
 \[
\mu^{\phantom{x}}_{(n,m);p}:T_{N^{\circ}}(R) \dashrightarrow T_{N^{\circ}}(R)
\]
defined at the level of characters for every $z^{m'}\in R[M^{\circ}]$ by
\begin{equation}
\label{general_mutation_w_coeficients}
\mu_{(n,m);p}^*(z^{m'}):=z^{m'} (p^- + p^+z^m)^{\langle m',n\rangle}.
\end{equation}
Let $\phi:\mathcal{M}\to \C$ be a map of monoids. The \emph{mutation associated to $(n,m)$, $\phi$ and $p$} is 
\begin{equation}
\label{general_mutation_w_s_coeficients}
\mu_{(n,m);\phi(p)}^*(z^{m'}):=z^{m'} (\phi(p^-) + \phi(p^+)z^m)^{\langle m',n\rangle}. 
\end{equation}
\end{definition}

We define a notion of duality for \eqref{general_mutation_w_coeficients} and \eqref{general_mutation_w_s_coeficients} that generalizes \thref{dual_mutation}. Moreover, this notion of duality specializes to the duality between $\cA_{\Gamma,s}$ and $\cX_{\Gamma^{\vee},s^{\vee}}$ when we take trivial coefficients, i.e. $R=\C$. 

\begin{definition}
Let $v\mapsto (\bp_v,B_v)$ be an initial $Y$-pattern. The $Y$-pattern ${v \mapsto (\bp_v^\vee,B_v^{\vee})}$ is defined by setting 
\[
\bp_{v_0}^\vee=\bp_{v_0} \ \text{and} \ B^{\vee}_{v_0}=-B_{v_0}^{T}.
\]
Let $\bp_{v}=(p_{1;v},\dots , p_{n;v})$ be a $Y$-seed of the initial $Y$-pattern. The \emph{mutation dual to} $ \mu_{(n,m);p_{k;v}}$ is 
\begin{equation*}
    \mu_{(m,-n);p^{\vee}_{k;v}}: T_{M^{\vee}}(R)\dashrightarrow T_{M^{\vee}}(R).
\end{equation*}
Explicitly, at the level of characters $\mu_{(m,-n);p^{\vee}_{k;v}}$ is given by
\begin{equation*}
z^{n'}\mapsto z^{n'} ((p^{\vee}_{k;v})^- + (p^{\vee}_{k;v})^+z^{-n})^{\langle n',m\rangle}
\end{equation*}
for every $z^{n'}\in R[N^{\vee}]$.
Similarly, the \emph{mutation dual to} $ \mu_{(n,m);\phi(p_{k;v})}$ is 
\begin{equation*}
    \mu_{(m,-n);\phi(p^{\vee}_{k;v})}: T_{M^{\vee}}\dashrightarrow T_{M^{\vee}}.
\end{equation*}
\end{definition}

Observe that if the $Y$-pattern $v\mapsto (\bp_v,B_v)$ is compatible with $ \Gamma $ and $s$ then the $Y$-pattern $v \mapsto (\bp_v^\vee,B_v^{\vee})$ is compatible with $\Gamma^{\vee} $ and $s^\vee$.

\begin{definition}
The \emph{duals of} $\cA_{\Gamma,s_0,\bp_{v_0}}$ and $\cX_{\Gamma,s_0,\bp_{v_0}}$ are
\[
\cA^{\vee}_{\Gamma,s_0,\bp_{v_0}}:=\cX_{\Gamma^\vee,s_0^\vee,\bp_{v_0}^\vee}\ \ \ \text{and} \ \ \ \cX^{\vee}_{\Gamma,s_0,\bp_{v_0}}:=\cA_{\Gamma^{\vee},s^{\vee}_0,\bp^{\vee}_{v_0}}.
\]
\end{definition}

\subsection{Cluster varieties with principal coefficients}\label{sec:cluster_var_prin_coeff}
We specify the constructions given in \S\ref{sec:cluster_var_coeff} and \S\ref{sec:cluster duality coeff} to the case of principal coefficients.

\begin{definition}\thlabel{def:Gamma_prin}
The fixed data with principal coefficients $\Gamma_{\text{prin}}$ associated to given fixed data $\Gamma$ is specified by
\begin{itemize}
    \item the indexing set $\widetilde{I}$ is a disjoint union of two copies of $I$;
    \item the set of unfrozen indices $\widetilde{I}_{\text{uf}}$ is the subset $\Iuf$ of the first copy of $I$;
    \item $\widetilde{N}:=N\oplus M^{\circ}$;
    \item $\widetilde{N}_{\text{uf}}:=\Nuf\oplus 0$;
    \item the skew-symmetric bilinear form $\lbrace \cdot , \cdot \rbrace:\widetilde{N}\times \widetilde{N}\to \Q$ is given by \[
    \{ (n_1,m_1),(n_2,m_2) \}:= \{ n_1,n_2\}+\langle n_1,m_2\rangle - \langle n_2,m_1\rangle;
    \]
    \item $\widetilde{N}^{\circ}:=N^{\circ}\times M$;
    \item the same $d_i$'s as for $\Gamma$.
\end{itemize}
Moreover, the $N$-seed $s$ defines the $\widetilde{N}$-seed $\tilde{s}=((e_1,0),   \dots, (e_n,0),(0,f_1),\dots (0,f_n))$.
\end{definition}

\begin{definition}
The $\cA$-cluster variety with \emph{principal coefficients} associated to $\Gamma$ and $s$ is the $\cA$-cluster variety with coefficients
\[
\cAp :=\cA_{\Gamma_{\text{prin}},\tilde{s}}.
\]
\end{definition}

\begin{remark}
In light of \thref{coefficients_vs_frozen_variables}, $\cAp$ can be viewed as a coefficient-free $\cA$-cluster variety with frozen directions. In \cite{GHKK} the authors define the cluster variety $\cX_{\text{prin}}$ as the (coefficient-free) $ \cX$-cluster variety associated to $(\Gamma^{\vee})_{\text{prin}}$ and $\widetilde{(s^\vee)}$. From the framework developed above we can see that

\begin{center}
    \emph{$\cXp$ is cluster dual to $\cAp$ as a scheme over $\C $, i.e.\\
    as a coefficient-free cluster variety with frozen directions.}
\end{center}
The dual of $\cAp $ as a cluster variety with coefficients, denoted by $\Xfam$, is a very interesting object and the main object of study throughout the rest of the paper.
\end{remark}

We conclude this section with definitions and notation that will be useful in the remainder of the paper.  
So far our notation for cluster varieties includes the subscript $\Gamma$. From now on we drop this subscript and denote $\cA_{\Gamma,s_0}$ simply by $\cA$, resp. $\cA_{\Gamma,s_0,\bp_{v_0}}$ by $\cA_{\bp_{v_0}}$.  We use analogous notation for $\cX$-cluster varieties. 

\begin{definition}
A rational function on a cluster variety (of whichever type) is called a {\it{cluster monomial}} if it is a Laurent monomial in the variables of some seed.
\end{definition}

\begin{notation}
\thlabel{tilde_extension}
Let $A_{i;s}$ be a cluster variable of $\cA$. By the Laurent phenomenon (\thref{Laurent_phenomenon}) the $A_{i;s}$ are global regular functions on $\cA$. We denote by $\widetilde{A}_{i;s}$ the extension of $A_{i;s}$ to $\cAps{s_0}$. These extensions are again globally regular. Explicitly, in the language of the preceding section we have 
\[
X^{\text{FZ}}_{i;v}=\widetilde{A}_{i;s_v}.
\]
We extend this notation to cluster monomials. 
Similarly, if $X_{i;s}$ is an $\cX$-cluster variable associated to $\cX$, we let $\widetilde{X}_{i;s}$ denote the extension of $X_{i;s}$ to $\Xfam_{s_0} $. 
Like the $\cX$-variables themselves, these extensions are only locally defined.
\end{notation}

We often want to express functions in terms of the variables of the initial seed $s_0$.
We introduce the following notation for this purpose.

\begin{notation}
\thlabel{in_seed_pullback}
Let $\mathcal{V}$ be a cluster variety of whichever type, with or without frozen directions and coefficients.
Let $f$ be a function on any torus in the atlas of $\mathcal{V}$.
We denote by $ \mu_{s_0}^*\lrp{f} $ 
the pull-back of $f$ to the initial torus associated to $s_0$,
given by iterated mutations.
\end{notation}

\section{\texorpdfstring{Tropical duality, the $\gv$-fan and the special completion of $\cX$}{Tropical duality and stuff}} \label{sec:tropical duality}

In this section we recall some basic facts about the tropical geometry of cluster varieties.
We explain how two important results in this area (tropical duality of \cite{NZ} and the fan structure of the cluster complex, see \cite[Theorem~2.13]{GHKK})
lead very naturally to the {\it{special completion}} $\Xsp$ of the $\cX$-variety (see \cite{FG_X}).

\subsection{Tropicalization}\label{sec:tropicalization}

Cluster varieties have a particularly rich theory of tropicalization.
As described in \cite{GHK_birational}, cluster varieties are a very special class of log Calabi--Yau varieties.
As such, we can describe the tropicalization of a cluster variety in terms of divisorial discrete valuations and view the tropical space as encoding the cluster variety's log geometry. 
Meanwhile, in \cite{FG09} Fock and Goncharov described cluster varieties as positive schemes.
From this perspective, we can apply a {\it{functor of points}} and ask for the $\PP$-points of a cluster variety, for any semifield $\PP$.
In particular, we can take $\PP$ to be a {\it{tropical semifield}}  (see \cite[\S1.1]{FG09}).
These perspectives are discussed and related in \cite[\S2]{GHKK}. 
For us it will be enough to recall the following facts:
\begin{itemize}
    \item For every lattice $L$, the $\PP$-points of $T_L$ are given by 
    \[
    T_L^{\trop}(\PP)\cong L\otimes_{\Z} \PP,
    \]
    where $\otimes$ is the tensor product of abelian groups.
    \item Mutation $\mu_k : T_{L,s}\dashrightarrow T_{L,s'}$ induces a piecewise linear isomorphism 
    \[
    \mu_k^{\trop}: T_{L,s}^{\trop}(\PP) \to T_{L,s'}^{\trop}(\PP).
    \]
    In particular, the inclusions of cluster charts $T_{N^\circ}\hookrightarrow \cA$ and $T_{M}\hookrightarrow \cX$ induce bijections of $\PP$-points:
    \[
      T_{N^\circ}^{\trop}(\PP)\cong  \cA^{\trop}(\PP) \ \ \text{and} \ \ T_{M}^{\trop}(\PP)\cong \cX^{\trop}(\PP).
    \]
  \end{itemize}
We often consider $\Z$ and $\R$ endowed with  auxiliary addition given by ${x\oplus y:= \min(x,y)}$ for $x,y\in \Z$ or $\R$. Whenever we write $\cX^\trop(\Z)$ (resp.  $\cX^\trop (\R)$), we implicitly mean taking the tropical points with respect to $\Z$ (resp. $\R$) with the aforementioned semifield structure. 
In particular, tropicalizing a mutation is given by the usual tropicalization of a rational function (replace $+$ by $\min$, and $\cdot$ by $+$).

\begin{remark}
\thlabel{warning}
Let $ \mathcal{V}$ be a cluster variety. Even if $\mathcal{V}^\trop(\Z)$ and $\mathcal{V}^\trop (\R)$ can be identified as sets with $\Z $ and $\R$, respectively, these tropical spaces do not have a natural notion of addition. However, they possess in a natural way the structure of an $\mathbb{N}$-torsor (we can multiply tropical points by elements of $\mathbb{N}$) and a piecewise linear structure.
\end{remark}

\subsection{Tropical duality}\label{sec:trop_duality}

Fix an initial exchange matrix $B$ corresponding to an initial $N$-seed $s$ (see \thref{rem:notation}). Recall the definitions of {\bf g}- and {\bf c}-vectors (\thref{g-vectors} and \thref{c-vectors}, respectively). Let us recall a list of fundamental results concerning these vectors.

\begin{definition}
Let $v\in \T^n$. The {\bf c}-matrix $ 
C_{v}^{B}$ is the square matrix whose $j^{\mathrm{th}}$ column is the {\bf c}-vector ${\bf c}_{j;v}$.
Similarly, the {\bf g}-matrix $ G_{v}^{B}$ 
is the square matrix whose $j^{\mathrm{th}}$ column is the {\bf g}-vector ${\bf g}_{j;v}$.
\end{definition}

\begin{theorem}
\thlabel{c_sing-coherence} 
(Sign-coherence of {\bf c}-vectors \cite[Corollary 5.5]{GHKK})
For each $i\in[1,n]$ and $v\in \T^n$ the {\bf c}-vector ${\bf c}_{i;v}$ is non-zero and has either all entries non-negative or all entries non-positive.
\end{theorem}

\begin{theorem}(Tropical duality  \cite{NZ})
\thlabel{th:tropical_duality}
For any skew-symmetrizable initial exchange matrix $B$ and any $v\in \T^n$, we have 
\begin{eqnarray}
    (G^B_v)^T=(C^{-B^T}_v)^{-1}.
\end{eqnarray}
\end{theorem}

\begin{remark}
Nakanishi and Zelevinsky proved \thref{th:tropical_duality} provided sign coherence of {\bf c}-vectors holds. 
Special cases of \thref{c_sing-coherence} were known before the most general case was tackled in \cite{GHKK}. 
\end{remark}

\begin{corollary}
Let $v\in \T^n$. The {\bf g}-vectors  ${\bf g}^B_{1;v}, \dots ,{\bf g}^B_{n;v} $ span a rational strongly convex polyhedral cone $ \mathcal{G}^B_v$ of dimension $n$ in $ \R^n$. The dual cone  $\lrp{\mathcal{G}^B_v}^\vee=\mathcal{C}^{-B^T}_v$ is the rational polyhedral cone spanned by the {\bf c}-vectors ${\bf c}^{-B^T}_{1;v}, \dots ,{\bf c}^{-B^T}_{n;v} $. 
\end{corollary}

\begin{theorem}(\cite[Theorem 2.13]{GHKK})\thlabel{g-fan}
For any initial data the cones $\chams{v}^B$ are the maximal cones of a possibly infinite simplicial fan.
\end{theorem}

This was originally conjectured by Fock and Goncharov in \cite[Conjecture~1.3]{FG_X},
and lead to the notion of the {\it{special completion}} of an $\cX$-variety as we will see below.
In their proof of \thref{g-fan}, Gross, Hacking, Keel and Kontsevich build {\it{scattering diagrams}} for cluster varieties
(see \cite[\S1]{GHKK}, \cite{GS11} and \cite{KS14} where the authors use the term \emph{wall crossing structures}). 
They show that a subset of the scattering diagram (called the {\it{cluster complex}}) has a generalized fan structure, with chambers full-dimensional simplicial cones when there are no frozen directions.
They then identify these chambers with the cones $\chams{v}^B$.
When there are frozen directions,
the chambers in the scattering diagram are no longer strongly convex\footnote{This is why we say {\it{generalized}} fan, as in \cite[Definition~6.2.2]{CLS}.}.
This lack of strong convexity corresponds to invertibility of frozen $\cA$-variables.
Allowing these frozen variables to vanish partially compactifies $\cA$, and correspondingly cuts a convex subset $\Xi$ out of the scattering diagram\footnote{$\Xi$ is usually described as $\lrc{W^\trop\geq 0}$, where $W:\cA^\vee \to \C$ is the Landau--Ginzburg potential associated to this partial compactification of $\cA$.}.
This subset $\Xi$ inherits a fan structure from the scattering diagram,
and the maximal cones of this fan are precisely the cones
$\chams{v}^B$.
More precisely, we have the following definition.

\begin{definition}
Let $\Delta^+$ be the (Fock–Goncharov) cluster complex associated to the initial data $\lrp{\Gamma,v_0}$  without frozen directions (see \cite[Definition 2.9]{GHKK}). 
Let $F\subset I $ be a subset of the index set. 
We denote by $\Delta^+_F$ the subfan of $\Delta^+$ consisting of cones that can be linked to the initial cone  by iterated mutations at vertices in $I\setminus F$. 
We call $\Delta^+_F$ the {\it{$\gv$-fan}} associated to the initial data $\lrp{\Gamma',v_0}$
obtained from $\lrp{\Gamma,v_0}$ by freezing $F \subset I$.
\end{definition}

See Figure~\ref{fig:Delta F +} for an example in type $A_3$. 

\noindent
\begin{minipage}{\linewidth}
\captionsetup{type=figure}
\begin{center}
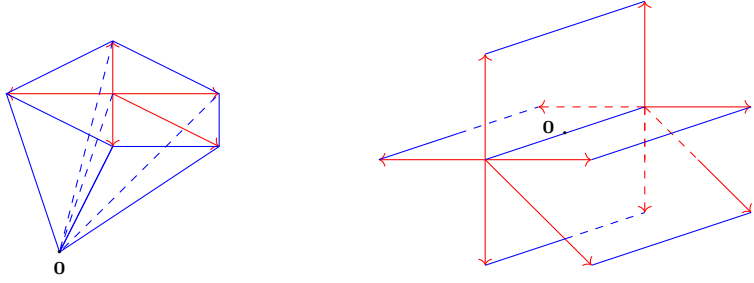

\begin{tikzpicture}[scale=.7]

\draw[->, red] (0,0) -- (0,1);
\draw[->, red] (0,0) -- (2,0);
\draw[->, red] (0,0) -- (2,-1);
\draw[->, red] (0,0) -- (0,-1);
\draw[->, red] (0,0) -- (-2,0);

\draw[blue] (0,1) -- (2,0) -- (2,-1) -- (0,-1) -- (-2,0) -- (0,1);
\draw[blue] (-2,0) -- (-1,-3) -- (0,-1);
\draw[blue] (0,-1) -- (-1,-3);
\draw[blue] (-1,-3) -- (2,-1);
\draw[dashed, blue] (0,1) -- (-1,-3) -- (2,0);
\draw[dashed, blue] (0,0) -- (-1,-3);
\node[below] at (-1,-3) {\tiny $\bf 0$};
\node at (-1,-3) {$\cdot$};

\begin{scope}[xshift=7cm, yshift=-1.25cm]

\draw[->, red] (0,0) -- (0,2);
\draw[->, red] (0,0) -- (2,0);
\draw[->, red] (0,0) -- (2,-2);
\draw[->, red] (0,0) -- (0,-2);
\draw[->, red] (0,0) -- (-2,0);

\draw[->, red] (3,1) -- (3,3);
\draw[->, red, dashed] (3,1) -- (1,1);
\draw[->, red] (3,1) -- (5,1);
\draw[->, red,dashed] (3,1) -- (3,-1);
\draw[dashed, red] (3,1) -- (4,0);
\draw[->, red] (4,0) -- (5,-1);
\draw[dashed, red] (3,1) -- (3,0);

\draw[blue] (0,0) -- (3,1);
\node[above,left] at (1.5,.625) {\tiny $\bf 0$};
\node at (1.5,.5) {$\cdot$};
\draw[blue] (2,0) -- (5,1);
\draw[blue] (0,2) -- (3,3);
\draw[blue] (2,-2) -- (5,-1);
\draw[blue] (0,-2) -- (1.5,-1.5);
\draw[blue, dashed] (1.5,-1.5) -- (3,-1);
\draw[blue] (-2,0) -- (-.5,.5);
\draw[blue, dashed] (-.5,.5) -- (1,1);

\end{scope}

\end{tikzpicture}
    \captionof{figure}{\label{fig:Delta F +}
    Comparison of $\gv$-fan $\Delta^+_F$ and cluster complex $\cc$ for $A_3$ quiver with one frozen direction.
    Here $Q = \tc{blue}{1}\leftarrow 2 \leftarrow 3$ and $F= \lrc{1}$. 
    A schematic of the $\gv$-fan appears on the left and corresponds to the case in which $A_1$ is  non-invertible.  The cluster complex is on the right, corresponding to an invertible $A_1$ variable.
    }
\end{center}
\end{minipage}

\begin{remark}
We do not lose any generality in the above definition by starting in the case without frozen directions, since
we can reduce to this case by simply unfreezing all directions.
\end{remark}

\begin{remark}\thlabel{rem:ConesSeeds}
Until this point we have been working with labeled clusters and labeled $Y$-seeds. Notice that the cones $ \mathcal{G}^B_v$ and $ \mathcal{G}^B_{v'}$ are equal whenever the unlabeled cluster associated to $v$ is equal to the unlabeled cluster associated to $v'$ (see \thref{unlabeling}). 
Algebraically this corresponds to the periodicity of the cluster pattern associated to $B$ and geometrically to the redundancies in the atlas of $\cA $.
By \cite[Theroem 5.9]{Nak19} the periodicity of the cluster pattern associated to $B$ is the same as the periodicity of the $Y$-pattern associated to $B$. Further, by \thref{periodicity} and \cite[Proposition 6.1]{CaoLi} we know that any cluster or $Y$-pattern with coefficients with initial matrix $B$ has the same periodicity as its coefficient-free counterpart.
Therefore, to avoid redundancies, we can index the tori in the atlas of a cluster variety (of either type $\cA$ or $\cX$, with or without coefficients) by the maximal cones of the corresponding $\gv$-fan rather than vertices of $\T^n$, e.g. we can replace ``$T_{N^\circ,s}$ for $s\in \orT$'' by ``$T_{N^\circ,\cham}$ for $\cham\in \Delta^+_F(\mathrm{max})$'', where $\Delta^+_F(\mathrm{max})$ denotes the set of maximal cones of $\Delta^+_F$.  
\end{remark}

We are now prepared to describe the special completion of an $\cX$-variety.
By the above discussion, we have a fan whose maximal cones are $\cham \in \Delta^+_F({\rm max}) $, and along with it we have the dual cones $\chamdual$.
This brings to mind the construction of toric varieties via fans, a construction which we review in \S\ref{sec:fan_toric}.

Let $S_{\cham}$ be the monoid of $M^\circ$-points of $\chamdual$.
The monoid algebra $\C\lrb{S_{\cham}}$ 
is a polynomial algebra in $n$ algebraically independent variables.
We write 
\eqn{
\A^{n}_{M^\circ,\cham}:= \Spec\lrp{\C\lrb{ S_{\cham}}} \ \ 
\text{ and } \ \ 
T_{M^\circ,\cham}:= \Spec\lrp{\C\lrb{ S_{\cham}^{\gp}}},
}
where $S_{\cham}^{\gp}$
is the group completion of the monoid $S_{\cham}$.

\begin{definition}\thlabel{special}
The {\it{special completion}} of the $\cX$-cluster variety $\cX_{\Gamma^\vee} = \cA^\vee$ is 
\eqn{
\Xsp_{\Gamma^\vee} := \bigcup_{\cham \in \Delta^+_F(\mathrm{max})}\A^{n}_{M^\circ,\cham},
}
the scheme with the patches $\A^{n}_{M^\circ,\cham}$ glued by the usual $\cX$-mutation formula \eqref{eq:X_mut}.
\end{definition}

\begin{remark}\thlabel{SpecialRemark}
Observe that this definition treats cluster varieties with and without frozen directions together.
Additionally, by starting with the Langlands dual data $\Gamma^\vee$, we obtain the special completion $\Xsp_{\Gamma}$ of $\cX_{\Gamma}$.
\end{remark}

We elaborate on this definition.
We take advantage of \thref{SpecialRemark} to treat the special completion $\Xsp_{\Gamma}$, which allows us to drop the superscript ${}^\circ$.
Explicitly, the birational map 
$\mu_k: \A^n_{M,\cham} \dashrightarrow  \A^n_{M,\cham'} $ 
restricts to an isomorphism
\eqn{\A^n_{M,\cham} \setminus \lrp{ \lrc{X_k=0} \cup \lrc{X_k=-1} }
\stackrel{\cong}{\longrightarrow}
\A^n_{M,\cham'} \setminus \lrp{ \lrc{X'_k=0} \cup \lrc{X'_k=-1} } ,}
and we glue $\A^n_{M,\cham}$ to $\A^n_{M,\cham'}$ along these subsets.
To see that this restriction is indeed an isomorphism, 
we show that $\mu_k^*$ identifies the rings
\[
A_{\cham}:= \C\lrb{X_k^{\pm 1},X_{i\neq k}, \lrp{X_k+1}^{-1}}
\ \ \ \text{and} \ \ \ 
A_{\cham'}:=\C\lrb{{X'_k}^{\pm 1},X'_{i\neq k}, \lrp{X'_k+1}^{-1}}.
\]
Note that it does not matter whether we adjoin $(X_k +1)^{-1}$ or $(X_k^{-1} +1)^{-1}$ as a generator of $A_{\cham}$ since $(X_k^{-1} +1)^{-1} = X_k (X_k+1)^{-1} $ and $X_k^{\pm 1}$ is in $A_{\cham}$.

Applying $\mu_k^*$ to generators of $A_{\cham'}$ gives the polynomial ring
\eqn{
 \C\lrb{X_k^{\pm 1},X_{i\neq k}\lrp{1+X_k^{- \sgn \epsilon_{ik}}}^{-\epsilon_{ik}},\lrp{X_k^{-1} +1}^{-1}}, 
}
and since this ring contains the elements $(1+X_k)^{\pm 1}$ and $(1+X_k^{-1})^{\pm 1}$, we can always cancel the factor $( 1+X_k^{- \sgn \epsilon_{ik}} )^{-\epsilon_{ik}}$, leaving $X_{i\neq k}$.
So
\eqn{
 \C\lrb{X_k^{\pm 1},X_{i\neq k}\lrp{1+X_k^{- \sgn \epsilon_{ik}}}^{-\epsilon_{ik}},\lrp{X_k^{-1} +1}^{-1}} = A_{\cham}, 
}
and $\mu_k$ restricts to an isomorphism of the open subschemes described above.

\begin{remark}\thlabel{nonseparated}
It is well known that $\cX$-cluster varieties are generally not separated (see e.g. \cite[Remark~4.2]{GHK_birational}).
However, with only one exception, special completions of $\cX$-cluster varieties are {\emph{never}} separated.
The one exception is type $A_1$, where we simply obtain $\PP^1$.
In all other cases, 
let $p$ be the point in $\A_{M,{\cham}}^n$ given by $X_k = -1$, $X_{i\neq k} =0$, and $q$ the point in $\A_{M,{\cham'}}^n$ given by $X_k' = -1$, $X_{i\neq k}' =0$.
Note that $p$ and $q$ are not identified since they lie in the complement of the gluing locus.
But, any (complex analytic) neighborhood of one must contain the other.
\end{remark}

\section[Toric degenerations]{Toric degenerations of the special completion of \texorpdfstring{$\cX$}{X}}
\label{sec:toric_degenerations}

In this section we build toric degenerations of specially completed $\cX$-varieties (with and without frozen directions).  
We then discuss properties of both the fibers and the total space.
The toric degenerations we construct here are flat families over the affine space ${\A^n= \Spec\lrp{\C\lrb{t_1,\dots,t_n}}}$ with central fiber a toric variety.
The idea is to take the $\A^n$-patches in the atlas for $\Xsp$ and alter the way they glue together.
So each fiber of the family will have the same collection of $\A^n$-patches, but we introduce coefficients to the mutation formula \eqref{eq:X_mut}, i.e. to the transition functions.
This modified mutation formula is the specialization of \eqref{eq:genralformula} to the case of principal coefficients. 

\begin{remark}
For expository purposes, we restrict to the skew-symmetric case. All of our results hold in the skew-symmetrizable case too. 
\end{remark}

Before getting into this discussion, we briefly recall the construction of toric varieties via fans.  
The discussion here closely follows portions of \cite{FultonTV}, \cite{CLS}, and \cite{Oda}.

\subsection{Fans for toric varieties}\label{sec:fan_toric}

A fan is a combinatorial object that gives an explicit recipe for building a toric variety as a scheme.
It is in essence a pictorial representation of an atlas, where the building blocks (affine schemes) are drawn as cones and gluing is given by intersection of cones.

We start with a lattice $N\cong \Z^n$, the cocharacter lattice of the defining torus $T$.
The fan will live in $N_\R := N\otimes \R$.
We call the fan $\Sigma$, and the associated toric variety $\TV{\Sigma}$.
By definition, $\Sigma$ is a collection of {\it{strongly convex}} {\it{rational polyhedral cones in $N_\R$}} satisfying certain properties, which we detail after explaining the italicized terms.
A {\it{rational polyhedral cone in $N_\R$}} is the $\R_{\geq 0}$-span of a finite number of integral vectors.
It is {\it{strongly convex}} if does not contain a non-zero linear subspace.
Observe that every face of a polyhedral cone is again a polyhedral cone, no matter the codimension of the face.
Now for $\Sigma$ to be a {\it{fan}}, it must satisfy:
\begin{itemize}
    \item If $\sigma \in \Sigma$, and $\tau$ is a face of $\sigma$, then $\tau \in \Sigma$ as well.
    \item If $\sigma, \sigma' \in \Sigma$, then their intersection $\sigma \cap \sigma'$ is a face $\tau$ of each.
\end{itemize}

\begin{remark}
In much of the literature, e.g. \cite{FultonTV} and \cite{CLS}, fans are assumed to consist of a {\emph{finite}} collection of cones.
The resulting toric varieties are then schemes of finite type over $\C$.
Most cluster varieties are only {\emph{locally}} of finite type, so unsurprisingly, we do not require our fans to have finitely many cones.
Toric varieties with infinite fans are treated in \cite{Oda}.
Both, toric varieties and cluster varieties are schemes of finite type precisely when the fan ($\gv$-fan in the cluster case) is finite.
Observe that a cluster variety is of {\emph{cluster}} finite type (meaning there are only finitely many seeds)
precisely when it is a finite type scheme over $\C$.
\end{remark}

Given a fan $\Sigma$ in $N_{\R}$, every cone $\sigma \in \Sigma$ defines an affine scheme in the atlas for $\TV{\Sigma}$ as follows.
Let $M= \Hom\lrp{N,\Z}$ and consider the dual cone $\sigma^\vee\subset M_{\R}$.
Its integral points live in $M$, so are characters of $T$.
These integral points $\sigma^\vee \cap M$ form a submonoid of $M$, denoted $S_\sigma$.
The affine patches in the atlas for $\TV{\Sigma}$ are spectra of the monoid rings $\C\lrb{S_\sigma}$.
We write $U_\sigma:= \Spec\lrp{\C\lrb{\sigma^\vee \cap M}}$.

Every face of $\sigma$ is the intersection of $\sigma$ with some hyperplane $m^\perp$, where $m\in \sigma^\vee \cap M$.
If $\tau$ is a face of $\sigma$, then $\sigma^\vee \subset \tau^\vee$ and $S_\tau = S_\sigma + \Z_{\geq 0}\lrp{-m} $ (see Figure~\ref{fig:cones and dual cones}).
Then $\C\lrb{S_\tau}$ is the localization of $\C\lrb{S_\sigma}$ by the character $z^{m}$. This gives an embedding $U_\tau \hookrightarrow U_\sigma$.

If $\tau$ is a common face of $\sigma$ and $\sigma'$, then the inclusions of $U_\tau\hookrightarrow U_\sigma$ and $U_\tau\hookrightarrow U_{\sigma'}$ glue the affine patches $U_\sigma$ and $U_{\sigma'}$ along $U_\tau$, just as the inclusion of $\tau$ into $\sigma$ and $\sigma'$ glues the cones $\sigma$ and $\sigma'$ together.
More precisely, if $\tau = \sigma \cap \sigma'$ then 
$\sigma\cap m^\perp = \tau = \sigma' \cap m^\perp$
for some $m \in \sigma^\vee \cap -\sigma'^\vee \cap M$,
and we have
\eqn{ U_\sigma \hookleftarrow \lrp{U_\sigma}_{z^m} = U_\tau = \lrp{U_{\sigma'}}_{z^{-m}} \hookrightarrow U_{\sigma'}.}

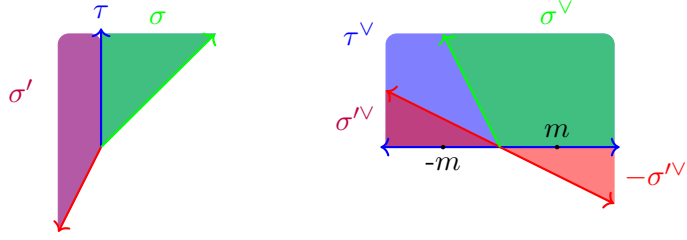
\begin{figure}
    \centering
\begin{tikzpicture}[scale=.75]
\draw[fill, blue, opacity=.5] (0,0) -- (2,2) -- (0,2) -- (0,0);
\draw[fill, green, opacity=.5] (0,0) -- (2,2) -- (0,2) -- (0,0);

\draw[fill, rounded corners, red, opacity=.5] (0,0) -- (0,2) -- (-.75,2) -- (-.75,-1.5) -- (0,0); 
\draw[fill, rounded corners, blue, opacity=.3] (0,0) -- (0,2) -- (-.75,2) -- (-.75,-1.5) -- (0,0); 

\draw[->, thick, blue] (0,0) -- (0,2.1);
\draw[thick, ->, green] (0,0) -- (2,2);
\draw[->,thick,red] (0,0) -- (-.75,-1.5);

\node[above] at (1,2) {\textcolor{green}{$\sigma$}};
\node[left] at (-1,1) {\textcolor{purple}{$\sigma'$}};
\node[above] at (0,2.1) {\textcolor{blue}{$\tau$}};


\begin{scope}[xshift=7cm]
\draw[fill, blue, rounded corners, opacity=.5] (-2,0) -- (-2,2) -- (2,2) -- (2,0);
\draw[fill, red, opacity=.5] (0,0) -- (-2,0) -- (-2,1) -- (0,0);
\draw[fill, green, opacity=.5, rounded corners] (0,0) -- (-1,2) -- (2,2) -- (2,0);

\draw[fill, red, opacity=.5] (0,0) -- (2,-1) -- (2,0);

\draw[<->, thick, blue] (-2.1,0) -- (2.1,0);
\draw[<->, red, thick] (2,-1) -- (-2,1);
\draw[->, thick, green] (0,0) -- (-1,2);

\draw[fill] (1,0) circle [radius=0.03];
\node[above] at (1,0) {$m$}; 
\draw[fill] (-1,0) circle [radius=0.03];
\node[below] at (-1,0) {-$m$}; 
\node[right] at (2,-.5) {\textcolor{red}{$-\sigma'^\vee$}};
\node[above] at (1,2) {\textcolor{green}{$\sigma^\vee$}};
\node[above, left] at (-2,2) {\textcolor{blue}{$\tau^\vee$}};
\node[left] at (-2,.5) {\textcolor{purple}{$\sigma'^\vee$}};

\end{scope}
\end{tikzpicture}

    \caption{Two adjacent cones $\sigma,\sigma'\subset N_{\mathbb R}$ with $\sigma\cap\sigma'=\tau$ on the left and their dual cones $\sigma^\vee,\sigma'^\vee\subset\tau^\vee\subset M_{\mathbb R}$ on the right.}
    \label{fig:cones and dual cones}
\end{figure}

\subsection{Toric degenerations of $\Xsp$}\label{sec:toric_degen_x}

Consider fixed data $\Gamma_{\text{prin}}$ as in \thref{def:Gamma_prin} and fix an initial $\gv$-cone $\cham_0$. 
Let $F\subset I$ be the set of frozen directions. 
In view of \thref{dictionary_2} and \thref{rem:ConesSeeds} set $t_i := p_i = p_{i;\cham_0}$ and $ X_i= X_{i;\cham_0} :=y_i $ for all $1 \leq i \leq n$.
Then the mutation formula \eqref{eq:X_mut_w_coef} specializes for $\cX$-variables with principal coefficients to 
\begin{align}\label{eq:mutfamily}
\mu_k^*\lrp{\Xt_{i;\cham'}} = \left\{\begin{matrix} \Xt_{i;\cham}^{-1} &\text{ if } i=k,\\
\Xt_{i;\cham}\lrp{\bt^{[\sgn(\epsilon_{ik})\bc_{k;\cham}]_+} + \bt^{[-\sgn(\epsilon_{ik})\bc_{k;\cham}]_+}\Xt_{k;\cham}^{-\sgn(\epsilon_{ik})}}^{-\epsilon_{ik}} &\text{ if } i\not =k.\end{matrix}\right.
\end{align}
We can use sign coherence of $\cv$-vectors (\thref{c_sing-coherence}) to rewrite the $i\neq k$ case.
If the sign of $\bc_{k;\cham}$ coincides with $\sgn(\epsilon_{ik})$, then for $i\not =k$ \eqref{eq:mutfamily} simplifies to
\begin{eqnarray}\label{eq:same sign}
\mu_k^*\lrp{\Xt_{i;\cham'}}=\Xt_{i;\cham}\lrp{\bt^{|\bc_{k;\cham}|}+\Xt_{k;\cham}^{-\sgn(\epsilon_{ik})}}^{-\epsilon_{ik}}.
\end{eqnarray}
If $\bc_k$ and $\epsilon_{ik}$ have opposite signs we instead obtain 
\begin{eqnarray}\label{eq:opp sign}
\mu_k^*\lrp{\Xt_{i;\cham'}}=\Xt_{i;\cham}\lrp{1+\bt^{|\bc_{k;\cham}|}\Xt_{k;\cham}^{-\sgn(\epsilon_{ik})}}^{-\epsilon_{ik}}.
\end{eqnarray}
Here $|\bc_{k;\cham}|$ denotes the positive $\cv$-vector, i.e. $|\bc_{k;\cham}|=\lrp{|c_{k1;\cham}|,\dots, |c_{kn;\cham}|}\in \Z_{\ge 0}^n$. 
An example in type $A_2$ is given in Table~\ref{tab:A2_principal coeff}.

\begin{table}[!htbp]
    \centering
    \begin{tabular}{|c|c|c|c|c|c|c|}
    \hline
         $s$ & $\boldsymbol{\varepsilon}_s$ & $\mathbf C_s$ & $t_{1;s}$ &$t_{2;s}$ & $\widetilde{X}_{1;s}$ & $\widetilde{X}_{2;s}$ \\
         \hline \hline
          
          0  & $\left(\begin{smallmatrix} 0& -1 \\ 1 &0\end{smallmatrix}\right)$ & 
          $\left(\begin{smallmatrix} 1 & 0 \\ 0 & 1\end{smallmatrix}\right)$ &
          $t_1$& 
          $ t_2$&
          $X_1$ & 
          $X_2$  \\  \hline
          $\updownarrow \mu_2$ \\ \hline
          1  & $\left(\begin{smallmatrix} 0& 1 \\ -1 &0\end{smallmatrix}\right)$ & 
          $\left(\begin{smallmatrix} 1 & 0 \\ 0 & -1\end{smallmatrix}\right)$ &
          $t_1$  &
          $\frac{1}{t_2} $&
          $X_1(t_2X_2+1)$ & 
          $\frac{1}{X_2}$ \\ \hline
          $\updownarrow \mu_1$ \\ \hline
          2  & $\left(\begin{smallmatrix} 0& -1 \\ 1 &0\end{smallmatrix}\right)$ & 
          $\left(\begin{smallmatrix} -1 & 0 \\ 0 & -1\end{smallmatrix}\right)$ &
          $\frac{1}{t_1}$&
          $\frac{1}{t_2}$ &
          $\frac{1}{X_1(t_2X_2+1)}$ &
          $\frac{t_1t_2X_1X_2+t_1X_1+1}{X_2} $  \\ \hline
          $\updownarrow \mu_2$ \\ \hline
          3  & $\left(\begin{smallmatrix} 0& 1 \\ -1 &0\end{smallmatrix}\right)$ & 
          $\left(\begin{smallmatrix} -1 & 0 \\ -1 & 1\end{smallmatrix}\right)$ &
          $\frac{1}{t_1 t_2}$&
          $t_2 $&
          $\frac{t_1X_1+1}{X_1X_2}$ &
          $\frac{X_2}{t_1t_2X_1X_2+t_1X_1+1}$ \\ \hline
          $\updownarrow \mu_1$ \\ \hline
          4  & $\left(\begin{smallmatrix} 0& -1 \\ 1 &0\end{smallmatrix}\right)$ & 
          $\left(\begin{smallmatrix} 1 & -1 \\ 1 & 0\end{smallmatrix}\right)$ &
          $t_1t_2$&
          $\frac{1}{t_1} $ &
          $\frac{X_1X_2}{t_1X_1+1}$ &
          $\frac{1}{X_1}$  \\ \hline
          $\updownarrow \mu_2$ \\ \hline
          5 & $\left(\begin{smallmatrix} 0& 1 \\ -1 &0\end{smallmatrix}\right)$ & 
          $\left(\begin{smallmatrix} 0 & 1 \\ 1 & 0\end{smallmatrix}\right)$ &
          $t_2$ &
          $t_1$ &
          $X_2$ & 
          $X_1$ \\ \hline
    \end{tabular}
    \caption{Example in type $A_2$. Our convention agrees with \cite[Tables 1-4]{FZ_clustersIV}, note only that we are considering geometric notation and hence the matrices $\boldsymbol{\varepsilon}_s=B_s^T$.}
    \label{tab:A2_principal coeff}
\end{table}

In order to talk about degrees of $\cX$-variables with principal coefficients we consider the usual $\Z^n_{\ge 0}$-grading on the polynomial ring in $n$ variables $X_{1},\dots,X_{n}$.
We extend this notion to a $\Z^n$-grading on the ring of rational functions $\C(X_{1},\dots,X_{n})$. For $h=\frac{f}{g}$ with $f$ and $g$ non-zero homogeneous polynomials let $\deg(h):=\deg(f)-\deg(g)$.
We call this the \emph{standard grading}. 
In the case of $\cAp$ by \thref{g-vectors}, \thref{dictionary p-map} and \thref{tilde_extension} we see that $-\deg(t_i)=b_i=\deg(\widetilde{p^*\lrp{X_{i}}})$.
Motivated by the desire to construct a family fiber-wise dual to $\cAp$, we therefore set
\[
\deg\lrp{t_i} := -\deg(\Xt_{i}).
\]
Note that for $R=\C[t_1,\dots,t_n]$ this gives us the notion of the standard $\Z^n$-grading also on $R(\Xt_{1},\dots,\Xt_{n})$.

\begin{proposition}\thlabel{c-deg}
The pull-back to $s_0$ of an $\cX$-variable with principal coefficients 
\eqn{\mu_{s_0}^*\lrp{\Xt_{i;\cham}}\in R(\Xt_{1},\dots,\Xt_{n})} 
is homogeneous with respect to the standard $\Z^n$-grading. 
Its degree is the $\bc$-vector of $X_{i;\cham}$ with respect to $s_0$:
\[
\deg\lrp{\mu_{s_0}^*\lrp{\Xt_{i;\cham}}}= \bc_{i;\cham}.
\]
\end{proposition}

\begin{proof}
This follows from \thref{separation} by specializing to principal coefficients.
We have
\[
p_{i;\cham}=\prod_{j=1}^n p_j^{c_{ji;\cham}} = \bt^{\bc_{i;\cham}}. 
\]
Consider the expression $Y_{i;\cham}\lrp{p_1 y_1,\dots,p_n y_n}$ in the numerator of \eqref{general_separation_formula}, with ${y_i =X_i}$ and $p_i=t_i$.
Since by definition we have $\deg(t_i X_{i})=0$ for all $i$, we deduce that ${\deg(Y_{i;\cham}\lrp{t_1X_1,\dots,t_nX_n})=0}$.
Then as by \eqref{general_separation_formula}
\eqn{
    \mu_{s_0}^*\lrp{\Xt_{i;\cham}}=\frac{Y_{i;\cham}(t_1 X_1,\dots,t_n X_n)}{\bt^{\bc_{i;\cham}}},
}
it follows that $\deg(\mu_{s_0}^*(\Xt_{i;\cham}))= -\deg\lrp{\bt^{\bc_{i;\cham}}} = \bc_{i;\cham} $.
\end{proof}

On account of \thref{c-deg}, it is natural to ask whether there is any special relationship between $\mu_{s_0}^*(\Xt_{i;\cham})$ and
$\vb{\Xt}^{\cv_{i;\cham}} := {\Xt_{1}}^{c_{i_1;\cham}}\cdots {\Xt_{n}}^{c_{i_n;\cham}}$. 
The answer is {\it{yes}}, and the following proposition makes this relationship precise.

\begin{proposition}
\thlabel{tto0}
The pull-back $\mu_{s_0}^*(\Xt_{i;\cham})$ of an $\cX$-variable with principal coefficients satisfies
\eq{
\lim_{\vb{t}\to 0} \mu_{s_0}^*\lrp{\Xt_{i;\cham}} =\vb{\Xt}^{\cv_{i;\cham}}.
}{eq:tto0}
Moreover, $\mu_{s_0}^*(\Xt_{i;\cham})$ is the unique homogenization of degree $\cv_{i;\cham}$ of the coefficient-free expression  $\mu_{s_0}^*(\Xt_{i;\cham})|_{\vb{t} = \vb{1} }$ satisfying \eqref{eq:tto0}.
\end{proposition}

\begin{proof}
For $\cham=\cham_0$, \eqref{eq:tto0} is clear.
Suppose now that it holds for some arbitrary $\gv$-cone $\cham$.
We show that it must hold for all $\gv$-cones $\cham'$ sharing a facet with $\cham$ as well.
Take $\cham'$ to be related to $\cham$ by mutation in direction $k$.
Using \eqref{eq:mutfamily}, we can express
$\Xt_{i;\cham'}$ in terms of $\Xt_{1;\cham},\dots, \Xt_{n;\cham}$, and then take the $\vb{t}\to 0$ limit.
First note that \eqref{eq:Y-pattern_mutation} gives us the following mutation formula for $\cv$-vectors:
\eq{\cv_i'= \begin{cases} 
-\cv_i & \text{ if } i=k\\
\cv_i + \sgn\lrp{\epsilon_{ik}} \lrb{\epsilon_{ik} \cv_k}_+ & \text{ if } i\neq k.
\end{cases} }{eq:mutcv}
For $i=k$ we have
\eqn{
\lim_{\vb{t}\to 0} \mu_{s_0}^*\lrp{\Xt_{k;\cham'}}
&=\lim_{\vb{t}\to 0} \mu_{s_0}^*\lrp{\mu_k^*\lrp{\Xt_{k;\cham'}}}
=\lim_{\vb{t}\to 0} \mu_{s_0}^*\lrp{\Xt_{k;\cham}^{-1}}
=\vb{\Xt}^{-\cv_{k;\cham}}
=\vb{\Xt}^{\cv_{k;\cham'}}.
}
Now take $i\neq k$. 
First suppose $\epsilon_{ik}$ and $\cv_{k;\cham}$ have the same sign.
Then
\eqn{
\lim_{\vb{t}\to 0} \mu_{s_0}^*\lrp{\Xt_{i;\cham'}}
&=\lim_{\vb{t}\to 0} \mu_{s_0}^*\lrp{\mu_k^*\lrp{\Xt_{i;\cham'}}}\\
&=\lim_{\vb{t}\to 0} \mu_{s_0}^*\lrp{ 
\Xt_{i;\cham}\lrp{\vb{t}^{\lrm{\cv_{k;\cham}}}+\Xt_{k;\cham}^{-\sgn(\epsilon_{ik})}}^{-\epsilon_{ik}}
}
= \vb{\Xt}^{\cv_{i;\cham} + \lrm{\epsilon_{ik}} \cv_{k;\cham} } = \vb{\Xt}^{\cv_{i;\cham'}}.
}
If $\epsilon_{ik}$ and $\cv_{k;\cham}$ have opposite signs, we have
\eqn{
\lim_{\vb{t}\to 0} \mu_{s_0}^*\lrp{\Xt_{i;\cham'}}
&=\lim_{\vb{t}\to 0} \mu_{s_0}^*\lrp{\mu_k^*\lrp{\Xt_{i;\cham'}}}\\
&=\lim_{\vb{t}\to 0} \mu_{s_0}^*\lrp{ 
\Xt_{i;\cham}\lrp{1+\vb{t}^{\lrm{\cv_{k;\cham}}}\Xt_{k;\cham}^{-\sgn(\epsilon_{ik})}}^{-\epsilon_{ik}}
}
= \vb{\Xt}^{\cv_{i;\cham} } 
= \vb{\Xt}^{\cv_{i;\cham'}}.
}
This establishes \eqref{eq:tto0}.

We move on to uniqueness.
The coefficient-free expression $\mu_{s_0}^*(\Xt_{i;\cham})|_{\vb{t} = \vb{1} }$
is a rational function in $\C(\Xt_1, \dots,\Xt_n)$.
Write
\eqn{
\left.\mu_{s_0}^*\lrp{\Xt_{i;\cham}}\right|_{\vb{t} = \vb{1} } = \text{\small{ $  \frac{f(\Xt_1, \dots,\Xt_n)}{g(\Xt_1, \dots,\Xt_n)}=:\frac{f(\vb{\Xt})}{g(\vb{\Xt})}.$ }}
}
Assume there are two homogenizations of $\mu_{s_0}^*(\Xt_{i;\cham})|_{\vb{t} = \vb{1} }$ such that \eqref{eq:tto0} holds. 
The key observation is that if \eqref{eq:tto0} holds for some homogenization,
then for this homogenization there {\emph{must}} be exactly one monomial in each $f$ and $g$ that does not pick up a (non-trivial) $\bt^{\vb{n}}$ coefficient\footnote{Assuming of course that the homogenizations of $f$ and $g$ are relatively prime.}.
The result follows immediately from here.

For the observation, first suppose 
$\frac{f(\vb{\Xt},\bt)}{g(\vb{\Xt},\bt)}\in R(\Xt_1,\dots,\Xt_n)$
is some arbitrary rational function in $\Xt_i$'s and $t_i$'s.
Given a path $\gamma: \lrb{0,1} \to \C^n$ with $\gamma(1) = 0$, the limit
\eqn{
\lim_{\bt \stackrel{\gamma}{\to} 0}\ \text{\small{$ \frac{f(\vb{\Xt},\bt)}{g(\vb{\Xt},\bt)}$}}
} 
picks out the summands of $f$ and $g$ that vanish most slowly along $\gamma$.
The result, if it exists, is a new rational function whose numerator is a partial sum coming from $f$ and denominator is a partial sum coming from $g$.
If this limit is a non-zero Laurent monomial, then the partial sums contain exactly one term.
We are interested in the case where this limit is the monomial ${\vb{\Xt}}^{\bc_{i;\cham}}$ for every such $\gamma$.
Take $\gamma$ and $\gamma'$ with 
\eqn{
\lim_{\bt \stackrel{\gamma, \gamma'}{\to} 0}\ \text{\small{ $ \frac{f(\vb{\Xt},\bt)}{g(\vb{\Xt},\bt)} $ }} = \vb{\Xt}^{\bc_{i;\cham}}.
} 
Associated to $\gamma$ we have the summands ${\vb{\Xt}}^{{\vb{n}_f}}$ of $f$ and  ${\vb{\Xt}}^{{\vb{n}_g}}$ of $g$, with ${\vb{n}_f}$, ${\vb{n}_g}$ $\in \Z_{\geq 0}^n$ and
\eqn{
\text{\small{ $ \frac{{\vb{\Xt}}^{{\vb{n}_f}}}{{\vb{\Xt}}^{{\vb{n}_g}}}$ }} = {\vb{\Xt}}^{\bc_{i;\cham}}.
}
We similarly associate ${\vb{\Xt}}^{\vb{n}_f'}$ and ${\vb{\Xt}}^{\vb{n}_g'}$ to $\gamma'$.
There are two cases to consider: either ${\vb{\Xt}}^{\vb{n}_f}$ divides ${\vb{\Xt}}^{\vb{n}_f'}$, respectively vice versa, or neither divides the other. 
Suppose ${\vb{\Xt}}^{\vb{n}_f}$ divides ${\vb{\Xt}}^{\vb{n}_f'}$.
Then to homogenize $f$, we have to multiply ${\vb{\Xt}}^{\vb{n}_f'}$ by the coefficient  $\bt^{\vb{n}_f'-\vb{n}_f}$.
In this case, assuming ${\vb{n}_f'-\vb{n}_f \neq 0}$,
there cannot exist a path along which the summand ${\vb{\Xt}}^{\vb{n}_f}$ vanishes more quickly than the summand ${\vb{\Xt}}^{\vb{n}_f'}$.
So we must have ${\vb{\Xt}}^{\vb{n}_f}={\vb{\Xt}}^{\vb{n}_f'}$, and hence ${\vb{\Xt}}^{\vb{n}_g}={\vb{\Xt}}^{\vb{n}_g'}$ as well.

On the other hand, suppose neither ${\vb{\Xt}}^{\vb{n}_f}$ nor ${\vb{\Xt}}^{\vb{n}_f'}$ divides the other.
Let $\vb{n}_f = \lrp{n_{f;1}, \dots, n_{f;n}}$
and $\vb{n}_f' = (n_{f;1}', \dots, n_{f;n}')$.
Set $\vb{d} = \sum_i [n_{f;i} - n_{f;i}']_+ e_i$
and $\vb{d}' = \sum_i [n_{f;i}' - n_{f;i}]_+ e_i$.
Both are non-zero.
In homogenizing $f$, ${\vb{\Xt}}^{\vb{n}_f}$ must obtain a coefficient of (at least) $\bt^{\vb{d}}$ and 
${\vb{\Xt}}^{\vb{n}_f'}$ a coefficient of $\bt^{\vb{d}'}$.   
Any additional factors of $\bt$ the two summands obtain must match.
Therefore, ${\vb{\Xt}}^{\vb{n}_g}$
and 
${\vb{\Xt}}^{\vb{n}_g'}$
must obtain these coefficients as well.
Along any path $\gamma''$ contained in the hyperplane $\bt^{\vb{d}} = \bt^{\vb{d}'}$,
we will have 
\eqn{
\lim_{\bt \stackrel{\gamma''}{\to} 0}\  \text{\small{ $  \frac{ \bt^{\vb{d}} {\vb{\Xt}}^{\vb{n}_f} + \bt^{\vb{d}'} {\vb{\Xt}}^{\vb{n}_f'}}{\bt^{\vb{d}} {\vb{\Xt}}^{\vb{n}_g} + \bt^{\vb{d}'} {\vb{\Xt}}^{\vb{n}_g'}} = \frac{ {\vb{\Xt}}^{\vb{n}_f} + {\vb{\Xt}}^{\vb{n}_f'}}{{\vb{\Xt}}^{\vb{n}_g} +  {\vb{\Xt}}^{\vb{n}_g'}}.$ }}
} 
Perhaps after considering additional summands (in the case that some other summand vanishes more quickly along $\gamma''$),
we find that the limit along some path is not a Laurent monomial.
This contradicts our assumption and proves the claim.
\end{proof}

\begin{remark}\thlabel{homomon}
The above proposition tells us we have a canonical homogenization of cluster variables.
Note that this in turn gives a homogenization of all cluster monomials.
\end{remark}

We are now prepared to describe toric degenerations of $\Xsp$.

\begin{definition} \thlabel{def:Xfam}
We define $\Xfsps{\cham_0}$ to be the scheme over $R$ with affine patches 
\eqn{U_\cham:=\Spec\lrp{R \lrb{\Xt_{1;\cham},\dots,\Xt_{n;\cham}}}} 
and gluing given by \eqref{eq:mutfamily}.
By analogy, we denote by $\Xfams{\cham_0}$ the scheme obtained by instead gluing $U_\cham^{\pm}:=\Spec( R [\Xt_{1;\cham}{}^{\pm 1},\dots,\Xt_{n;\cham}{}^{\pm 1}]) $.
Recall from \thref{rem:ConesSeeds} that the schemes $\Xfams{\cham_0}$ and $\Xfam_{s_0}$ are isomorphic.
\end{definition}

The open subschemes along which $U_\cham$ and $U_{\cham'}$ are glued have a very similar description to those discussed below \thref{special}.
To simplify notation, let $\bc_k=\bc_{k;\cham}$ and $\bc_{k}'=\bc_{k;\cham'}$ for two adjacent $\gv$-cones $\cham$ and $\cham'$ sharing a facet contained in $\cv_k^\perp$.
On the level of rings we set
\eqn{
A_\cham = R\lrb{\Xt_{k;\cham}^{\pm 1}, \Xt_{i;\cham}, \lrp{ \bt^{\lrb{\sgn\lrp{\epsilon_{ik}} \cv_k}_+} + \bt^{\lrb{-\sgn\lrp{\epsilon_{ik}} \cv_k}_+ }\Xt_{k;\cham}^{-\sgn\lrp{\epsilon_{ik}} }}^{-1} }_{i\neq k}.
}
Note that 
$\sgn(\epsilon_{ik}) \cv_k = \sgn(\epsilon'_{ik}) \cv_k'$ 
and
$ \mu_k^*( \Xt_{k;\cham'}^{ -\sgn (\epsilon'_{ik})  } )= \Xt_{k;\cham}^{ -\sgn \lrp{\epsilon_{ik}}  }  .$
Then applying $\mu_k^*$ to generators of $A_{\cham'}$ yields
\eq{\mu_k^*\lrp{ \Xt_{k;\cham'}^{\mp 1} }= \Xt_{k;\cham}^{\pm 1},}{eq:muXk}
\eq{\mu_k^*\lrp{\Xt_{i;\cham'}} = \Xt_{i;\cham} \lrp{ \bt^{\lrb{\sgn\lrp{\epsilon_{ik}} \cv_k}_+} + \bt^{\lrb{-\sgn\lrp{\epsilon_{ik}} \cv_k}_+ }\Xt_{k;\cham}^{-\sgn\lrp{\epsilon_{ik}} }}^{-\epsilon_{ik}},}{eq:muXi}
\begin{eqnarray}\label{eq:muStuff}
    &\mu_k^*\lrp{ \bt^{[\sgn(\epsilon'_{ik}) \cv_k']_+} + \bt^{[-\sgn(\epsilon'_{ik}) \cv_k']_+ }\Xt_{k;\cham'}^{-\sgn(\epsilon'_{ik}) }}^{-1}&     \\
\nonumber    = &\lrp{ \bt^{\lrb{\sgn\lrp{\epsilon_{ik}} \cv_k}_+} + \bt^{\lrb{-\sgn\lrp{\epsilon_{ik}} \cv_k}_+ }\Xt_{k;\cham}^{-\sgn\lrp{\epsilon_{ik}} }}^{-1}.& 
\end{eqnarray}
Only \eqref{eq:muXi} needs further consideration.
If $-\epsilon_{ik} > 0 $, we can use \eqref{eq:muStuff} to recover $\Xt_{i;\cham}$ from \eqref{eq:muXi}.
Meanwhile, if $-\epsilon_{ik} < 0$, the term within parenthesis is in $R[X_{k;\cham}^{\pm 1}]$,
so we can use \eqref{eq:muXk} to recover \eqref{eq:muXi}.
On the other hand, the expressions on the right hand side of \eqref{eq:muXk}, \eqref{eq:muXi} 
and \eqref{eq:muStuff} are all in $A_\cham$, so $\mu_k^*$ defines an isomorphism of $A_\cham$ and $A_{\cham'}$.
Geometrically, we are gluing the subschemes
\eq{
U_\cham \setminus \lrp{ \lrc{\Xt_{k;\cham} = 0 } \bigcup \lrc{ \bt^{\lrb{\sgn\lrp{\epsilon_{ik}} \cv_k}_+} + \bt^{\lrb{-\sgn\lrp{\epsilon_{ik}} \cv_k}_+ }\Xt_{k;\cham}^{-\sgn\lrp{\epsilon_{ik}} } = 0 }_{i\neq k} },
}{eq:FamilyOverlap}
\eq{
U_{\cham'} \setminus \lrp{ \lrc{\Xt_{k;\cham'} = 0 } \bigcup \lrc{ \bt^{[\sgn(\epsilon'_{ik}) \cv'_k]_+} + \bt^{[-\sgn(\epsilon'_{ik}) \cv'_k]_+ }\Xt_{k;\cham'}^{-\sgn(\epsilon'_{ik}) } = 0 }_{i\neq k}}.
}{eq:FamilyOverlapPrime}

To see that \thref{def:Xfam} makes sense, we also need to verify that the cocycle condition holds.
But this follows immediately from \thref{periodicity}.
Now observe that the fiber over 1 is precisely $\Xsp$, and the restriction of $\Xt_{i;\cham}$ to this fiber is $X_{i;\cham}$.

\begin{proposition}\thlabel{flat}
The family $\pi:\Xfsps{\cham_0} \to \Spec\lrp{R}$ is a flat family.
\end{proposition}

\begin{proof}
This follows from the definition of $\Xfsps{\cham_0}$.
Flatness is a local property, so we are free to restrict to affine patches.
The induced map of rings $R \to R[\Xt_{1;\cham}, \dots, \Xt_{n;\cham}]$ is flat for every $\cham$, so $\Xfsps{\cham_0} \to \Spec\lrp{R}$ is a flat morphism of schemes.
\end{proof}

\begin{theorem}\thlabel{toric_degen}
The flat family $\Xfsps{\cham_0} \to \Spec\lrp{R}$ is a toric degeneration of $\Xsp$ to $\TV{\ccF}$.
\end{theorem}

\begin{proof}
We have already noted that the fiber over $1$ is $\Xsp$ and shown flatness. 
We only need to check that the central fiber $\Xsp_0$ is indeed $\TV{\ccF}$.
Clearly, we can identify the affine patches of $\Xsp_0$  associated to maximal cones of $\ccF$ with
those of $\TV{\ccF}$; in both cases we have a copy of $\A^n$ for every maximal cone of $\ccF$.
We just need to demonstrate that the {\emph{gluing}} of these patches in $\Xsp_0$ coincides with the gluing in $\TV{\ccF}$.
This amounts to verifying two things:
\begin{enumerate}[(i)]
    \item \label{overlap} the overlap of neighboring affine patches in $\Xsp_0$ and $\TV{\ccF}$ agree;
    \item \label{functions} the transition functions on these overlaps agree.
\end{enumerate}

For $\TV{\ccF}$, $S_{\cham}$ is the monoid of integral points in the dual cone $\chamdual$.
The non-identity generators of this monoid are precisely $\bc_{1;\cham},\dots,\bc_{n;\cham} $, 
with $\cv_{k;\cham}$ normal to the facet of $\cham$ spanned by 
$\lrc{\gv_{i;\cham}}_{i\neq k}$.
If $\cham'$ shares this facet with $\cham$, 
then in $\TV{\ccF}$ the subscheme of $\A^n_{\cham}$ overlapping $\A^n_{\cham'}$ is the localization $\lrp{\A^n_{\cham}}_{\Xt_{k;\cham}} = \A^n_\cham \setminus \{ \Xt_{k;\cham} = 0 \}$.

To establish (\ref{overlap}), we need to see that specializing $\bt=0$ in \eqref{eq:FamilyOverlap}
results in precisely the same subset of $\A^n_\cham$ being removed.
The first piece removed in \eqref{eq:FamilyOverlap} is exactly $\{ \Xt_{k;\cham} = 0 \}$, 
so we have to show that nothing else is removed when $\bt = 0$.
When $\epsilon_{ik}$ and $\cv_k$ have the same sign,
\eqn{\bt^{\lrb{\sgn\lrp{\epsilon_{ik}}\bc_k}_+} + \bt^{\lrb{-\sgn\lrp{\epsilon_{ik}}\bc_k}_+}\Xt_{k;\cham}^{-\sgn\lrp{\epsilon_{ik}}}
=
\bt^{\lrm{\bc_k}}+\Xt_{k;\cham}^{-\sgn\lrp{\epsilon_{ik}}}.
}
For $\bt=0$ we eliminate the locus $\{\Xt_{k;\cham}^{-\sgn(\epsilon_{ik})} = 0 \}$.
If the exponent is positive, we have already removed this.
If it is negative, the locus is empty.
On the other hand, when $\epsilon_{ik}$ and $\cv_k$ have opposite signs,
\eqn{\bt^{\lrb{\sgn\lrp{\epsilon_{ik}}\bc_k}_+} + \bt^{\lrb{-\sgn\lrp{\epsilon_{ik}}\bc_k}_+}\Xt_{k;\cham}^{-\sgn\lrp{\epsilon_{ik}}}
=
1+ \bt^{\lrm{\bc_k}}\Xt_{k;\cham}^{-\sgn\lrp{\epsilon_{ik}}}.
}
When $\bt = 0 $ in this case, we eliminate the locus $\lrc{1=0}$, which is obviously empty.
This establishes (\ref{overlap}).

For \eqref{functions}, note that elements of $S_{\mathcal G}$ correspond to functions on $\TV{\ccF}$ that restrict to characters on the defining torus of $\TV{\ccF}$.
Specifically, they are the characters whose exponent vectors are integral points of $\mathcal C$.
On intersections, the relations between elements of $S_{\mathcal G}$ and $S_{\mathcal G'}$ are precisely the relations between their exponent vectors, living in $\chamdual$ and $\chamdual'$ respectively.
The first part of \thref{tto0} and \thref{homomon} show that the $\cX$-monomials on the patches $\A^n_\cham$ and $\A^n_{\cham'}$ of $\Xsp_0$ are related in the same way.
\end{proof}

\begin{remark}\thlabel{refine}
The key idea here is that by restricting to a domain of linearity in $\cX^\trop\lrp{\R}$ 
we find ourselves in the cozy world of toric varieties,
where relations between generators are as straightforward as vector addition.
When we cross a wall of the scattering diagram, 
we relate generators by mutation instead.
As we approach the central fiber, mutation relations revert back to the toric version where we simply add exponent vectors.
It is not necessary to insist that the patches we glue by mutation be copies of $\A^n$.
We can repeat the construction with different compactifications of the cluster tori.
If we consider a refinement $\Sigma$ of the $\gv$-fan, we can take our ``patches'' to be the toric varieties associated to subfans whose support is a domain of linearity.
This gives a different partial minimal model $\Xsp_\Sigma$ of $\cX$.
Mutation gives the gluing of these patches as before,
and the family is a toric degeneration of $\Xsp_\Sigma$ to $\TV{\Sigma}$. 
We will see examples of this in a sequel with M.-W. Cheung, with such compactifications and their toric degenerations arising as duals to compactified $\cA$-varieties and the $\cAp$-family.
\end{remark}

We can use the family $\Xfsps{\cham_0}$ to give a new notion of $\cv$-vector analogously to the definition of $\gv$-vectors given in \cite[Definition~5.6]{GHKK}.

\begin{definition}\thlabel{def:cvNew}
Let $f\in R(X_1,\dots,X_n)$ be a homogeneous rational function on $\Xfsps{\cham_0}$.
We define the {\it{$\cv$-vector of $f$ at $\cham_0$}} (or {\it{at $s_0$}}) by $\cv_{\cham_0}(f)=\cv_{s_0}(f) := \deg(f)$ and the {\it{$\cv$-vector at $\cham_0$ of a cluster monomial on $\cX$}} as the $\cv$-vector at $\cham_0$ of its extension to $\Xfsps{\cham_0}$. 
\end{definition}

We now turn our attention to the fibers of the family.


\begin{proposition}\thlabel{iso of fibres}
For $\mathbf{u},\mathbf{u}'\in (\mathbb C^*)^n$ the fibers $\Xsp_{\mathbf u}$ and $\Xsp_{\mathbf u'}$ of the family ${\Xfsps{\cham_0} \to \Spec\lrp{R}}$ are isomorphic.
\end{proposition}

\begin{proof}
For a given $\gv$-cone $\cham$ let $\left.U_\cham\right|_{\mathbf t=\mathbf u}$ denote by a little abuse of notation the affine patch $\Spec(\mathbb C[\Xt_{1;\cham},\dots,\Xt_{n;\cham}])$ of $\Xsp_{\mathbf u}$. 
Similarly, let $\left.U_{\cham}\right|_{\mathbf t=\mathbf u'}$ be  $\Spec(\mathbb C[\Xt'_{1;\cham},\dots,\Xt'_{n;\cham}])$ of $\Xsp_{\mathbf u'}$.
We define an isomorphism $\psi_\cham:U_{\cham}\vert_{\mathbf t=\mathbf u'}\to U_{\cham}\vert_{\mathbf t=\mathbf u}$ by the pullback on functions:
\begin{eqnarray*}
\psi^*_\cham: \mathbb C[\Xt_{1;\cham},\dots,\Xt_{n;\cham}] \to \mathbb C[\Xt'_{1;\cham},\dots,\Xt'_{n;\cham}],\ \ \text{ where } \ \
\Xt_{i;\cham} \mapsto \frac{\mathbf u'^{\cv_{i;\cham}}}{\mathbf u^{\cv_{i;\cham}}}\Xt'_{i;\cham} \ \forall i.
\end{eqnarray*}
Let $\cham'$ be the $\gv$-cone related to $\cham$ by mutation in direction $k$, i.e. sharing the facet in $\cv_k^\perp$. 
We have to show that this isomorphism commutes with the gluing given by $\mu_k$.
The transition functions on intersections of the affine patches in $\Xsp_{\mathbf u}$ are obtained by specializing $\mathbf t=\mathbf u$ in \eqref{eq:mutfamily} and similarly for $\Xsp_{\mathbf u'}$.
We show commutativity on the level of rings, i.e. $\psi_{\cham}^*\circ \mu_k^*(\Xt_{i;\cham'})=\mu_k^*\circ \psi^*_{\cham'}(\Xt_{i;\cham'})$ for every $i\in[1,n]$. 
For $i=k$ by \eqref{eq:mutfamily} and \eqref{eq:mutcv} we have 
\begin{eqnarray*}
{\psi_\cham^*}\circ \mu_k^*\lrp{{\Xt}_{k;\cham'}}
&=&
\psi_\cham^*\lrp{{\Xt}{}^{-1}_{k;\cham}} = \frac{\mathbf u^{\mathbf c_{k;\cham}}}{{\mathbf u'}{}^{\mathbf c_{k;\cham}}}{{\Xt}'{}^{-1}_{k;\cham}}\\
&=& {\frac{{\mathbf u'}^{\mathbf c_{k;\cham'}}}{\mathbf u^{\mathbf c_{k;\cham'}}}}{{\Xt}'{}^{-1}_{k;\cham}} = \mu_k^*\lrp{ \frac{{\mathbf u'}^{\mathbf c_{k;\cham'}}}{\mathbf u^{\mathbf c_{k;\cham'}}}{\Xt}'_{k;\cham'}}={\mu_k^*}\circ \psi^*_{\cham'}\lrp{\Xt_{k;\cham'}}.
\end{eqnarray*}
For $i\not =k$ we distinguish two cases:
\begin{itemize}
    \item[\bf Case 1:] If $\sgn(\epsilon_{ik})\not=\sgn(\mathbf c_{k;\cham})$, note that $\cv_{i;\cham'}=\cv_{i;\cham}+\sgn(\epsilon_{ik})[\epsilon_{ik}\cv_{k;\cham}]_+=\cv_{i;\cham}$, then
\begin{eqnarray*}
     \psi_\cham^{*} \circ \mu_k^{*} 
        \lrp{  
            \Xt_{i;\cham'} 
        } 
    &=& \psi_\cham^{*} 
        \lrp{
            \Xt_{i;\cham} 
            \lrp{
                1+\mathbf{u}^{|\cv_{k;\cham}|}\Xt_{k;\cham}^{-\sgn(\epsilon_{ik})}}^{-\epsilon_{ik}
            } 
        } \\
    &=& \frac{ {\mathbf{u}'}^{\cv_{i;\cham}} }{ \mathbf{u}^{\cv_{i;\cham}} }
        \Xt'_{i;\cham}
        \lrp{
            1+\mathbf{u}^{| \cv_{k;\cham} |}
                \lrp{
                    \frac{ {\mathbf{u}'}^{\cv_{k;\cham}} }{ \mathbf{u}^{\cv_{k;\cham}} }
                    \Xt'_{k;\cham}
                }^{-\sgn(\epsilon_{ik})}
        }^{-\epsilon_{ik}} \\
    & \stackrel{ |\cv_{k;\cham}|=-\sgn(\epsilon_{ik})\cv_{k;\cham} }{=} &
        \frac{ {\mathbf{u}'}^{\mathbf c_{i;\cham}}}{ \mathbf{u}^{\cv_{i;\cham}} }
        \Xt'_{i;\cham}
        \lrp{ 
            1+{\mathbf{u}'}^{| \cv_{k;\cham}|}{{\Xt}'_{k;\cham}}{}^{-\sgn(\epsilon_{ik})}
        }^{-\epsilon_{ik}} \\
    & \stackrel{ \cv_{i;\cham'}=\cv_{i;\cham} }{=} & 
        \frac{ {\mathbf{u}'}^{\cv_{i;\cham'}} }{ \mathbf{u}^{\cv_{i;\cham'}} }
        \Xt'_{i;\cham}
            \lrp{
                1+{\mathbf{u}'}^{ |\cv_{k;\cham}| } {\Xt'_{k;\cham}}{}^{ -\sgn(\epsilon_{ik}) }
            }^{-\epsilon_{ik}} \\
    &=& \mu_k^*
        \lrp{
            \frac{{\mathbf u'}^{\cv_{i;\cham'}}}{\mathbf u^{\cv_{i;\cham'}}}\Xt'_{i;\cham'}
        } 
    = \mu_k^* \circ \psi_{\cham'}^* \lrp{ \Xt_{i;\cham'} }.
\end{eqnarray*}

    \item[\bf Case 2:] If $\sgn(\epsilon_{ik})=\sgn(\mathbf c_{k;\cham})$ note that $\cv_{i;\cham'}=\cv_{i;\cham}+\sgn(\epsilon_{ik})[\epsilon_{ik}\cv_{k;\cham}]_+=\cv_{i;\cham}+\sgn(\epsilon_{ik})\epsilon_{ik}\cv_{k;\cham}$. This is used in the step labelled by $\bigstar$ below. We compute:
    \begin{eqnarray*}
     &\psi_\cham^*\circ \mu_k^*\lrp{\Xt_{i;\cham'}}
    =& \psi_\cham^*\lrp{ \Xt_{i;\cham}\lrp{\mathbf u^{| \cv_{k;\cham}|}+\Xt_{k;\cham}^{-\sgn(\epsilon_{ik})}}^{-\epsilon_{ik}} }\\
    &=& \frac{\mathbf u'^{\cv_{i;\cham}}}{\mathbf u^{\cv_{i;\cham}}}\Xt_{i;\cham}'\lrp{\mathbf u^{| \cv_{k;\cham} |}+\lrp{\frac{\mathbf u'^{\cv_{k;\cham}}}{\mathbf u^{\cv_{k;\cham}}}\Xt'_{k;\cham}}^{-\sgn(\epsilon_{ik})}}^{-\epsilon_{ik}}\\
    &\stackrel{| \cv_{k;\cham}|=\sgn(\epsilon_{ik})\cv_{k;\cham}}{=}&
    \frac{\mathbf u'^{\cv_{i;\cham}}}{\mathbf u^{\cv_{i;\cham}}}\Xt_{i;\cham}'
    \lrp{
    \frac{1}{\mathbf u^{-\sgn(\epsilon_{ik})\cv_{k;\cham}}}+ \frac{\mathbf u'^{-\sgn(\epsilon_{ik})\cv_{k;\cham}}}{\mathbf u^{-\sgn(\epsilon_{ik})\cv_{k;\cham}}} \Xt_{k;\cham}'^{-\sgn(\epsilon_{ik})}
    }^{-\epsilon_{ik}}\\
    &=&
    \frac{\mathbf u'^{\cv_{i;\cham}}}{\mathbf u^{\cv_{i;\cham}}}\Xt_{i;\cham}'
    \lrp{
    \frac{\mathbf u'^{-\sgn(\epsilon_{ik})\cv_{k;\cham}}}{\mathbf u^{-\sgn(\epsilon_{ik})\cv_{k;\cham}}}
        \lrp{
        \frac{1}{\mathbf u'^{-\sgn(\epsilon_{ik})\cv_{k;\cham}}} + \Xt_{k;\cham}'^{-\sgn(\epsilon_{ik})}
        }
    }^{-\epsilon_{ik}}\\
    &{=}&
    \frac{\mathbf u'^{\cv_{i;\cham}+\sgn(\epsilon_{ik})\epsilon_{ik}\cv_{k;\cham}}}{\mathbf u^{\cv_{i;\cham}+\sgn(\epsilon_{ik})\epsilon_{ik}\cv_{k;\cham}}}
    \Xt_{i;\cham}'
    \lrp{
    \mathbf u'^{| \cv_{k;\cham}|} + \Xt_{k;\cham}'^{-\sgn(\epsilon_{ik})}
    }^{-\epsilon_{ik}}\\
    &\stackrel{\bigstar}{=}&
    \mu_{k}^*\lrp{\frac{\mathbf u'^{\cv_{i;\cham'}}}{\mathbf u^{\cv_{i;\cham'}}} \Xt_{i;\cham'}' } = \mu_k^*\circ \psi_\cham'^*\lrp{\Xt_{i;\cham'}}.
    \end{eqnarray*}
    
\end{itemize}
\end{proof}

\begin{proposition}\thlabel{open}
$\Xfsps{\cham_0} \to \Spec\lrp{R}$ is an open morphism.
\end{proposition}

\begin{proof}
Openness is a local property of morphisms so we are free to restrict to affine patches.
We just need to see that $\Spec(R[\Xt_{1;\cham}, \dots, \Xt_{n;\cham}]) \to \Spec\lrp{R}$ is open.
But this is clear,
see e.g. \cite[\href{https://stacks.math.columbia.edu/tag/037G}{Lemma 037G}]{stacks-project}.
\end{proof}

\begin{corollary}\thlabel{smooth}
For $\vb{u} \in \lrp{\C^*}^n$ the fiber $\Xsp_{\vb{u}}$ is smooth.
In particular, $\Xsp$ is smooth.
\end{corollary}

\begin{proof}
First, we observe that $\Xsp_0$ is smooth.
The $\gv$-fan is simplicial, with the generators of each maximal cone forming a $\Z$-basis of $M$.
This implies smoothness for the toric central fiber \cite[Theorem~1.10]{Oda}. 
By \cite[Theorem~12.1.6(iii)]{EGAIV3}, 
the set of $x\in \Xfsps{\cham_0}$ such that $\Xsp_{\pi(x)}$ is smooth is open in $\Xfsps{\cham_0}$.
So, by \thref{open}, the set of $\vb{u} \in \Spec\lrp{R}$ such that $\Xsp_{\vb{u}}$ is smooth is open in $\Spec\lrp{R}$.
Then there is an open neighborhood of $0$ having smooth fibers.
The claim follows from \thref{iso of fibres}.
\end{proof}


\section{Strata and their degenerations}\label{sec:strata}

The results of the last section show that $\Xsp_{\vb{t} }$ (for a given $\vb{t}\in \mathbb{C}^n$) is constructed from a fan much the same way as a toric variety.
Unsurprisingly then, $\Xsp_{\vb{t} }$ is naturally stratified in a manner reminiscent of the usual stratification of a toric variety by lower dimensional toric subvarieties.
In fact, the affine patches $ \mathbb{A}^n_\cham$ defining $\Xsp_{\vb{t} }$ {\emph{are}} toric varieties, and the restriction of $\Xsp_{\vb{t} }$ strata to $ \mathbb{A}^n_\cham$ gives exactly the toric strata of $ \mathbb{A}^n_\cham$.
To get the full strata for $\Xsp_{\vb{t} }$, we simply glue together toric strata from different affine patches via the extended $\Xfsps{\cham_0}$ mutation formula. 
The main results of this section are that this does indeed yield a stratification of $\Xsp_{\vb{t}}$, the resulting strata are lower dimensional $\Xsp$-cluster varieties with specified coefficients, and these strata degenerate precisely to the toric strata of the central fiber $\Xsp_0 = \TV{\ccF}$.

Let us start by fixing terminology and briefly reviewing the toric description, following \cite[\S3.1]{FultonTV}.  
We say a scheme $Y$ is {\it{stratified by subschemes of type $P$}} if there is a collection $\mathcal{S}$ of subschemes $Y_i$ of type $P$ (called {\it{strata}})
such that 
\eqn{
Y= \bigcup_{\mathcal{S}} Y_i
}
and for every pair $\lrp{Y_i,Y_j}$ of elements of $\mathcal{S}$,
$Y_i \cap Y_j$ is also in $\mathcal{S}$.
For example, we discuss below how every toric variety is stratified by toric varieties-- the torus orbit closures.
Alternatively, though this is not the approach we adopt, we can stratify toric varieties by tori (torus orbits) viewing $\varnothing$ as a torus of dimension $-1$.\footnote{There is an analogous stratification of $\Xsp_{\vb{t}}$ is by lower dimensional $\cX$-varieties with specified coefficients, rather than their special completions.}

Let $\TV{\Sigma}$ be the toric variety associated to a fan $\Sigma$.
Let $N$ be the cocharacter lattice of the defining torus and $M$ its character lattice.
For every cone $\tau$ in $\Sigma$, we have a torus orbit $O(\tau)$, which itself is a torus.
We denote the closure of $O(\tau)$ by $V(\tau)$.
Then $V(\tau)$ is a toric variety, and it is what we refer to as a \emph{toric stratum}.
The torus $O(\tau)$ is $\Spec(\C[\tau^\perp \cap M ])$, so the fan for $V(\tau)$ lives in 
$(\tau^\perp \cap M)^*  \otimes \R$.
Let $N(\tau):= (\tau^\perp \cap M)^*$ be the quotient of $N$ by the group completion of $\tau \cap N$.
The fan for $V(\tau)$, denoted $\Star(\tau)$, is built from the cones of $\Sigma$ containing $\tau$ as a face.
If $\sigma$ is such a cone, we take its image $\overline{\sigma}$ in $N(\tau) \otimes \R$.
The collection of these $\sigma$ form a subfan of $\Sigma$, and their images $\overline{\sigma}$ form the fan $\Star(\tau)$ in $N(\tau)\otimes \R$.

To realize $V(\tau)$ as a closed subscheme of $\TV{\Sigma}$, we work with the defining affine patches.
In $V(\tau)$, these have the form $\Spec(\C[\overline{\sigma}^\vee \cap \tau^\perp \cap M])$.
Note that $\overline{\sigma}^\vee \cap \tau^\perp \cap M$ is the same as $\sigma^\vee \cap \tau^\perp \cap M$, and $\Spec\lrp{\C\lrb{\sigma^\vee \cap M}}$ is a defining affine patch of $\TV{\Sigma}$.
We have $\Spec(\C[\sigma^\vee \cap \tau^\perp \cap M]) \hookrightarrow \Spec(\C[\sigma^\vee \cap M])$ via the projection of rings:
\eqn{
\C[\sigma^\vee \cap M] \to \C[\sigma^\vee \cap \tau^\perp \cap M],\ \
z^m \mapsto \begin{cases} z^m  & \mathrm{if }\ m \in \sigma^\vee \cap \tau^\perp \cap M, \\
0 & \mathrm{otherwise}.
\end{cases}
}

Since the defining affine patches of $\Xsp_{\vb{t} }$ are precisely the maximal affine toric varieties making up $\TV{\ccF}$,
these affine patches are stratified as described above.
But how do the strata of different patches relate?
Extending the toric analogy, if we fix a cone $\tau$ in $\ccF$, do the associated strata in distinct affine patches glue together to form a closed subscheme of $\Xsp_{\vb{t}}$?  
And if so, does this subscheme inherit any structure from $\Xsp_{\vb{t}}$, particularly structure encoded by a tropical version of $\Star(\tau)$?

In \cite{FG_X} Fock and Goncharov give results in this direction for $\Xsp = \Xsp_{\vb{1}}$.
Specifically, they show that if we take an $M$-seed $s$, freeze some directions and set the corresponding $\cX$-variables to $0$, we obtain an $\cX$-cluster variety as an immersed subscheme at the boundary of $\Xsp$.
This line of reasoning extends directly to the case of coefficients and is a central argument in \thref{strata}.
The other key ingredient is Cao and Li's result \cite[Theorem~6.2]{CaoLi}. Stated in the terminology of this paper it is as follows:
\begin{theorem}[\cite{CaoLi}]
Let $\lrp{\Gamma,v_0}$ be arbitrary skew-symmetrizable initial data with frozen directions. Then any two clusters containing the same subset $\lbrace A_1,\dots,A_a\rbrace$ of $\cA$-cluster variables can be connected to each other by iterated mutation without exchanging any of the cluster variables in $\lbrace A_1,\dots,A_a\rbrace$.
\end{theorem}
This result leads to the {\emph{stratification}} of $\Xsp_{\vb{t}}$ by specially completed $\cX$-varieties with specified coefficients, an upgrade on $\Xsp_{\vb{t}} \setminus \cX_{\vb{t}}$ having immersed specially completed $\cX$-varieties with specified coefficients.

\begin{theorem}\thlabel{strata}
    Let $\Xsp_{\vb{t}}$ be a specially completed $\cX$-variety with specified coefficients, and 
    let $\tau$ be a $k$-dimensional cone in $\ccF$.
    Then $\tau$ defines a codimension $k$ closed subscheme $V(\tau)_{\vb{t}}$ of $\Xsp_{\vb{t}}$ for each $\vb{t}$.
    Moreover, $V(\tau)_{\vb{t}}$ is again a specially completed $\cX$-variety with specified coefficients $\vb{t}$.
\end{theorem}

\begin{proof}
On the affine patches defining $\Xsp_{\vb{t}}$ we still have the inclusion
\eqn{
\Spec\lrp{\C\lrb{ \sigma^\vee \cap \tau^\perp \cap M}} \hookrightarrow \Spec\lrp{\C\lrb{ \sigma^\vee \cap M}}           
}
for all cones $\sigma$ of $\ccF$ containing $\tau$.
Note that $\tau$ is the $\R_{\geq 0}$-span of $k$-many $\gv$-vectors.
Then for any $\gv$-cone $\cham$ containing $\tau$, $\tau^\perp$ is the $\R$-span of the $(n-k)$-many $\cv$-vector generators of $\chamdual = \cham^\vee$ that are orthogonal to $\tau$.

Take $\sigma$ to be an $m$-dimensional cone of $\ccF$ with $\tau \subset \sigma \subset \cham$.
Then
 ${\C[\sigma^\vee \cap \tau^\perp \cap M]}$ is just $\C [ X_{i_1},\dots, X_{i_{m-k}},X_{i_{m-k+1}}^{\pm 1},\dots,X_{i_{n-k}}^{\pm 1} ]$.
Here the first $m-k$ terms are the $\cX$-variables associated to generators of $\chamdual$ that are orthogonal to $\tau$ but not $\sigma$,
while the remaining terms correspond to generators orthogonal
to both $\tau$ and $\sigma$.
There are two extreme cases.
When $\sigma = \tau$, we have $\sigma^\vee \cap \tau^\perp = \tau^\perp$. 
Then
\eqn{
\Spec\lrp{\C\lrb{ \sigma^\vee \cap \tau^\perp \cap M}} =  \Spec\lrp{
\C\lrb{X_{i_1}^{\pm 1},\dots,X_{i_{n-k}}^{\pm 1}}
},           
} 
which will be a torus in the atlas of the $\cX$-variety associated to $\tau$.
We call it $T_{M,\cham}(\tau)$.
On the other side of the spectrum, when $\sigma = \cham$ we have 
\eqn{
\Spec\lrp{\C\lrb{ \sigma^\vee \cap \tau^\perp \cap M}} =  \Spec\lrp{
\C\lrb{X_{i_1},\dots,X_{i_{n-k}}}
}.           
}
This will be the corresponding copy of $\A^{n-k}$ 
in the atlas of this $\cX$-variety's special completion with specified coefficients.  Call it $\A^{n-k}_{M,\cham}\lrp{\tau}$.

Mutation is given by crossing walls of $\ccF$.
The relevant mutations here correspond to walls containing $\tau$, and mutation directions correspond to indices of $\tau^\perp$.
Specifically, if we set ${I(\tau) :=\lrc{i_1,\dots,i_{n-k}}}$, the relevant mutation directions are $\Iuf(\tau):= I(\tau)\setminus F$.
Observe that:
\begin{enumerate}
    \item \label{consistent} Since we will never mutate at $j \notin I(\tau)$, by the mutation formula (\ref{eq:mutfamily}) we have 
\eqn{\mu_k^*\lrp{\Xt_{j;\cham'}} = \Xt_{j;\cham}\lrp{\bt^{[\sgn(\epsilon_{jk})\bc_{k;\cham}]_+} + \bt^{[-\sgn(\epsilon_{jk})\bc_{k;\cham}]_+}\Xt_{k;\cham}^{-\sgn(\epsilon_{jk})}}^{-\epsilon_{jk}}}
    for all $j \notin I(\tau)$.
    \item \label{independent} For $i \in I(\tau)$, the mutation formula (\ref{eq:mutfamily}) is independent of $\Xt_{j;\cham}$ for all $j \notin I(\tau)$.
\end{enumerate}
 
On account of (\ref{consistent}), we can consistently set $\Xt_{j;\cham}=0 $ for all $j \notin I(\tau)$ and all $\gv$-cones $\cham$ containing $\tau$.
By (\ref{independent}), doing so will not alter the relations among the variables we have not set to 0-- those indexed by $I(\tau)$.
Now fix some $\cham_0 \supset \tau$.
By the preceding discussion, if we set $\Xt_{j;\cham}=0$ for $j \notin \Iuf(\tau)$, $\cham \supset \tau$, and restrict the indices of our exchange matrix $\boldsymbol{\varepsilon}$ at $\cham_0$ to $I(\tau)$,
what we obtain will be an immersed $(n-k)$-dimensional $\cX$-variety with specified coefficients $\cX_{\vb{t}}(\tau)_{\cham_0}$ glued from the tori $T_{M,\cham}(\tau)$ mutation equivalent to $T_{M,\cham_0}(\tau)$.
Replacing $T_{M,\cham}(\tau)$ with $\A^{n-k}_{M,\cham}(\tau)$ gives its special completion $\Xsp_{\vb{t}}(\tau)_{\cham_0}$.

Now let $\cham_0'$ be another $\gv$-cone containing $\tau$.
If $\cham_0$ and $\cham_0'$ are connected by mutation {\emph{in}} $\Iuf(\tau)$, then $\cX_{\vb{t}}(\tau)_{\cham_0} = \cX_{\vb{t}}(\tau)_{\cham_0'}$ and $\Xsp_{\vb{t}}(\tau)_{\cham_0} = \Xsp_{\vb{t}}(\tau)_{\cham_0'}$.  
If, on the other hand, $\cham_0$ and $\cham_0'$ were only connected by mutation in $\Iuf$ (rather than $\Iuf(\tau)$) then mutation in $\Iuf$ would glue $\cX_{\vb{t}}(\tau)_{\cham_0}$ with $ \cX_{\vb{t}}(\tau)_{\cham_0'}$
and $\Xsp_{\vb{t}}(\tau)_{\cham_0} $ with  $\Xsp_{\vb{t}}(\tau)_{\cham_0'}$
in the ambient space $\Xsp_{\vb{t}}$.  
But $\cX_{\vb{t}}(\tau)_{\cham_0}$ and $\cX_{\vb{t}}(\tau)_{\cham_0'}$ would be distinct immersed $\cX$-cluster varieties with specified coefficients (and their special completions would likewise be distinct), and we would {\emph{not}} have the claimed stratification.
However, the content of \cite[Theorem~6.2]{CaoLi} is precisely that $\cham_0$ and $\cham_0'$ {\emph{are}} connected by mutation in $\Iuf(\tau)$.
That is, $\Xsp_{\vb{t}}(\tau)_{\cham}$ is independent of $\cham$, and is exactly the immersed subscheme given by $\Star\lrp{\tau}$.
We call it $V\lrp{\tau}_{\vb{t}}$.
\end{proof}

From this perspective, we have a family of specially completed $\cX$-varieties with specified coefficients, which are in turn stratified by specially completed $\cX$-varieties with specified coefficients.
The $\Xsp_{\vb{t}}$ degenerate to toric varieties, which themselves are stratified by toric varieties.
It would be very satisfying then if the $\Xsp_{\vb{t}}$ strata degenerate to the toric strata.
In fact, from the description of the $\Xsp_{\vb{t}}$ strata given above, it is almost immediate that this occurs.
Taking $\vb{t}$ to $0$ in \eqref{eq:mutfamily} yields the following result.

\begin{theorem}\thlabel{strata2strata}
Let $\Xsp_{\vb{t}}$ and $V(\tau)_{\vb{t}}$ be as in \thref{strata}.
Then $V(\tau)_{\vb{t}}$ degenerates to the toric stratum $V(\tau)$ of the toric central fiber.
\end{theorem}

\section{Examples and applications}\label{sec:examples}

In this section we give examples of toric degenerations of compactified cluster varieties, and we relate the $\cA$ and $\cX$ degenerations in these examples.

\subsection{Rietsch--Williams degeneration of Grassmannians}\label{ex:RW}
In \cite{RW17}, Rietsch and Williams study Newton--Okounkov bodies for the anticanonical divisor of the Grassmannian.
For a given seed in the underlying cluster algebra they construct a valuation using {\emph{flow polynomials}} and Postnikov's \emph{plabic graphs}\footnote{In this setting plabic graphs encode cluster seeds consisting of only Pl\"ucker coordinates. We refer to \cite[\S3]{RW17} for details.} (see \cite{Pos06}).
By \cite{An13} and \cite{Kav15} the Newton--Okounkov bodies (rational polytopes in this case) induce toric degenerations of the Grassmannain.
Further, they show that the Newton--Okounkov polytopes are in fact equal to polytopes arising from a Landau--Ginzburg potential introduced by Marsh and Rietsch in \cite{MR_Bmod_publ}.

Below we relate the Rietsch--Williams construction to the $\Xfam$- and $\cAp$-families for $\Grass_2(\C^5)$.
Passing though $\Xfam$, we show that the Rietsch--Williams toric degeneration is precisely the $\cAp$-family of \cite{GHKK}.
We identify both the Newton--Okounkov body $\Delta_{\mathrm{RW}}$ and the superpotential polytope  $\Gamma_{\mathrm{RW}}$ of \cite{RW17} with the superpotential polytope $\Xi_{\mathrm{GHKK}}$ of \cite{GHKK}, giving a new proof of the equality of the Rietsch--Williams polytopes.
We will treat the general case of $\Grass_k\lrp{\C^n}$ in a separate article \cite{BCMNC}. 
We use the notation introduced in \cite{RW17}.
The strategy is as follows.
\begin{enumerate}
    \item Identify the Marsh--Rietsch potential on $\XXdo \subset \Grass_2(\C^5)$ with the Gross--Hacking--Keel--Kontsevich potential on its mirror $\XXo$ via an isomorphism of $\XXdo \stackrel{\sim}{\to} \XXo $. This identifies $\Gamma_{\mathrm{RW}}$ with $\Xi_{\mathrm{GHKK}}$ and it identifies as well the uncompactified families $\Xfam$ and $\cAp$.
    \item Use this isomorphism to reinterpret the Rietsch--Williams flow polynomials for Pl\"ucker coordinates on $\XXdo$ as polynomials in the $\cX$-variables of $\XXo$.
    \item Extend the flow polynomials to $\Xfam$ and identify the Rietsch--Williams valuation $\val_G\lrp{p_{ij}}$ of each Pl\"ucker coordinate $p_{ij}$ with $\cv\lrp{p_{ij}}$. This identifies $\Delta_{\mathrm{RW}}$ with $\Xi_{\mathrm{GHKK}}$, and establishes that the toric degeneration {\emph{of the compactification}} $\Grass_2(\C^5)$ in \cite{RW17} matches the \cite{GHKK} toric degeneration $\overline{\cAp}$.
\end{enumerate}
In the affine cone $C(\Grass_2(\C^5))$ of $\Grass_2(\C^5)$ take the divisors $D_{i,i+1}:=\lrc{p_{i,i+1}=0}$ for $i\in \Z/ 5 \Z$ and let $D$ be their union.
The $\cA$-cluster variety we are considering is the complement of $D$ in $C(\Grass_2(\C^5))$.
We discussed this cluster variety briefly in \thref{ex:GrassCoeff} to illustrate how to pass from a cluster pattern with geometric coefficients to a higher rank coefficient free cluster pattern with frozen directions.
Let $\XXdo$ denote the image of $\cA$ in $\Grass_2(\C^5)$. It inherits a cluster structure from $\cA$.
Consider the plabic graph $G^{\mathrm{rec}}_{3,5}$ with perfect orientation as shown in Figure~\ref{fig:plabic}.\smallskip

\noindent
\begin{minipage}{\linewidth}
\captionsetup{type=figure}
\begin{center}
\begin{tikzpicture}[scale=.6]

    \draw[name path=circle] (0,0) circle (3cm);
    
    \node[circle, fill=black, inner sep=0pt, minimum size=5pt] (mid) at (0,0) {};
    \node[circle, fill=black, inner sep=0pt, minimum size=5pt] (r) at (2,0) {};
    \node[circle, fill=black, inner sep=0pt, minimum size=5pt] (l) at (-2,0) {};        
    \node[circle, draw=black, fill=white, inner sep=0pt, minimum size=5pt] (tr) at (1,1) {};
    \node[circle, draw=black, fill=white, inner sep=0pt, minimum size=5pt] (tl) at (-1,1) {};
    \node[circle, draw=black, fill=white, inner sep=0pt, minimum size=5pt] (br) at (1,-1) {};
    \node[circle, draw=black, fill=white, inner sep=0pt, minimum size=5pt] (bl) at (-1,-1) {};
    
    \path[name path=path1] (0,1)--(3,1);
    \path[name path=path2] (0,0)--(4,0);
    \path[name path=path3] (1,0)--(1,-3);
    \path[name path=path4] (-1,0)--(-1,-3);
    \path[name path=path5] (0,0)--(-3,-2);
    
    \path[name intersections={of=circle and path1,by=1c}];
    \path[name intersections={of=circle and path2,by=2c}];
    \path[name intersections={of=circle and path3,by=3c}];
    \path[name intersections={of=circle and path4,by=4c}];
    \path[name intersections={of=circle and path5,by=5c}];

    \path (1c)++(0.25,0) node {\small $1$}; 
    \path (2c)++(0.25,0) node {\small $2$}; 
    \path (3c)++(0,-0.4) node {\small $3$}; 
    \path (4c)++(0,-0.4) node {\small $4$}; 
    \path (5c)++(-0.2,-0.25) node {\small $5$}; 

    \draw[postaction={decorate},decoration={markings,mark=at position .5 with {\arrow[line width=1.25pt] {>};}}](1c)--(tr);
    \draw[postaction={decorate},decoration={markings,mark=at position .5 with {\arrow[line width=1.25pt] {>};}}](tr)--(tl);
    \draw[postaction={decorate},decoration={markings,mark=at position .5 with {\arrow[line width=1.25pt] {>};}}](2c)--(r);
    \draw[postaction={decorate},decoration={markings,mark=at position .5 with {\arrow[line width=1.25pt] {>};}}](r)--(br);
    \draw[postaction={decorate},decoration={markings,mark=at position .5 with {\arrow[line width=1.25pt] {>};}}](tr)--(r);
    \draw[postaction={decorate},decoration={markings,mark=at position .5 with {\arrow[line width=1.25pt] {>};}}](br)--(mid);
    \draw[postaction={decorate},decoration={markings,mark=at position .5 with {\arrow[line width=1.25pt] {>};}}](br)--(3c);
    \draw[postaction={decorate},decoration={markings,mark=at position .5 with {\arrow[line width=1.25pt] {>};}}](mid)--(bl);
    \draw[postaction={decorate},decoration={markings,mark=at position .5 with {\arrow[line width=1.25pt] {>};}}](bl)--(4c);
    \draw[postaction={decorate},decoration={markings,mark=at position .5 with {\arrow[line width=1.25pt] {>};}}](bl)--(l);
    \draw[postaction={decorate},decoration={markings,mark=at position .5 with {\arrow[line width=1.25pt] {>};}}](l)--(5c);
    \draw[postaction={decorate},decoration={markings,mark=at position .5 with {\arrow[line width=1.25pt] {>};}}](tl)--(l);
    \draw[postaction={decorate},decoration={markings,mark=at position .5 with {\arrow[line width=1.25pt] {>};}}](tl)--(mid);

    \node at (1,0) {$\Yng(2)$};
    \node at (-1,0) {$\Yng(1)$};
    \node at (2.25,0.5) {$\Yng(3)$};
    \node at (2,-1.25) {$\Yng(3,3)$};
    \node at (0,-1.5) {$\Yng(2,2)$};
    \node at (-1.75,-1.25) {$\Yng(1,1)$};
    \node at (0,2) {$\varnothing$};

\end{tikzpicture}
\captionof{figure}{\label{fig:plabic} $G^{\mathrm{rec}}_{3,5}$ with perfect orientation, as in \cite[Figure~7]{RW17}. 
}  \bigskip
\end{center}
\end{minipage}

Faces of the plabic graph, labeled by Young diagrams $\mu$, are associated with variables $x_{\mu}$.
The flow model assigns to every Pl\"ucker coordinate $p_{ij}\in\C[C(\Grass_2(\C^5))]$ a polynomial in the variables $x_{\mu}$.
This yields the flow polynomials (see \cite[Example~6.11]{RW17}): 

\eq{
\flow (p_{12}) &= 1,  \qquad\qquad\qquad\qquad\qquad\qquad
\flow (p_{13}) = x_{\scaleto{\Yng(3,3)\mathstrut}{5pt}},\\
\flow (p_{14}) &= x_{\scaleto{\Yng(2,2)\mathstrut}{5pt}}x_{\scaleto{\Yng(3,3)\mathstrut}{5pt}}, \qquad\qquad\qquad\qquad\qquad \ 
\flow (p_{15}) = x_{\scaleto{\Yng(1,1)\mathstrut}{5pt}}x_{\scaleto{\Yng(2,2)\mathstrut}{5pt}}x_{\scaleto{\Yng(3,3)\mathstrut}{5pt}},\\
\flow (p_{23}) &= x_{\scaleto{\Yng(3)\mathstrut}{5pt}}x_{\scaleto{\Yng(3,3)\mathstrut}{5pt}}, \qquad\qquad\qquad\qquad \ \ \ \ \ \ 
\flow (p_{24}) = x_{\scaleto{\Yng(3)\mathstrut}{5pt}}x_{\scaleto{\Yng(2,2)\mathstrut}{5pt}}x_{\scaleto{\Yng(3,3)\mathstrut}{5pt}}(1+x_{\scaleto{\Yng(2)\mathstrut}{5pt}}),\\
\flow (p_{25}) &= x_{\scaleto{\Yng(3)\mathstrut}{5pt}}x_{\scaleto{\Yng(1,1)\mathstrut}{5pt}}x_{\scaleto{\Yng(2,2)\mathstrut}{5pt}}x_{\scaleto{\Yng(3,3)\mathstrut}{5pt}}(1+x_{\scaleto{\Yng(2)\mathstrut}{5pt}}+x_{\scaleto{\Yng(1)\mathstrut}{5pt}}x_{\scaleto{\Yng(2)\mathstrut}{5pt}}), \ \ \ \ \ 
\flow (p_{34}) = x_{\scaleto{\Yng(2)\mathstrut}{5pt}}x_{\scaleto{\Yng(3)\mathstrut}{5pt}}x_{\scaleto{\Yng(2,2)\mathstrut}{5pt}}x^2_{\scaleto{\Yng(3,3)\mathstrut}{5pt}},\\
\flow (p_{35}) &= x_{\scaleto{\Yng(2)\mathstrut}{5pt}}x_{\scaleto{\Yng(3)\mathstrut}{5pt}}x_{\scaleto{\Yng(1,1)\mathstrut}{5pt}}x_{\scaleto{\Yng(2,2)\mathstrut}{5pt}}x^2_{\scaleto{\Yng(3,3)\mathstrut}{5pt}}(1+x_{\scaleto{\Yng(1)\mathstrut}{5pt}}), \qquad\qquad
\flow (p_{45}) = x_{\scaleto{\Yng(1)\mathstrut}{5pt}}x_{\scaleto{\Yng(2)\mathstrut}{5pt}}x_{\scaleto{\Yng(3)\mathstrut}{5pt}}x_{\scaleto{\Yng(1,1)\mathstrut}{5pt}}x^2_{\scaleto{\Yng(2,2)\mathstrut}{5pt}}x^2_{\scaleto{\Yng(3,3)\mathstrut}{5pt}}.
}{FlowPoly}

We interpret $\flow(p_{12}) = 1$ as a normalization.
Dividing every Pl\"ucker coordinate by $p_{12}$ allows us to think of them as functions on $\XXdo$ rather than $\cA$. 
In light of this, we view the arguments of $\flow$ as functions on $\XXdo$. 
We will reinterpret the flow polynomials in terms of $\cX$-variables, and then extend the resulting functions to $\Xfam$. 
Observe that the flow polynomials in \eqref{FlowPoly} which are actually {\emph{monomials}} determine the cluster depicted in Figure~\ref{fig:quiver Gr(2,5)}.\medskip

\noindent
\begin{minipage}{\linewidth}
\captionsetup{type=figure}
\begin{center}
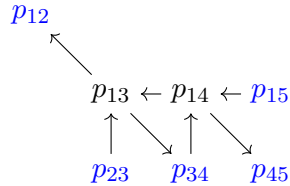

\begin{tikzpicture}[scale=.7]

    \node (12) at (0,3) {\tc{blue}{$p_{12}$}};
    \node (13) at (1.5,1.5) {\tc{black}{$p_{13}$}};
    \node (14) at (3,1.5) {\tc{black}{$p_{14}$}};
    \node (15) at (4.5,1.5) {\tc{blue}{$p_{15}$}};
    \node (23) at (1.5,0) {\tc{blue}{$p_{23}$}};
    \node (34) at (3,0) {\tc{blue}{$p_{34}$}};
    \node (45) at (4.5,0) {\tc{blue}{$p_{45}$}};

    \draw[->] (13)--(12);
    \draw[->] (14)--(13);
    \draw[->] (15)--(14);
    \draw[->] (23)--(13);
    \draw[->] (34)--(14);
    \draw[->] (13)--(34);
    \draw[->] (14)--(45);

\end{tikzpicture}
\captionof{figure}{\label{fig:quiver Gr(2,5)} Cluster determined by the monomials of \eqref{FlowPoly}.}
\end{center} \bigskip
\end{minipage}

\noindent
We take this to be the initial cluster and write $s_0$ for the initial seed.
Throughout this example, cluster variables written without a subscript specifying the seed will come from this seed. 
So when we write, for instance, $X_{i}$, we will mean $X_{i;s_0}$.

The Pl\"ucker coordinates are $\cA$-variables, and we want to express them in terms of $\cX$-variables.
To do this, we make use of a map $p:\cA\to\cX$  (see \thref{def:p-map}).
Recall that this map $p$ is defined in terms of pull-backs of functions.
On character lattices of cluster tori, it is given by a map $p^*:N {\to } M $
that agrees with $n\mapsto \sk{n}{\ \cdot\ }$ up to a choice of map $N/\Nuf \to \Nuf^\perp$.\footnote{As the exchange matrix in this example is skew-symmetric, $\cA^\vee = \cX$ and the superscript ${}^\circ$ for lattices becomes unnecessary.}

Recall that $A_i= z^{f_i}$ and $X_i = z^{e_i}$. 
For $e_{13}$ and $e_{14}$, we compute
\eq{p^*\lrp{X_{13}} &= \frac{p_{12} p_{34}}{p_{14} p_{23}}, 
\qquad \text{\small{ $ \frac{\flow\lrp{p_{12}} \flow\lrp{p_{34}}}{\flow\lrp{p_{14}}\flow\lrp{ p_{23}}} $}} = x_{\yd{2}}, \ \ \ 
\\
p^*\lrp{X_{14}} &= \frac{p_{13} p_{45}}{p_{15} p_{34}}, 
\qquad \text{\small{ $\frac{\flow\lrp{p_{13}} \flow\lrp{p_{45}}}{\flow\lrp{p_{15}}\flow\lrp{ p_{34}}} $}} = x_{\yd{1}}.}{X13}

The remaining pull-backs depend on the choice of map
$N/\Nuf \to \Nuf^\perp$,
and we will compare the Marsh--Rietsch and Gross--Hacking--Keel--Kontsevich potentials in order to make this choice.
From the \cite{GHKK} perspective, each component $D_{i,i+1}$ of $D$ defines a $\tf$-function $\tf_{i,i+1}$ on the mirror $\cX$ of $\cA$. 
If $s$ is a seed optimized\footnote{This refers to the vertex $v_{i,i+1}$ being a sink in our case.} for the frozen index ${}_{i,i+1}$ (see \cite[Definition~9.1, Lemma~9.3]{GHKK}), 
the $\tf$-function $\tf_{i,i+1}$ is given by 
$\tf_{i,i+1} = z^{-e_{i,i+1;s}} = X_{i,i+1;s}^{-1}$.
We can mutate from $s$ to $s_0$ to express $\tf_{i,i+1}$ in terms of the variables of the initial seed.
The potential $W_{\text{GHKK}}:\cX\to\C$ is the sum of the $\tf$-functions associated to each component $D_{i,i+1}$ of $D$.
It is straightforward to compute these $\tf$-functions in terms of the $\cX$-variables of $s_0$:
\eq{
\tf_{12} &= X_{12}^{-1}, \\ 
\tf_{23} &= X_{23}^{-1} + X_{23}^{-1} X_{13}^{-1}, \\
\tf_{34} &= X_{34}^{-1} + X_{34}^{-1} X_{14}^{-1}, \\
\tf_{45} &= X_{45}^{-1}, \\ 
\ \ \tf_{15} &= X_{15}^{-1} + X_{15}^{-1} X_{14}^{-1} + X_{15}^{-1} X_{14}^{-1} X_{13}^{-1}  .}{thetas}

Next, the Marsh--Rietsch potential $W_{\text{MR}}$ on $\XXdo \times \C^*$ also has a summand $W_i$ associated to each component of $D$, and the potential $W_{\text{MR}}|_{q=1}:\XXdo\to\C$ is the sum of the $W_i$'s.  
See \cite[Equations~10.1~and~10.2]{RW17}. 
We can choose the map $N/\Nuf \to \Nuf^\perp$ such that $p^*\lrp{\vartheta_{i+1,i+2}}$ is the summand $W_i= \dfrac{p_{i+1,i+3}}{p_{i+1,i+2}}$ of \cite[Equation~10.2]{RW17}.
For each $i$ we compute $\flow(W_i)=\dfrac{\flow(p_{i+1,i+3})}{\flow(p_{i+1,i+2})}$ and obtain
\eq{
\flow(W_5) &= x_{\yd{3,3}}, \\
\flow(W_1) &= x_{\yd{2,2}}\lrp{1+ x_{\yd{2}}}, \\
\flow(W_2) &= x_{\yd{1,1}}\lrp{1+ x_{\yd{1}}}, \\
\flow(W_3) &= (x_{\yd{1}}  x_{\yd{2}}  x_{\yd{3}} x_{\yd{1,1}}  x_{\yd{2,2}}  x_{\yd{3,3}})^{-1}, \\ 
\qquad
\flow(W_4) &=  x_{\yd{3}}\lrp{1+x_{\yd{2}}+x_{\yd{1}}x_{\yd{2}}}
 .}{flows}

Note that functions on $\XXdo$ are functions on $\cA$ fixed by the $\C^*$-action of simultaneous scaling.
We have a $\tf$-basis for $\ssO(\cA)$ from $\cX^\trop(\Z)$.
The slice $\lrp{\XXo}^\trop(\Z)$ corresponding to functions fixed by this action gives a $\tf$-basis for $\ssO(\XXdo)$. 
Explicitly, let 
\eqn{
a = a_{12} f_{12} + a_{13} f_{13} + a_{14} f_{14} + a_{15} f_{15}+ a_{23} f_{23} + a_{34} f_{34}  + a_{45} f_{45}
}
be an arbitrary element of $M \cong \cX^\trop(\Z)$.
Then $a\in \lrp{\XXo}^\trop(\Z)$ if and only if 
\eqn{a_{12} + a_{13} + a_{14}  + a_{15} + a_{23} + a_{34} + a_{45} = 0.}
Dually, we identify $(\XXdo)^\trop\lrp{\Z}$ with the quotient of $N$ by 
\eqn{\vb{1}=e_{12}+ e_{13} + e_{14}  + e_{15} + e_{23} + e_{34} + e_{45} .}
Observe that $p^*\lrp{N}= \vb{1}^\perp\subset M$ and $\Z \cdot \vb{1} = \ker \lrp{p^*}$, so $p^*$
induces an isomorphism 
\eq{
\overline{p}^*: (\XXdo)^\trop\lrp{\Z} \stackrel{\sim}{\to} \lrp{\XXo}^\trop\lrp{\Z},
}{p-star-iso} 
i.e. an isomorphism $\overline{p}:\XXdo \stackrel{\sim}{\to} \XXo$.
We have the following commutative diagrams.

\noindent
\begin{minipage}{\linewidth}
\begin{center}
\begin{tikzcd}

    C\lrp{\Grass_2\lrp{\C^5}} \arrow[hookleftarrow]{r} \arrow{d}  & \cA \arrow{r}{p} \arrow{d} &\cX 
    
    & \cA^\trop\lrp{\Z} \cong N \arrow{r}{p^*} \arrow[twoheadrightarrow]{d}  & M \cong \cX^\trop\lrp{\Z}
    \\
    
    \Grass_2\lrp{\C^5}\arrow[hookleftarrow]{r} & \XXdo \arrow{r}{\overline{p}}[swap]{\sim} & \XXo \arrow[hookrightarrow]{u}
    
    & (\XXdo)^\trop\lrp{\Z} \cong N/\vb{1} \arrow{r}{\overline{p}^*}[swap]{\sim} & \vb{1}^\perp \cong (\XXo)^\trop\lrp{\Z} \arrow[hookrightarrow]{u}
    
\end{tikzcd}
\end{center}
\end{minipage}

Below, we identify the $x_\mu$ with functions on $\XXo$, and conclude that $\flow$ is precisely $\lrp{\overline{p}^{-1}}^*$.
Given a function $F$ on $\cX$ (denoted with a capital letter), denote by $f$ (in lowercase) the associated function $ \lrp{\overline{p}^{-1}}^*(p^*\lrp{F})$ on $\XXo$.
For example, for every $\cX$-variable $X_{ij}$ we obtain a function on $\XXo$ denoted $x_{ij}=(\overline{p}^{-1})^*\circ (p^*(X_{ij}))$.

In light of \eqref{X13}, we make the identifications
\eqn{
x_{\yd{2}} = x_{13}
 \ \ \text{ and } \ \
x_{\yd{1}} = x_{14}.
}
Next, comparing \eqref{thetas} and \eqref{flows}, we set
\eqn{
x_{\yd{3,3}} &= \vb{x}^{\cv\lrp{\tf_{12}}}= x_{12}^{-1}, \qquad \qquad
x_{\yd{2,2}} = \vb{x}^{\cv\lrp{\tf_{23}}}= x_{23}^{-1}x_{13}^{-1},\\
x_{\yd{1,1}} &= \vb{x}^{\cv\lrp{\tf_{34}}}= x_{34}^{-1}x_{14}^{-1}, \qquad
x_{\yd{3}} = \vb{x}^{\cv\lrp{\tf_{15}}}= x_{15}^{-1}x_{14}^{-1}x_{13}^{-1}.
}
Finally, observe \eqn{\cv\lrp{X_{13}}+\cv\lrp{X_{14}}+\cv\lrp{\tf_{12}}+\cv\lrp{\tf_{23}}+\cv\lrp{\tf_{34}}+\cv\lrp{\tf_{45}}+\cv\lrp{\tf_{15}} = -\vb{1} \in \ker(p^*) .}
With this in mind,
we set
\eqn{x_\varnothing x_{\yd{1}} x_{\yd{2}} x_{\yd{3}} x_{\yd{1,1}} x_{\yd{2,2}} x_{\yd{3,3}}=1\ \
\text{ and } \ \
x_\varnothing = \vb{x}^{\cv\lrp{\tf_{45}}} = x_{45}^{-1}.
}
With these identifications, it is immediate that $\flow=\lrp{\overline{p}^{-1}}^* $.
We can now extend all flow polynomials to the family $\Xfams{s_0}$ {\emph{deforming $\XXo$}} (as opposed to the family deforming $\cX$).
The coefficients satisfy the same relations as the $x$'s.
The monomials in \eqref{FlowPoly} extend directly, e.g. 
\eqn{
\widetilde{\flow\lrp{p_{13}}} = \widetilde{x}_{12}^{-1} = \widetilde{x}_{\yd{3,3}} .
}
The extensions of the three remaining flow polynomials are
\eqn{
\widetilde{\flow (p_{24})} &= \frac{ \widetilde{x}_{34} \widetilde{x}_{45}\lrp{1+t_{13} \widetilde{x}_{13}}}{ \widetilde{x}_{13}} 
=   \widetilde{x}_{\yd{3}} \widetilde{x}_{\yd{2,2}} \widetilde{x}_{\yd{3,3}} \lrp{1+t_{\yd{2}} \widetilde{x}_{\yd{2}}}\\
\widetilde{\flow (p_{25})} &= \frac{\widetilde{x}_{45} \lrp{1 + t_{13} \widetilde{x}_{13} + t_{13}t_{14} \widetilde{x}_{13} \widetilde{x}_{14} } }{  \widetilde{x}_{13} \widetilde{x}_{14}}  = 
\widetilde{x}_{\yd{3}} \widetilde{x}_{\yd{1,1}}  \widetilde{x}_{\yd{2,2}} \widetilde{x}_{\yd{3,3}} \lrp{1+t_{\yd{2}} \widetilde{x}_{\yd{2}}+ t_{\yd{1}}t_{\yd{2}} \widetilde{x}_{\yd{1}} \widetilde{x}_{\yd{2}}}\\
\widetilde{\flow (p_{35})} &= \frac{ \widetilde{x}_{45} \lrp{1+t_{14} \widetilde{x}_{14}}}{ \widetilde{x}_{12} \widetilde{x}_{14} } =  \widetilde{x}_{\yd{2}}\widetilde{x}_{\yd{3}} \widetilde{x}_{\yd{1,1}} \widetilde{x}_{\yd{2,2}} {\widetilde{x}_{\yd{3,3}}}^2 \lrp{1+t_{\yd{1}} \widetilde{x}_{\yd{1}}}.
}
\begin{remark}\thlabel{val}
Note that the homogeneous degree of the extension, i.e. its $\cv$-vector, is precisely Rietsch--Williams' valuation $\val_G$ of \cite[Definition~8.1]{RW17}.
\end{remark}

On account of \eqref{p-star-iso}, the family $\Xfams{s_0}$ deforming $\XXo$ is the $\cAps{s_0}$-family for $\XXdo$, 
with $\overline{p}$ identifying the fibers of the two families.  
As mentioned in \thref{rem:AprinXfam},
the partial compactifications of these families on the other hand are different.  
Since the uncompactified families agree, we have a choice to make here.
Rietsch and Williams describe the Grassmannian in terms of a Newton--Okounkov body for the anticanonical divisor: 
the integral points of the Newton--Okounkov body give a basis of global sections for the anticanonical bundle.  
This is how we view the $\cA$-side of the picture, so the relevant partial compactification and toric degeneration here is \cite{GHKK}'s $\overline{\cAps{s_0}}$.
Indeed, as the Grassmannian is separated, this toric degeneration can not be $\Xfsps{s_0}$.
The functions and valuations however come from $\Xfams{s_0}$, with $\overline{p}$ allowing us to identify the uncompactified families.

We complete the identification of the two pictures as follows.
Fixing $q=1$ in \cite[Equation~10.1]{RW17} gives the potential $W_{\text{MR}}|_{q=1}:\XXdo\to \C$ on $\XXdo$ rather than $\XXdo \times \C^*$.
From the Gross--Hacking--Keel--Kontsevich potential $W_{\text{GHKK}}:\cX\to\C$ (see \cite[p.506]{GHKK}) we obtain $w_{\text{GHKK}}:\XXo\to \C$ given as a sum of $\tf$-functions on $\XXo$. 
It pulls back to the Marsh--Rietsch potential on $\XXdo$:
\eqn{
\overline{p}^*\lrp{w_{\mathrm{GHKK}}} = W_{\mathrm{MR}}|_{q=1}.
}
Identifying potentials identifies the polytopes defined by these potentials:
\eqn{ 
\overline{p}^*\lrp{\Gamma_{\mathrm{RW}}}
=
\Xi_{\mathrm{GHKK}}.
}
Compare \cite[Definition~10.10]{RW17} and \cite[Equation~8.7]{GHKK}.
Finally, on account of \thref{val}, $\flow=\lrp{\overline{p}^{-1}}^*$ identifies the $\gv$-vector of a $\tf$-function on $\XXdo$ with the valuation of the function-- the $\cv$-vector of $\lrp{\overline{p}^{-1}}^*\lrp{\tf}$ on $\XXo$. 
So, taking $\overline{p}: \XXdo \to \XXo$ rather than $\overline{p}^{-1}: \XXo \to \XXdo$, 
we also have 
\eqn{
\overline{p}^*\lrp{\Delta_{\mathrm{RW}}}
=
\Xi_{\mathrm{GHKK}}.
}
See \cite[Definition~8.2]{RW17}.
This explains from the \cite{GHKK} point of view the equality of the polytopes $\Gamma$ and $\Delta$ in \cite{RW17}.

\subsection{Del Pezzo surface of degree five}\label{ex:dP5}

The point of this example is to indicate a potential connection between cluster duality for cluster varieties with coefficients and Batyrev duality for Gorenstein toric Fano varieties.
Recall that Batyrev introduced a method for constructing mirror families of Calabi--Yau varieties in \cite{Bat}, taking families of anticanonical hypersurfaces in a pair of Gorenstein toric Fano varieties.  
The pair of Fanos are the toric varieties associated to a pair of (polar) dual reflexive polytopes, or equivalently to a reflexive polytope and its face-fan.
As described in the introduction, outside of the toric world these two descriptions are inequivalent, and the one relevant here is the latter.

There is a natural bounded polytope $P\subset \cX_{\Gamma^\vee}^\trop\lrp{\R} $ we can associate to a finite type $\cA$-cluster variety without frozen directions.  
First, we build a simplex in each chamber by taking the convex hull of the origin and the $\gv$-vectors generating the chamber. 
Then we take $P$ to be the union of these simplices.  
After identifying $\cX_{\Gamma^\vee}^\trop\lrp{\R}$ with $M_{\R}^\circ$ by a choice of seed, 
{\emph{if}} $P$ is convex, then it is a reflexive polytope whose unique integral interior point is the origin.\footnote{We make the identification with $M_{\R}^\circ$ so that we can talk about convexity here, as $\cX_{\Gamma^\vee}^\trop\lrp{\R}$ does not have a linear structure.}
The $\gv$-fan is the face-fan of $P$.
Suppose we compactify the $\cA$-cluster variety using $P$ and produce a toric degeneration using \cite{GHKK}'s $\cAps{s}$ construction, having central fiber $\TV{P}$.
Then $\TV{P}$ is a Gorenstein toric Fano variety, and the central fiber $\TV{\cc}$ of $\Xfsps{s}$ is the Batyrev dual Gorenstein toric Fano. 
While the general fibers of $\Xfsps{s}$ are non-separated, and {\it{a fortiori}} not Fano, we are really interested in a family of hypersurfaces in these fibers. 
These hypersurfaces may remain well-behaved away from the central fiber.
This observation will be explored and generalized in the sequel
with M.-W. Cheung.
Below we study the case of the del Pezzo surface of degree five.
This surface can be realized as a compactification of the $\cA$-cluster variety associated to the $A_2$ quiver (see \cite[Example~8.31]{GHKK}).
We review this construction. 
For more details on the degree five del Pezzo surface we refer the reader to \cite[\S8.5]{Dol12}.

The del Pezzo surface of degree five $S$ is isomorphic to a surface obtained by 
blow-up of the projective plane $\mathbb{P}^2$ at four points $q_1, q_2, q_3,q_4$ in general position. 
It is well-known
that $S$ contains ten $(-1)$-curves, four corresponding to the exceptional divisors and six corresponding to the strict transforms of lines joining pairs of points, and the incidence 
graph of these curves is given by the Petersen graph. 

\begin{figure}[htb]
\centerline{
\includegraphics[scale=1]{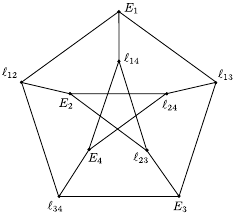}}
\caption{ \label{petersen} Petersen graph associated to the surface $S$.}
\end{figure}

Once we choose an anticanonical divisor consisting of curves forming a cycle, one can contract two non-consecutive
curves to obtain a toric model of $S$. Without loss of generality, we can assume that the anticanonical
cycle is the inner cycle in the Petersen graph. Considering Figure \ref{petersen}, we obtain the five toric models contracting the curves $\{ E_4, E_2\}$, $\{ E_2, \ell_{14} \}$, 
$\{ \ell_{14}, \ell_{24} \}$, $\{ \ell_{24}, \ell_{23} \}$ 
and $\{ \ell_{23}, E_4\}$. 
Here, $E_i$ stands for the exceptional divisor associated to the point $q_i$ and 
$\ell_{ij}$ stands for the strict transform of the line in $\mathbb{P}^2$ joining $q_i$ and $q_j$. These models are
isomorphic to the blow-up of $\mathbb{P}^1\times\mathbb{P}^1$ at one point and are related by performing
an elementary transformation at one point lying on a curve of self-intersection zero in the cycle. For example, Figure \ref{mutationdp5} shows how the toric
model associated to $\{ E_2, \ell_{14} \}$ can be obtained from the model associated to $\{ E_2, E_4\}$ by performing
an elementary transformation at a point lying on $\ell_{34}$. 

\begin{figure}[htb]
\centerline{
\includegraphics[scale=1]{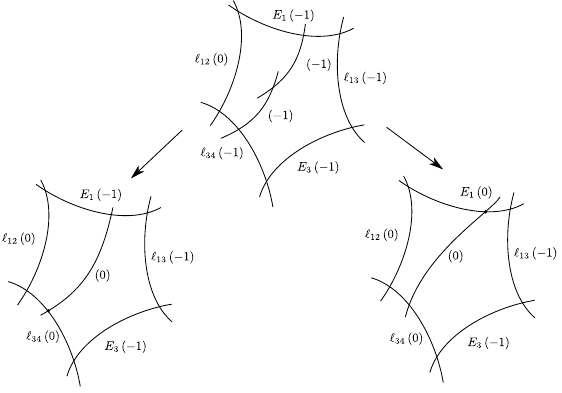}}
\caption{ \label{mutationdp5} Elementary transformation between two toric models of $S$.}
\end{figure}

Another important feature about the del Pezzo surface of degree five is that it is isomorphic to a linear section of the Grassmannian variety $\Grass_2(\C^5)$ (see \cite[Proposition 8.5.1]{Dol12}). We use this property to prove that the del Pezzo surface of degree five is a compactification of the $\cA$-cluster variety constructed from a type $A_2$ quiver. Note that we can recover the defining
equations of $\Grass_2(\C^5)$ with respect to the Pl\"ucker embedding
considering triangulations
of the 5-cycle, with Pl\"ucker coordinates associated to lines in the triangulations (see \cite[p.8]{FZ_clustersI}).
Considering all possible mutations we obtain the equations:
\eqn{
  p_{13}p_{24}&=p_{23}p_{14}+p_{12}p_{34},\\ 
  p_{14}p_{25}&=p_{15}p_{24}+p_{12}p_{45}, \\
  p_{24}p_{35}&=p_{34}p_{25}+p_{23}p_{45}, \\
  p_{13}p_{25}&=p_{12}p_{35}+p_{15}p_{23}, \\ 
  p_{14}p_{35}&=p_{45}p_{13}+p_{15}p_{34}.
}
Now consider the variety $Y\subset \mathbb{P}^9$ defined by $V(p_{12}-p_{23}, p_{23}-p_{34}, p_{34}-p_{45}, p_{45}-p_{15})$ 
and take the intersection of $\Grass_2(\C^5) \subset \PP^9$ with $Y$. 
The intersection, denoted by $T$, is defined by the following equations in $Y\cong \mathbb{P}^5$:
\eq{
  p_{13}p_{24}&=p_{12}p_{14}+p_{12}^2,\\ 
  p_{14}p_{25}&=p_{12}p_{24}+p_{12}^2,\\ 
  p_{24}p_{35}&=p_{12}p_{25}+p_{12}^2,\\
  p_{13}p_{25}&=p_{12}p_{35}+p_{12}^2,\\
  p_{14}p_{35}&=p_{12}p_{13}+p_{12}^2.} 
{eq:dp5}
One verifies computationally that $T$ is a smooth irreducible surface. To construct the canonical divisor of $T$ we do the following. We consider $D_1=V(p_{12}-p_{23})$, $D_2=V(p_{23}-p_{34})$, $D_3=V(p_{34}-p_{45})$, and $D_4=V(p_{45}-p_{15})$ 
as divisors in $\Grass_2(\C^5)$. Notice that all of them are linearly equivalent. In this way we have $T=\cap_{i=1}^4D_i$. Let $D=D_1+D_2+D_3+D_4$. 
We write the anticanonical divisor of $\Grass_2(\C^5)$ as $\sum_{ i\in \mathbb{Z}/5 \mathbb{Z}} D_{i,i+1}$,
 where $D_{i,j}=V(p_{ij})$.
 Note that the components here are also linearly equivalent to the $D_i$'s above.
Then by the adjunction formula 
\begin{align} \label{eq:-kdp5}
  K_T =(K_{\Grass_2(\C^5)}+D)|_{T} = \ssO_{\Grass_2(\C^5)}(-D_{i,j})|_T.
\end{align}
Therefore, $\omega_T\cong \ssO_T(-1)$ and then $-K_T$ is an ample divisor. Since all del Pezzo surfaces of degree five are isomorphic, we conclude that $T$ is isomorphic to $S$. So, without loss of generality we may assume that the equations of \eqref{eq:dp5} define $S$.

Now, identifying Pl\"ucker coordinates with $\cA$-cluster variables as follows
\eq{
A_1:=p_{13},\ A_2:=p_{14},\ A_3:=p_{24},\ A_4:=p_{25},\ A_5:=p_{35},
}{A-Plucker}
and taking the divisor $H=V(p_{12})$, it follows from \eqref{eq:-kdp5} that $H$ is an anticanonical divisor on $S$.
Then $U=S\backslash H$  is a log Calabi--Yau surface defined by the equations $A_{i-1}A_{i+1}=A_i+1$ for $i\in\mathbb Z/5\mathbb Z$. Note that such equations are the exchange relations for the $\cA$-cluster variety obtained from a type $A_2$ quiver.

We can describe the above, and deformations of it, using the notion of a \emph{positive polytope}, see \cite[Definiton~8.6]{GHKK}. \medskip

{\bf{Warning:  }} {\it{To be in harmony with \cite[Example~8.31]{GHKK}, we flip the orientation of the $A_2$ quiver relative to previous examples, where orientation was chosen for ease of comparison with \cite[Tables 1-4]{FZ_clustersIV}.  This change is the source of the difference in wall placement between Figures~\ref{fig:A2_family}~and~\ref{fig:A2_gfan_poly}.}}\medskip

The cones of the $\gv$-fan are domains of linearity of $\cX^\trop\lrp{\R}$.
So, we can take the convex hull of the $\gv$-vectors that generate each cone together with the origin.
We define the polytope $P$ as the union of these convex hulls, see Figure~\ref{fig:A2_gfan_poly}. \bigskip

\noindent
\begin{minipage}{\linewidth}
\captionsetup{type=figure}
\begin{center}
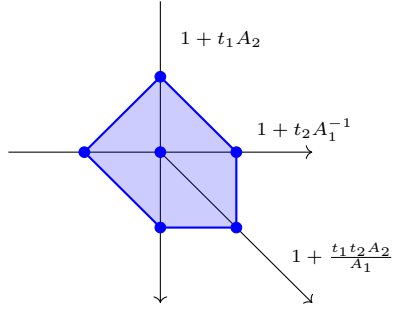

\begin{tikzpicture}

\draw (-2,0) -- (0,0);
\draw [->] (0,0) -- (2,0);
\draw [->] (0,0) -- (0,-2);
\draw (0,0) -- (0,2);
\draw [->] (0,0) -- (2,-2);

\draw[blue, thick, fill= blue, fill opacity= 0.2] (1,0) -- (0,1) -- (-1,0) -- (0,-1) -- (1,-1) -- (1,0);

\draw [blue, fill] (1,0) circle [radius=0.07];
\draw [blue, fill] (0,1) circle [radius=0.07];
\draw [blue, fill] (-1,0) circle [radius=0.07];
\draw [blue, fill] (0,-1) circle [radius=0.07];
\draw [blue, fill] (1,-1) circle [radius=0.07];
\draw [blue, fill] (0,0) circle [radius=0.07];

\node at (0.8,1.5) {\tiny $ 1 + t_1 A_2 $};
\node at (1.9,0.3) {\tiny $ 1 + t_2 A_1^{-1} $};
\node at (2.4,-1.4) {\tiny $ 1 + \frac{t_1 t_2 A_2}{ A_1} $};

\end{tikzpicture}
    \captionof{figure}{\label{fig:A2_gfan_poly}The $A_2$ scattering diagram and the polytope $P$. The integral points of $P$, shown as filled circles, correspond to $\tf$-function generators of a graded algebra.
    The scattering functions shown include principal coefficients.
    To recover the $\gv$-fan, simply forget the scattering functions and wall directions.}

\end{center}
\end{minipage}
\bigskip

To describe the family of projective varieties,
we view the integral points of $P$ and its dilations as giving a $\tf$-function basis  
for a graded $R$-algebra, where $R= \C[t_1,t_2]$.
Following \cite[Example~8.31]{GHKK}, we call the $\tf$-function associated to the interior point $\tf_0$, and the remainder $\tf_1, \dots, \tf_5$ starting with the point at $\lrp{1,0}$ and proceeding counter-clockwise.
Observe that for $i\neq 0$, $\tf_i$ is the extension of $A_i$ (from \eqref{A-Plucker}) to $\cAp$.
We take $\tf_0$ to be a homogenization variable, much like $p_{12}$ in \eqref{eq:dp5}.
The graded $R$-algebra is generated in degree 1 by these six $\tf$-functions. 
The relations between degree 1 generators are given by homogenizing the wall-crossing relations (i.e. the mutation relations in $\cAp$), giving:
\eqn{
 \tf_1 \tf_3 &= t_1 \tf_0 \tf_2 + \tf_0^2,\\
 \tf_2 \tf_4 &= t_2 \tf_0 \tf_3 + \tf_0^2,\\
 \tf_3 \tf_5 &= \tf_0 \tf_4 + t_1 \tf_0^2,\\
 \tf_4 \tf_1 &= \tf_0 \tf_5 + t_1 t_2 \tf_0^2,\\ 
 \tf_5 \tf_2 &= \tf_0 \tf_1 + t_2 \tf_0^2.
}
(Again, see \cite[Example~8.31]{GHKK}.) The central fiber in this family, i.e. when $t_1=t_2=0$, is simply the toric variety  $\TV{P}$ associated to $P$.
Observe that $P$ is the polytope 5a in \cite[\S8, Table~2]{CLS},
and $\TV{P}$ the associated Gorenstein toric Fano.

Meanwhile, the $\gv$-fan also gives us a compactification and toric degeneration of the $\cX$-variety associated to the $A_2$ quiver.
This is described in \S\ref{sec:ex_A2}, albeit for the opposite orientation of this quiver\footnote{The two choices obviously give isomorphic spaces.}.
In this case the central fiber is the toric variety $\TV{\cc}$ associated to the underlying {\emph{fan}} $\cc$ of Figure~\ref{fig:A2_gfan_poly}. Note that $\TV{\cc}$ is the del Pezzo surface of degree seven.
It is straightforward to verify that $\cc$ is the normal fan of the polytope 7a of \cite[\S8, Table~2]{CLS}.
The polytopes 5a and 7a are polar dual, and the central fibers $\TV{P}$ and $\TV{\cc}$ of the two families are Batyrev dual toric Fanos.
Our hope-- which we will explore in further work with M.-W. Cheung-- is that general fibers of the two families have the right to be called Batyrev dual as well.
We formalize this hope with the following question:

\begin{quest}
Does there exist a family $\mathcal{F}_{\vb{t}}$ of Calabi--Yau hypersurfaces in $\Xsp_{\vb{t}}$ satisfying:
\begin{enumerate}
    \item $\mathcal{F}_{\vb{t}}$ degenerates in $\Xfsps{s}$ to a family of anticanonical hypersurfaces in $\TV{\cc}$, and
    \item the Calabi--Yau hypersurfaces of $\mathcal{F}_{\vb{t}}$ are mirror to generic anticanonical hypersurfaces in the minimal model for $\cA_{\vb{t}}$ determined by $P$?
\end{enumerate}
\end{quest}

\end{document}